\documentclass[11pt]{article}
\usepackage{amsthm, amsmath, amssymb,amsbsy, epic,graphicx,mathrsfs,enumerate, bm}
\usepackage[all]{xy}
\usepackage{multicol,longtable}
\usepackage{ stmaryrd }
\usepackage{array}

\usepackage{multirow}

\usepackage[margin=3cm]{geometry}

\numberwithin{equation}{section}
\numberwithin{table}{section}

\newcolumntype{I}{!{\vrule width 1.2pt}}

\newlength\savedwidth

\usepackage{caption}
\DeclareCaptionLabelSeparator{period}{. }
\newtheorem{thm}{Theorem}[section]
\newtheorem{lem}[thm]{Lemma}
\newtheorem{cor}[thm]{Corollary}
\newtheorem{prop}[thm]{Proposition}

\newtheorem{defi}[thm]{Definition}

\theoremstyle{definition}
\newtheorem{remark}[thm]{Remark}
\newtheorem{exam}[thm]{Example}

\newcommand{\ls}[2]{{\vphantom{#2}}^{#1\!}{#2}} 	

\newcommand{\lsb}[2]{{\vphantom{#2}}^{#1\!\!}{#2}}

\newcommand{\dmod}{\mod \;}

\newcommand{\supp}{\operatorname{supp}}

\newcommand{\ra}{\rightarrow}
\newcommand{\mZ}{\mathbb{Z}}

\newcommand{\mC}{\mathbb{C}}

\newcommand{\mN}{\mathbb{N}}

\newcommand{\Stab}{\operatorname{Stab}}

\renewcommand{\rm}{\mathrm}

\newcommand{\mbf}{\mathbf}
\newcommand{\Cal}{\mathcal}

\newcommand{\N}{\operatorname{N}}

\newcommand{\lda}{\lambda}

\newcommand{\fr}{\mathfrak}

\newcommand{\scr}{\mathscr}

\newcommand{\ol}{\overline}

\newcommand{\Irr}{\operatorname{Irr}}
\renewcommand{\a}{\alpha}
\renewcommand{\b}{\beta}
\newcommand{\e}{\epsilon}
\newcommand{\g}{\gamma}
\renewcommand{\d}{\delta}
\renewcommand{\k}{\kappa}
\renewcommand{\th}{\theta}
\newcommand{\s}{\sigma}
\newcommand{\om}{\omega}
\newcommand{\Om}{\Omega}
\newcommand{\Ind}{\operatorname{Ind}}
\newcommand{\Res}{\operatorname{Res}}
\newcommand{\Syl}{\operatorname{Syl}}

\renewcommand{\fr}{\mathfrak}

\newcommand{\Inf}{\operatorname{Inf}}
\newcommand{\Bl}{\operatorname{Bl}}

\newcommand{\Proj}{\operatorname{Proj}}
\newcommand{\D}{\Delta}
\renewcommand{\O}{\mathcal O}
\newcommand{\lan}{\langle}
\newcommand{\ran}{\rangle}

\newcommand{\Supp}{\operatorname{supp}}
\newcommand{\Lda}{\Lambda}

\newcommand{\shr}{\operatorname{shr}}
\newcommand{\sgn}{\operatorname{sgn}}
\newcommand{\Ht}{\operatorname{ht}}
\newcommand{\witi}{\widetilde}
\newcommand{\cl}{\mathcal}
\newcommand{\Sym}{\mathrm{Sym}}
\newcommand{\ch}{\operatorname{ch}}
\newcommand{\PMap}{\operatorname{PMap}}
\newcommand{\GIBr}{\operatorname{GIBr}}
\newcommand{\wh}{\widehat}
\newcommand{\IBr}{\operatorname{IBr}}
\newcommand{\tp}{\operatorname{tp}}
\newcommand{\CF}{\operatorname{CF}}

\newcommand{\Tup}{\operatorname{Tup}}
\newcommand{\wre}{\mathrm{wr}}
\newcommand{\ups}{\upsilon}
\newcommand{\npa}{\medskip\noindent}
\newcommand{\Th}{\Theta}
\newcommand{\pri}{\mathrm{pri}}

\renewcommand\footnotemark{}

\title{Character correspondences for symmetric groups and wreath products}
\author{Anton Evseev
\footnote{
The author was supported by the EPSRC Postdoctoral Fellowship EP/G050244. 
} 
}

\date{} 

\begin{document}
\maketitle
 \begin{abstract} 
The Alperin--McKay conjecture relates irreducible characters of a block of an arbitrary finite group to those of its $p$-local subgroups.
A refinement of this conjecture was stated by the author in a previous paper. We prove that this refinement holds for all blocks of symmetric groups. 
Along the way we identify a ``canonical'' isometry between the principal block of $S_{pw}$ and that of $S_p\wr S_w$. We also prove a general theorem
on expressing virtual characters of wreath products in terms of certain induced characters. Much of the paper generalises character-theoretic results on
blocks of symmetric groups with abelian defect and related wreath products to the case of arbitrary defect. 
 \end{abstract}

\section{Introduction}

\subsection{A refinement of the Alperin--McKay conjecture}\label{sub:11}

Let $G$ be a finite group and $p$ be a prime. If $b$ is a $p$-block of $G$, denote by $\Irr(G,b)$ 
the set of ordinary irreducible characters of $G$ belonging to $b$ and by 
$\Irr_0 (G,b)$ the set of characters of height $0$ in $\Irr(G,b)$ 
(cf.\ Section~\ref{sec:main}). 
Let $P$ be a defect group of $b$ and $H$ be a subgroup of $G$ containing $N_G (P)$. Let the $p$-block $c$ of $H$ be the Brauer correspondent of $b$. 
The Alperin--McKay conjecture, which is one of the most important open problems in 
representation theory of finite groups, asserts that $|\Irr_0 (G,b)| = |\Irr_0 (H,c)|$ (see e.g.\ \cite{Alperin1976}). 
A property (IRC) that refines this conjecture is stated in~\cite{Evseev2010m}. This property holds in many cases but fails for some finite groups; 
however, a somewhat weaker property is conjectured in~\cite{Evseev2010m} to hold in all cases. 

In the case when $G$ is a symmetric group $S_m$ and $H=N_G (P)$, the Alperin--McKay conjecture was proved by Olsson~\cite{Olsson1976}. Moreover, Fong~\cite{Fong2003} showed that
a refinement of that conjecture due to Isaacs and Navarro~\cite{IsaacsNavarro2002} (which is implied by (IRC)) also holds in this situation. 
The main aim of this paper is to show that the property (IRC) is true in this case as well, thus strengthening the results of Olsson and Fong. 

It is convenient to work with the ``blockwise'' version (IRC-Bl) of (IRC), which we now recall (for an arbitrary $G$). 
 We write $\Cal C(G)$ for the abelian group 
$\mZ[\Irr(G)]$ of all virtual characters of $G$. A set $\cl S$ of subgroups of $P$ is said to be \emph{downwards closed} if for every $Q\in \cl S$ and $T\le Q$ we have $T\in \cl S$.

\begin{defi}\label{def:I}
 Let $\cl S$ be a downwards closed set of subgroups of $P$. Then $\Cal I(G,P,\cl S)$ is the subgroup of $\cl C(G)$ generated by all virtual characters of the form
$\Ind_L^G \phi$ where $\phi\in\Cal C(L)$ and $L\le G$ satisfies
\begin{enumerate}[(i)]
 \item $L\cap P$ is a Sylow $p$-subgroup of $L$; 
 \item $L\cap P\in \cl S$.
\end{enumerate}
\end{defi}
The map $\Proj_P\colon \Cal C(H)\to \Cal C(H)$ is defined on irreducible characters $\chi$ of $H$ by
\[
 \Proj_P (\chi) = 
\begin{cases}
 \chi & \text{if } P \text{ contains a defect group of the block containing } \chi, \\
 0 & \text{otherwise,}
\end{cases}
\]
 and is extended linearly to the whole of $\Cal C(H)$. If $P$ is a Sylow $p$-subgroup of $H$ then $\Proj_P$ is simply the identity map. 

Define 
$\cl S(G,P,H)$ to be the set of all subgroups of $P$ which are contained in a subgroup of the form $\ls{g}P\cap P$ for some $g\in G-H$.
Let $\cl S=\cl S(G,P,H)$, $\cl I_G = \cl I(G,P,\cl S)$ and $\cl I_H=\cl I(H,P,\cl S)$. 
By Theorem 2.6 and Proposition 2.11 of~\cite{Evseev2010m}, there is a canonical isomorphism 
\[
 \frac{\mZ[\Irr(G,b)]}{\mZ[\Irr(G,b)]\cap \cl I_G} \simeq \frac{\mZ[\Irr(H,c)]}{\mZ[\Irr(H,c)] \cap \cl I_H}.
\]
This isomorphism is given by the map $\Proj_P \Res^G_H$, while its inverse is induced by the map $\Ind_H^G$ composed with the standard identification 
between $(\mZ[\Irr(G,b)] + \cl I_G)/ \cl I_G$ and $\mZ[\Irr(G,b)]/(\mZ[\Irr(G,b)]\cap \cl I_G)$. 
We may then ask whether a bijection between $\Irr_0 (G,b)$ and $\Irr_0 (H,c)$ can be 
chosen to be compatible with this natural isomorphism. 
If $S_1$ and $S_2$ are subsets of some abelian groups, we say that a map $F\colon \pm S_1 \to \pm S_2$ is \emph{signed} if 
$F(-\chi) = -F(\chi)$ for all $\chi\in S_1$. 

\begin{defi}[\protect{\cite[Section 2]{Evseev2010m}}]\label{def:ircbl}
 Let $G$, $P$, $H$, $b$, $c$ be as above and $\cl S= \cl S(G,P,H)$. We say that the quadruple $(G,b,P,H)$ satisfies the property (IRC-Bl) if there exists a
signed bijection $F\colon \pm\Irr_{0}(G,b) \to \pm\Irr_{0}(H,c)$ such that
\[
F(\chi) \equiv \Proj_P \Res_H^G \chi \dmod  \Cal I(H,P,\cl S) \quad \text{ for all } \chi\in \pm \Irr_{0}(G,b).
\]
\end{defi}

\begin{thm}\label{thm:symirc}
 Let $b$ be a $p$-block  of a symmetric group $S_m$ and $P$ be a defect group of $b$. 
Then (IRC-Bl) holds for the quadruple $(S_m,b, P, N_{S_m} (P))$. 
\end{thm}

Most of the paper is devoted to a proof of Theorem~\ref{thm:symirc}.

\subsection{A simpler statement}\label{sub:12}


An important part of the proof may be represented by the following less technical result, which may be of independent interest. 
If $h$ is an element of a finite group, let $h_p$ be the $p$-part of $h$.
Let $w$ be any positive integer. 
Consider the wreath product $S_p\wr S_w$ as a subgroup of $S_{pw}$ in the obvious way. Define 
\[
 \Cal W = \{ h\in S_p\wr S_w \mid C_{S_{pw}} (h_p) \le S_p\wr S_w \}.
\]

\begin{thm}\label{thm:val}
Let $b_0$ and $c_0$ be the principal $p$-blocks of $S_{pw}$ and $S_p\wr S_w$ respectively. Then 
 there exists a unique signed bijection $F\colon \pm\Irr(S_{pw}, b_0) \to \pm\Irr(S_p \wr S_w, c_0)$ such that $F(\chi)(h)=\chi(h)$ for all $h\in \cl W$.
Moreover, $F$ preserves heights of characters. 
\end{thm}
In fact, $F$ is one of a family of signed bijections constructed by Rouquier~\cite[{\S}2.3]{Rouquier1994}.

Returning to the notation of Theorem~\ref{thm:symirc}, let $w$ be the weight of the block $b$ (cf.\ \S\ref{sub:22}) and $e=m-pw$.
Then $P$ may be taken to be a Sylow $p$-subgroup of $S_p\wr S_w\le S_{pw}$, and
$H=(S_p\wr S_w)\times S_e$ may be viewed as a subgroup of $S_{pw}\times S_e$ and hence of $S_m$. 
Let $c_0$ denote the principal $p$-block of $S_p\wr S_w$. 

One of the most important steps in the proof of Theorem~\ref{thm:symirc} is to show that (IRC-Bl) holds for the quadruple $(S_m, b, H, P)$.
In the case when $e=0$, this turns out to be a consequence of Theorem~\ref{thm:val} (together with some other results of Sections~\ref{sec:main}--\ref{sec:I}). 
If $e>0$, the existence part of an analogue of Theorem~\ref{thm:val} still holds (cf.\ Theorem~\ref{thm:main}) and still leads to a proof of (IRC-Bl) 
for $(S_m,b,H,P)$, but the uniqueness part is no longer true. 

\begin{remark} Fix a $p$-element $x$ of $S_p\wr S_w$ such that $C = C_{S_{pw}} (x) \le S_p\wr S_w$. 
Then the condition in Theorem~\ref{thm:val} that $F(\chi)(h)=h$ for all $h\in S_p\wr S_w$ such that $h_p = x$ 
may be expressed by the following commutative diagram:
\[
 \xymatrix{
  \pm \Irr(S_{pw}, b_0) \ar[d]_{d^x} \ar[r]^-F & \pm \Irr(S_p\wr S_w, c_0) \ar[d]_{d^x} \\
 \CF(C) \ar@{=}[r] & \CF(C)
 }
\]
Here $\CF(C)$ is the set of $\mC$-valued class functions on $C$, and $d^x$ is the generalised decomposition map defined by
\[
 d^x(\chi) (z) = 
\begin{cases}
  \chi(xz) & \text{if } z \text{ is } p\text{-regular}, \\
 0 & \text{otherwise.}
\end{cases}
\]
This is an instance of the isotypy condition due to Brou{\'e} (see e.g.\ \cite[D{\'e}finition 4.6]{Broue1990}), which is usually considered
in cases where $P$ is abelian. Here we have added an extra requirement that the bottom map of the diagram be the identity one, and as a result only one of $p!$ isotypies
constructed by Rouquier~\cite{Rouquier1994} (when $w<p$) satisfies the requirements of Theorem~\ref{thm:val}.
\end{remark}

\begin{remark}
 For each $h\in \Cal W$, the second orthogonality relations for characters of $S_{pw}$ and $S_p\wr S_w$ yield
\begin{equation}\label{eq:secorth}
 \sum_{\chi\in \Irr(S_{pw})} |\chi(h)|^2 = |C_{S_{pw}} (h)| = |C_{S_p\wr S_w}(h)| = \sum_{\phi\in\Irr(S_p\wr S_w)} |\phi(h)|^2.
\end{equation}
Thus, Theorem~\ref{thm:val} says that, at least for characters lying in the principal blocks of $S_{pw}$ and $S_p\wr S_w$, we have not only an equality between the sums on the two sides
of~\eqref{eq:secorth}, but also a bijective correspondence between the summands such that the corresponding summands are equal. 
Moreover, this correspondence works simultaneously for all $h\in \cl W$.
\end{remark}

Due to all these observations, it seems reasonable to suggest that the map $F$ of Theorem~\ref{thm:val} yields a \emph{natural} isometry between $\mZ[\Irr(S_{pw}, b_0)]$ and 
$\mZ[\Irr(S_p\wr S_w, c_0)]$. The uniqueness part of that theorem is noteworthy: when $G$ is an arbitrary finite group, it is usually impossible to distinguish a particular signed 
bijection between $\pm \Irr(G, b)$ and $\pm \Irr(H,c)$ among other such bijections. Thus, it seems worthwhile to investigate the isometry $F$ in more detail: 
see \S\ref{sub:13} and Section~\ref{sec:prop}.

\subsection{Outline of the paper}\label{sub:13}

In Section~\ref{sec:prelim} we fix notation and establish some standard results or easy consequences thereof, in particular, concerning class functions on wreath products. 
Sections~\ref{sec:main}--\ref{sec:indarg} are devoted to a proof of Theorem~\ref{thm:symirc}. 
In Section~\ref{sec:main} we prove the existence part of Theorem~\ref{thm:val}. Indeed, Theorem~\ref{thm:main} gives a more precise and general statement. 
(Uniqueness in Theorem~\ref{thm:val} is proved in Section~\ref{sec:un}.)

In \S\ref{sub:41} we prove a general result on expressing virtual characters of wreath products in terms of certain induced characters (Theorem~\ref{thm:wrind}), which in some sense is a generalisation of a 
case of Brauer's induction theorem (see Example~\ref{exam:Br}). This result may be of independent interest. Using it and the results of Section~\ref{sec:main}, 
we show in \S\ref{sub:42} that the quadruple $(S_m, b, P, (S_p\wr S_w)\times S_e)$ (cf.\ Theorem~\ref{thm:symirc}) satisfies (IRC-Bl).

In Section~\ref{sec:indarg} we construct an appropriate 
signed bijection between $\pm\Irr_0 (S_p\wr S_w, c_0)$ and $\pm \Irr_0 (\N_{S_{pw}} (P))$. Composing it with a signed bijection
from $\pm \Irr_0(S_m,b)$ to $\pm \Irr_0(S_p\wr S_w, c_0)$ obtained in Section~\ref{sec:main} allows us to complete the proof of Theorem~\ref{thm:symirc}. Along the way, in 
\S\ref{sub:51}, we refine a general result due to Marcus~\cite{Marcus1996} on construction of double complexes that give derived equivalences between blocks of wreath products. 
We remark that many of the results of Sections~\ref{sec:I}--\ref{sec:indarg} are easy or known when $w<p$, but the general case appears to require much more technical arguments.

If $w<p$, then $P$ is abelian and the bijection $F$ of Theorem~\ref{thm:val} induces a perfect isometry in the sense of Brou{\'e}~\cite{Broue1990}, due to~\cite{Rouquier1994}. 
It is not known how one should generalise the perfectness property to blocks with non-abelian defect groups. Due to the uniqueness of the isometry $F$ in Theorem~\ref{thm:val}, 
it appears to be plausible that $F$ should satisfy 
whatever conditions one may wish to impose to generalise perfectness. This motivates a more detailed study of the properties of $F$, which we perform in Section~\ref{sec:prop}. 
In \S\ref{sub:71} we develop a modified version of modular character theory for groups of the form $L\wr S_w$
when $w\ge p$: roughly speaking, instead of looking at the set of $p$-regular elements, as one does in the theory of Brauer characters, we consider a somewhat larger subset of $L\wr S_w$.
After specialising to $L=S_p$, this leads to a modification of some of the conditions required for an isometry to be perfect. In \S\ref{sub:72} 
we show that $F$ satisfies these adjusted conditions for all $w$.

\begin{remark}
  A generalisation of the perfectness property to the non-abelian case has been proposed by
Narasaki and Uno~\cite{NarasakiUno2009}. However, it 
does not appear to be directly relevant to the situation of Theorem~\ref{thm:val} because
the condition on control of fusion 
(see~\cite[Definition 25]{NarasakiUno2009}) is not satisfied in that case. Indeed, specialising to $w=p=2$ in Theorem~\ref{thm:val}, 
one can check that there is no isometry between the principal blocks of $S_4$ and $D_8$ satisfying condition (RP1) of~\cite[Definition 20]{NarasakiUno2009} (with $Q=[P,P]$). 
\end{remark}

\bigskip
\noindent
\textbf{Acknowledgments.} The author is grateful to Joseph Chuang for several helpful discussions. He would also like to thank the referee for suggesting a considerable simplification of the proof of Theorem~\ref{thm:wrind}.

\section{Notation and preliminary results}\label{sec:prelim}

\subsection{General notation and conventions}\label{subsec:gennot}

We fix some notation, which is mostly standard. 
Throughout this subsection, we assume that $p$ is a fixed prime, $G$ is a finite group, and $R$ an arbitrary unital commutative ring.

\npa
\textbf{Integers.} We write $\mN$ for the set of positive integers and $\mZ_{\ge 0}$ for that of nonnegative integers. 
If $i,j\in \mZ$, we write $[i,j] = \{ l\in \mZ \mid i\le l\le j \}$. The $p$-adic valuation of an integer $n$ is denoted by $v_p (n)$.

\npa
\textbf{Groups.} If $g\in G$, we denote by $g_p$ and $g_{p'}$ the $p$-part and the $p'$-part of $g$ respectively. We write $G_{p'} = \{ g\in G \mid g_{p'} =g\}$. 
 The trivial group will be denoted by
$\mbf 1$. The set of Sylow $p$-subgroups of $G$ is $\Syl_p (G)$. If $G$ acts on a set $X$, we will denote the stabiliser in $G$ of an element $x\in X$ by $G_x$ or $\Stab_G (x)$.
By $\cl A(G)$ we denote the set of all subgroups of $G$.  

\npa
\textbf{Rings.} We assume that $(K,\cl O,k)$ is a fixed $p$-modular system with the usual properties. 
That is, $\cl O$ is a discrete valuation ring with field of fractions $K$ which has characteristic $0$ and is ``large enough'', i.e.\ is a splitting field for 
all finite groups in question; and the quotient $k$ of $\cl O$ by its maximal ideal is an algebraically closed field of characteristic $p$.

\npa
\textbf{Class functions.}
In general, $\CF(G;R)$ is the ring of $R$-valued functions on $G$ that are constant on conjugacy classes.
However, unless otherwise specified, we assume class functions to be $K$-valued. In particular, $\Irr(G)$ denotes the set of irreducible characters of $G$ with values in $K$. 
The trivial character of $G$ is denoted by $1_G$. 
We write $\cl C(G)=\mZ[\Irr(G)]$. The set of $\chi\in\Irr(G)$ such that $p$ does not divide $\chi(1)$ is denoted by $\Irr_{p'}(G)$. 
The inner product of two class functions $\xi$ and $\th$ is denoted by
$\lan \xi, \th\ran$. Whenever the context does not specify a particular prime $p$, this prime may be chosen arbitrarily: in such situations (e.g.\ in Section~\ref{sec:main}), one may as well assume that $K$ is replaced by $\mC$. 
For $\chi\in \cl C(G)$, we define $\ol\chi \in \cl C(G)$ by $\ol\chi(g) = \chi(g^{-1})$ for $g\in G$. 

Suppose that $N$ is a normal subgroup of $G$ and $\phi\in \Irr(N)$. We denote by $\Irr(G|\phi)$ the set of $\chi\in \Irr(G)$ such that $\phi$ is a summand of $\Res^G_N \chi$ and write
$\cl C(G|\phi)= \mZ[\Irr(G|\phi)]$. 
If $\xi\in \CF(G/N; K)$, then $\Inf_{G/N}^G \xi$ is the inflation of $\xi$ to $G$. It is defined by $(\Inf_{G/N}^G \xi) (g) = \xi(gN)$, $g\in G$. 

\npa
\textbf{Blocks.} If $f\in Z(K G)$ is an idempotent, we define the projection map $\Proj_f \colon \CF(G; K) \to \CF(G; K)$ by 
$\Proj_f (\xi) (g) = \xi(gf)$ for $g\in G$ and $\xi\in \CF(G;K)$. The image of $\Proj_f$ is denoted by $\CF(G, f; K)$ (as in~\cite{Broue1990}).

By a \emph{$p$-block} we mean a primitive idempotent $b$ of $Z(\cl O G)$. As usual, we write $\Irr(G,b) = \Irr(G) \cap \CF(G,b; K)$: this is the set of irreducible characters
belonging to $b$. Further, $\cl C(G, b) = \mZ[\Irr(G,b)]$. The set of all blocks of $G$ will be denoted by $\Bl(G)$. 

\npa
\textbf{Modules.} Modules are assumed to be left ones unless we specify otherwise. 
All $\cl OG$-modules are assumed to be $\cl O$-free of finite rank. 

\npa
\textbf{Products.} If $\xi\in \CF(G; K)$ and $\th\in \CF(H;K)$, we write $\xi\times \th\in \CF(G\times H; K)$ for the Kronecker product of $\xi$ and $\th$, defined by 
$(\xi\times \th)(g,h) = \xi(g)\th(h)$. We will write $\prod_{i=1}^r a_i = a_1 \times \cdots \times a_r$ and $a^{\times r} = \prod_{i=1}^r a = a\times \cdots \times a$ for any objects
$a,a_1,\ldots,a_r$ for which the products on the right-hand sides make sense. Similarly, 
$a^{\otimes r} = a\otimes \cdots \otimes a$ ($r$ factors) whenever the last
tensor product is defined. 

If $w_1,\ldots,w_r \in \mZ_{\ge 0}$ and $w\ge w_1+\cdots +w_r$, then 
$\prod_{i=1}^r S_{w_i}$ will be viewed as a Young subgroup of $S_w$, so that the factor $S_{w_i}$ is the subgroup of $S_w$ consisting of all the elements that fix all points in $[1,w]-I_i$
where $I_i$ is a subset of $[1,w]$ of size $w_i$ and $I_1,\ldots, I_r$ are disjoint. 

\npa
\textbf{Partitions.} 
A \emph{partition} is a finite non-increasing sequence of positive integers. The set of all partitions will be denoted by $\cl P$ and the empty partition by $\varnothing$. Suppose that $\lda =(\lda_1,\ldots,\lda_r)$ and $\mu=(\mu_1,\ldots,\mu_s)$ are partitions. We will write $|\lda| = \lda_1+\ldots+\lda_r$ (the \emph{size} of $\lda$) and $l(\lda) = r$ (the \emph{length} of $\lda$). The conjugate partition of $\lda$ is denoted by $\lda'$.
For $i\in \mN$, we write $m_{i}(\lda) = \{ s\in [1,l(\lda)] \mid \lda_s=i\}$.
By $\lda \sqcup \mu$ we mean the partition obtained by ordering the sequence $(\lda_1,\ldots,\lda_r,\mu_1,\ldots,\mu_s)$ in the non-increasing order.
The set of all partitions of size $n$ will be denoted by $\Cal P(n)$.
If $m\in \mN$, we define $m\lda = (m\lda_1,\ldots,m\lda_r)$. We write $(m^d)$ for the partition $(m,\ldots, m)$ ($d$ parts). 

%
We say that $\lda$ \emph{contains} $\mu$ and write $\lda\supset \mu$ if $s\le r$ and $\mu_i\le \lda_i$ for all $i\in [1,s]$. In this case, $\lda$ and $\mu$ define a \emph{skew partition}, denoted by $\lda/\mu$, and we write $|\lda/\mu| = |\lda|-|\mu|$. 
Each partition $\lda$ will be identified with the skew partition $\lda/\varnothing$. 


\subsection{Characters of symmetric groups and abacuses}\label{sub:22}

We introduce notation and recall certain well-known facts on characters of symmetric groups and related combinatorics. 
The details may be found e.g.\ in~\cite{JamesKerber1981}.

To each skew partition $\lda/\mu$ of size $m$ one attaches a character $\chi^{\lda/\mu}$ of $S_m$ in a standard way (see~\cite[\S 2.3]{JamesKerber1981}).
If $\mu=\varnothing$ (the empty partition), this character is irreducible, and in fact $\Irr(S_m) = \{ \chi^{\lda} \mid \lda\in \Cal P(m) \}$. 

Let $p\in \mN$ and let $N$ be a fixed integer divisible by $p$ such that $N-p$ is greater than the number of parts in every partition considered. Let $\lda\in \cl P$.
We will denote by $\b(\lda, p)$ the $p$-abacus with $N$ beads representing $\lda$: see~\cite[\S 2.7]{JamesKerber1981}. 
By convention, the runners are numbered $0,1,\ldots, p-1$ from left to right, and slots on each runner by $0,1,2,\ldots$ from the top down.
We say that the slot in column $i$ and row $r$ \emph{represents} the number $rp+i$. Note that, by the choice of $N$, each runner of $\b(\lda,p)$ contains at least one bead. 
If $\lda, \mu\in \Cal P$, we will write $\lda\supset_p \mu$ if $\b(\mu,p)$ can be obtained from
$\b(\lda,p)$ by moving beads one step up (to a free space) several times (or, equivalently, if $\mu$ can be obtained from $\lda$ by removing several rim $p$-hooks). 

A partition $\rho$ is a \emph{$p$-core} if there is no partition $\kappa$ such that $\rho/\kappa$ is a rim $p$-hook. Equivalently, $\rho$ is a $p$-core if and only if 
within each runner of $\b(\rho,p)$
there are no gaps between beads. The \emph{$p$-core of a partition} $\lda$ is the partition $\rho$ such that $\b (\rho,p)$ is obtained from $\b(\lda,p)$ by moving all the beads up along each runner as far as possible. 
By $(\lda(0),\ldots,\lda(p-1))$ we denote the 
\emph{$p$-quotient} of $\lda$, which is defined as follows: $\lda(i)$ is the unique partition
whose $1$-abacus is column $i$ of $\b(\lda, p)$ (up to adding a number of beads to the top of the column).

\begin{thm}[See~\protect{\cite[Theorem 2.7.30]{JamesKerber1981}}]\label{thm:corequot}
 Let $\rho$ be a $p$-core. Then $\lda\mapsto (\lda(0),\ldots,\lda(p-1))$ is a bijection from the set of partitions with $p$-core $\rho$ onto
$\cl P^{\times p}$.
\end{thm}

By a famous result (known as Nakayama's Conjecture), if $p$ is a prime, then the $p$-blocks of a symmetric group $S_m$ are in a bijective correspondence with the $p$-core partitions $\rho$ such that $\rho$ is the $p$-core of at least one element of $\cl P(m)$. The \emph{weight} of such a block is the integer $w$ such that $m = |\rho| + pw$. 

Suppose that $\lda$ and $\mu$ are partitions such that $\lda\supset_p \mu$. 
In the \emph{natural numbering} of $\lda$, the beads of $\b(\lda, p)$ are marked with numbers $1,\ldots, N$ in the increasing order 
of numbers represented by the slots occupied by those beads. If we start with this numbering and move the beads of $\b(\lda,p)$ along the runners 
one step up at a time  (keeping the numbers) in such a way that we eventually get $\b(\mu,p)$, then we obtain
a numbering of $\b(\mu,p)$. The \emph{$p$-sign} of $\lda/\mu$, denoted by $\e_p (\lda/\mu)$, is defined as the sign of the permutation mapping this numbering of $\b(\mu,p)$ to the 
natural numbering of $\b(\mu,p)$ (see~\cite[2.7.18--2.7.26]{JamesKerber1981}).

\subsection{Wreath products}\label{sub:wr}

Let $L$ be a fixed finite group and $w\in \mN$. Following~\cite[\S 4.1]{JamesKerber1981}, we represent elements of the wreath product $L\wr S_w$ in the form
$(x_1,\ldots,x_w; \sigma)$ ($x_i\in L$, $\sigma\in S_w$) where multiplication is given by
\[
 (x_1,\ldots, x_w; \sigma) (y_1,\ldots,y_w; \tau)=(x_1 y_{\s^{-1}(1)}, \ldots, x_w y_{\s^{-1}(w)} ; \s\tau).
\]
By convention, both $S_0$ and $L\wr S_0$ will be identified with the trivial group. 
If $A\le L$ and $B\le S_w$, then $A\wr B$ is viewed as a subgroup of $L\wr S_w$ in the obvious way; 
in particular, $B$ becomes a subgroup of $L\wr S_w$ after it is identified with $\mbf 1 \wr B$.

Let us view $S_w$ as the group of all permutations of the set $[1,w]$. 
We define a \emph{marked cycle} in $S_w$ as either a non-identity cyclic permutation $\sigma \in S_w$ or an element of $[1,w]$. 
The \emph{product} $\sigma_1 \cdots \sigma_r \in S_w$ 
of several marked cycles is defined in the usual way after all multiples $\s_j$ which are elements of $[1,w]$ are replaced with the identity element of $S_w$. 
We define the \emph{support} $\Supp(\s)$ of a marked cycle $\s$ as follows: if $\s$ is a non-identity cycle, then $\Supp(\s)$ is the set of points in $[1,w]$ moved by $\s$; if $\s=j\in [1,w]$, then
$\Supp(\s)=\{ j \}$. The \emph{order} of a marked cycle $\s$ is defined by $o(\s) =|\Supp(\s)|$.
Less formally, a marked cycle is either a non-identity cycle or the identity element with an assigned singleton support set. 

The above definition ensures that every element $\s\in S_w$ decomposes as a product $\s=\s_1\cdots \s_r$ where $\s_1,\ldots,\s_r$ are marked cycles in $S_w$ and
$[1,w] = \sqcup_{i} \Supp(\s_i)$. 
Moreover, this decomposition is unique up to permutation of factors. 
The tuple $(o(\s_1),\ldots, o(\s_r))$ (in an arbitrary order) is said to be the \emph{cycle structure} of $\s$.

Let $\sigma$ be a marked cycle in $S_w$ and $i$ be the smallest element of $\Supp(\sigma)$. For $x\in L$ we define
\[
 y_{\sigma} (x) = (1,\ldots,1,x,1,\ldots,1; \sigma) \in L\wr S_w
\]
where $x$ occurs in the $i$-th position. The following lemma is standard and easy to prove.

\begin{lem}\label{lem:cyctype}
 Suppose that $L=S_n$ for some $n\in \mN$. 
 Let $\sigma\in S_w$ decompose into a product of marked cycles as $\sigma=\sigma_1\cdots \sigma_r$, with $[1,w]=\sqcup_i \supp(\sigma_i)$.
If $x\in S_n$ is an $n$-cycle, then the cycle structure of $y_{\s_1} (x) \cdots y_{\s_r} (x)$, viewed as an element of $S_{nw}$ via the natural inclusion $S_n \wr S_w \le S_{nw}$, is 
$(no(\s_1),\ldots,n o(\s_l))$.
\end{lem}

We describe the conjugacy classes of $L\wr S_w$. The proofs may be found e.g.\ in~\cite[\S 4.2]{JamesKerber1981}. 
Let $\s,\s'\in S_w$ and consider  decompositions $\sigma=\s_1\cdots \s_r$ and
$\s'=\s'_1\cdots \s'_l$
 into marked cycles with orders summing to $w$ in each case.
Then, for any $x_1,\ldots,x_w\in L$, there exist $z_1,\ldots,z_r\in L$ such that  $(x_1,\ldots,x_w; \s)$ is $(L\wr S_w)$-conjugate to $y_{\s_1} (z_1) \cdots y_{\s_r}(z_r)$.
Moreover, two elements $y_{\s_1}(z_1) \cdots y_{\s_r} (z_r)$ and $y_{\s'_1}(z'_1)\cdots y_{\s'_l}(z'_l)$ are $(L\wr S_w)$-conjugate if and only if $l=r$ and 
there exists a permutation $\tau$ of $[1,r]$ such that $o(\s_{j})=o(\s'_{\tau j})$ and $z_j$ is $L$-conjugate to $z'_{\tau j}$ for every $j\in [1,r]$.


We now construct some class functions on $L\wr S_w$ and, in particular, describe $\Irr(L\wr S_w)$.
If $R$ is any unital commutative ring and $M$ is an $RL$-module, then $M^{\otimes w}$ has an $R(L\wr S_w)$-module structure given by the action
\[
(x_1,\ldots,x_w; \s) (v_1\otimes \cdots \otimes v_w) = x_1v_{\s^{-1} 1} \otimes \cdots \otimes x_w v_{\s^{-1} w}
\]
(see~\cite{JamesKerber1981}, Eq.\ 4.3.7). This  $R(L\wr S_w)$-module will be denoted by $M^{\witi{\otimes} w}$.
Let $\phi$ be a character of $L$, and let $M$ be a $KL$-module affording it.
We write $\phi^{\witi{\times} w}$ for the character of $L\wr S_w$ afforded by $M^{\widetilde{\otimes} w}$. 
The values of this character are given by the following lemma.

\begin{lem}[\protect{\cite[Lemma 4.3.9]{JamesKerber1981}}]\label{lem:wrcharval}
 Let $\s_1, \ldots, \s_r$ be disjoint marked cycles in $S_w$ with orders summing to $w$. 
For any $x_1,\ldots,x_r\in L$ we have
\begin{equation}\label{eq:wrcharval}
 \phi^{\witi{\times} w} (y_{\s_1}(x_1) \cdots y_{\s_r} (x_r)) 
 = \phi(x_1) \cdots \phi(x_r).
\end{equation}
\end{lem}

For an arbitrary class function $\phi$ on $L$, the formula~\eqref{eq:wrcharval} defines a class function on $L\wr S_w$, 
which will also be denoted by $\phi^{\witi{\times} w}$. 
Let $\Tup_w (L)$ be the set of tuples of the form $\Theta=((\phi_1,\chi_1),\ldots,(\phi_s,\chi_s))$ such that 
\begin{enumerate}[(i)]
 \item $\phi_i\in \CF(L; K)$ for each $i$;
 \item $\chi_i\in \cl C(S_{w_i})$ for each $i$, where $w_1,\ldots, w_s\in \mZ_{\ge 0}$ and $w_1+\cdots +w_s =w$.
\end{enumerate}
For such a tuple $\Theta\in \Tup_w (L)$ define a class function $\zeta_{\Theta}$ on $L\wr S_w$ by 
\begin{equation}\label{eq:defzeta}
 \zeta_{\Theta} = \Ind_{\prod_{i} (L\wr S_{w_i})}^{L\wr S_w} \prod_{i=1}^s (\phi_i^{\witi{\times} w_i} \Inf_{S_{w_i}}^{L\wr S_{w_i}} \chi_i).
\end{equation}

Suppose that $\Phi\colon \Irr(L)\to \cl{P}$ is a map satisfying 
$\sum_{\phi\in\Irr(L)} |\Phi(\phi)|= w$. Define the character $\zeta_{\Phi}$ of $L\wr S_w$ by setting $\zeta_{\Phi} = \zeta_{\Theta}$ where
\[
 \Theta = ( (\phi, \chi^{\Phi(\phi)} ) \mid \phi\in \Irr(L) ).
\]
For any finite set $X$ we define $\PMap_w (X)$ to be the set of maps 
$\Phi\colon X\to \cl P$ such that $\sum_{a\in X} |\Phi(a)|=w$.

\begin{thm}[\protect{\cite[Theorem 4.3.34]{JamesKerber1981}}]\label{thm:wrirr}
The map $\Phi\mapsto \zeta_{\Phi}$ is a bijection from $\PMap_w (\Irr(L))$ onto $\Irr(L\wr S_w)$. 
\end{thm}



We will need some further results concerning class functions on wreath products. 
 If $\lda$ is a partition, let $g_{\lda}$ denote an (arbitrary) element of $S_{|\lda|}$ of cycle type $\lda$.
For $n\in \mZ_{\ge 0}$ and $a\in \mZ$, let $\k_{n,a}\in \CF(S_n; K)$ be the class function defined by $\k_{n,a} (g_{\lda}) = a^{l(\lda)}$ for all $\lda\in \cl P(n)$. 

\begin{lem}\label{lem:lincomb}
 Let $\phi_1,\phi_2\in \CF(L;K)$ and $a_1,a_2\in K$. Then 
\begin{align}
 (a_1\phi_1 + a_2\phi_2)^{\witi{\times} w} & \notag \\
 &\!\!\!\!\!\!\!\!\!\!\!\!\!\!\!=
 \sum_{j=0}^{w} \Ind_{(L\wr S_j) \times (L\wr S_{w-j})}^{L\wr S_w} 
\left(\phi_1^{\witi{\times} j}\Inf_{S_j}^{L\wr S_j} \k_{j,a_1} 
\times \phi_2^{\witi{\times} (w-j)} \Inf_{S_{w-j}}^{L\wr S_{w-j}} \k_{w-j,a_2}\right).
\label{eq:lincomb1}
\end{align}
\end{lem}

\begin{proof}
  Consider an element 
$g=y_{\s_1}(x_1) \cdots y_{\s_r}(x_r)\in L\wr S_w$, where $\s_1,\ldots,\s_r$ are disjoint marked cycles in $S_w$ with orders summing to $w$ and $x_1,\ldots,x_r\in L$.  Let $\Cal F$ be the set of all maps $f\colon [1,r] \to \{1,2\}$. 
For each $f\in\Cal F$ let $j(f)=\sum_{i\in f^{-1}(1)} o(\s_i)$. Then for $j\in [0,w]$ the left $(S_j\times S_{w-j})$-cosets $h(S_j \times S_{w-j})$ in $S_w$ 
satisfying $\ls{h^{-1}}g \in (L\wr S_j)\times (L\wr S_{w-j})$
are parameterised by the set $\{ f\in \cl F \mid j(f)=j\}$. 
By Eq.\ \eqref{eq:wrcharval} and
the definition of induced class function, the value of the right-hand side of~\eqref{eq:lincomb1} on $g$ is 
\[
\begin{split}
\sum_{f\in \cl F} \left(
\prod_{i\in f^{-1}(1)} (a_1 \phi_1 (x_i)) \prod_{i\in f^{-1}(2)} (a_2 \phi_2 (x_i)) \right)
= \prod_{i=1}^r (a_1\phi_1 (x_i) + a_2\phi_2(x_i)) 
= (a_1\phi_1 +a_2\phi_2)^{\witi{\times} w}(g).
\end{split}
\]
The result follows. 
\end{proof}

\begin{lem}\label{lem:xina}
For any $n\ge 0$ and $a\in \mZ$, we have $\k_{n,a} \in \cl C(S_n)$.  
\end{lem}

\begin{proof}
For $a\ge 0$, we obtain the result by applying Lemma~\ref{lem:lincomb} with $L=\mbf 1$, $a_1=a$, $a_2=0$, $\phi_1 = 1_{\mbf 1}$ and $w=n$. 
 For $a<0$, the result follows from the identity $\k_{n,a} = (-1)^n \k_{n,-a}\sgn$, where $\sgn$ is the sign character of $S_n$. 
\end{proof}

\begin{cor}\label{cor:wrextgen}
For every $\phi\in \Cal C(L)$ we have $\phi^{\witi{\times} w} \in \Cal C(L\wr S_w)$. 
\end{cor}

\begin{proof}
Since $\phi$ can be written as a difference of two characters of $L$, the result follows from
Lemmas~\ref{lem:lincomb} and~\ref{lem:xina}. 
\end{proof}

\begin{cor}\label{cor:otherbasis}
Suppose that for each $n\in [0,w]$ we have 
$\mZ$-bases $\{ \a_{\lda} \mid \lda\in \cl P(n)\}$
and $\{ \a'_{\lda} \mid \lda\in \cl P(n)\}$ of $\cl C(S_n)$. 
\begin{enumerate}[(i)]
\item\label{ob1}
Let $B$ and $B'$ be subsets of $\cl C(L)$ with equal $\mZ$-spans. 
Then the $\mZ$-spans of the sets 
$\{ \zeta_{((\phi, \a_{\Phi(\phi)} ) \mid \phi \in B )} \mid \Phi\in \PMap_w (B)\}$ and 
$\{ \zeta_{((\phi', \a'_{\Phi'(\phi')}) \mid \phi' \in B')} \mid \Phi'\in \PMap_w (B')\}$
are equal.
\item\label{ob2}
If $B$ is a $\mZ$-basis of $\cl C(L)$, then 
$\{ \zeta_{((\phi, \a_\Phi(\phi)) \mid \phi \in B )} \mid \Phi\in \PMap_w (B)\}$
is a $\mZ$-basis of $\cl C(L\wr S_w)$. 
\end{enumerate}
\end{cor}

\begin{proof}
\eqref{ob1}
It follows from the hypothesis and Lemmas~\ref{lem:lincomb} and~\ref{lem:xina} that for every $n\in \mZ_{\ge 0}$, $\phi'\in B'$ and $\lda\in \cl P(n)$, the character 
$(\phi')^{\witi\times n} \Inf_{S_n}^{L\wr S_n}\a'_{\lda}$ belongs to the $\mZ$-span of class functions
of the form 
\[
\Ind_{\prod_i (L\wr S_{n_i})}^{L\wr S_n} \left(\phi_1^{\witi\times n_1}\Inf_{S_{n_1}}^{L\wr S_{n_1}}\a_{\lda^{(1)}} \times \cdots \times \phi_{s}^{\witi\times n_s} 
\Inf_{S_{n_s}}^{L\wr S_{n_s}}\a_{\lda^{(s)}} \right)
\] 
where
$\phi_1,\ldots, \phi_s \in B$,  $\sum_{i=1}^s n_i=n$ and $\lda^{(i)} \in \cl P(n_i)$ for all $i\in [1,s]$.
Now the assertion of~\eqref{ob1} becomes a consequence of the general identity
\begin{equation}\label{inftens}
\Ind_{\prod_i (L\wr S_{n_i})}^{L\wr S_n} \left(
\prod_{i=1}^s \psi^{\witi\times n_i} \Inf_{S_{n_i}}^{L\wr S_{n_i}} \g_i \right)
=
\psi^{\witi\times n} \Inf_{S_n}^{L\wr S_n} \Ind_{\prod_i S_{n_i}}^{S_n} \prod_{i=1}^s \g_i 
\end{equation}
for $\g_i\in \CF(S_{n_i};K)$ and $\psi\in \CF(L;K)$. 

\eqref{ob2} 
By~\eqref{ob1}, the $\mZ$-span of the set 
$\{ \zeta_{((\phi, \a_\Phi(\phi)) \mid \phi \in B )} \mid \Phi\in \PMap_w (B)\}$
is equal to $\cl C(L\wr S_w)$. Since $|\Irr(L\wr S_w)|=|\PMap_w (B)|$ by Theorem~\ref{thm:wrirr}, the result follows. 
\end{proof}

 

\begin{lem}\label{lem:wrmodind}
 Let $U$ be a subgroup of $L$. Let $M$ be an $RU$-module, where $R$ is any unital commutative ring. Then
\[
 \Ind_{U\wr S_w}^{L\wr S_w} M^{\witi{\otimes} w} \simeq \left( \Ind_U^L M \right)^{\widetilde{\otimes} w}.
\]
\end{lem}

\begin{proof}
 Let $T$ be a set of representatives of left $U$-cosets in $L$. Then $T^{\times w}$ is a set of representatives of left $U\wr S_w$-cosets in $L\wr S_w$, as one can readily check.
Therefore, we have the following equality of free $R$-modules:
\[
 \Ind_{U\wr S_w}^{L\wr S_w} M^{\witi{\otimes} w} = \bigotimes_{t_1,\ldots,t_w\in T} (t_1,\ldots, t_w; 1) \otimes M^{\witi{\otimes} w}.
\]
It is easy to see that the $R$-linear extension of the map 
\[
 (t_1,\ldots,t_w;1)\otimes (m_1\otimes \cdots \otimes m_w) \mapsto (t_1 \otimes m_1) \otimes \cdots \otimes (t_w\otimes m_w), \quad t_i\in T, \; m_i \in M,
\]
is an isomorphism from 
$\Ind_{U\wr S_w}^{L\wr S_w} M^{\witi{\otimes} w}$ onto $\left( \Ind_U^L M \right)^{\widetilde{\otimes} w}$.
\end{proof}

\begin{lem}\label{lem:blwr}
Let $b$ be a block of $L$ with positive defect. Then $b^{\otimes w}$ is a block of $L\wr S_w$.
\end{lem}

\begin{proof}
Let $D$ be a defect group of $b$, so that $D^{\times w}$ is a defect group of $b^{\otimes w}$ as a block of $L^{\times w}$. Since $D\ne \mbf 1$, it is easy to see 
that $C_{L\wr S_w} (D^{\times w}) \le L^{\times w}$. By~\cite[Theorem (61.2)(iv)]{CRII}, this implies that $L\wr S_w$ has only one block covering $b^{\otimes w}$, 
which means precisely that  $b^{\otimes w}$ is a primitive idempotent of $Z(\cl O(L\wr S_w))$. 
\end{proof}

\section{A signed bijection between characters of a symmetric group and a wreath product}\label{sec:main}

Fix $p\in \mN$ and $w,e\in \mZ_{\ge 0}$. Note that here and in some other parts of the paper we do not assume that $p$ is prime, as the arguments are purely combinatorial.

 Consider $S_{pw+e}$, the group of permutations of the set $[1,pw+e]$. 
Let us view $S_p \wr S_w$ as the subgroup consisting of the permutations that fix the elements $pw+1,\ldots,pw+e$ and stabilise the family 
\begin{equation}\label{eq:sets}
 \{ [1,p], [p+1,2p], \ldots, [(w-1)p+1,wp] \}
\end{equation}
of subsets of $[1,pw]$. Also, $(S_p\wr S_w) \times S_e$ may be viewed as a subgroup of $S_{pw}\times S_e$, and hence of $S_{pw+e}$. 

We define the \emph{$p$-type} $\tp_p (g)$ (respectively $\tp_p^{\wre}(g)$) of an element of $S_{pw+e}$ (respectively, $S_p\wr S_w$), as follows (cf.~\cite[Section 2]{Rouquier1994}). 
Suppose that $g\in S_{pw+e}$, and let $(i_1,\ldots,i_r)$ be the cycle type of $g$. Then $\tp_p (g)$ is the partition consisting  
of the numbers $i_l/p$ where $l$ runs over the indices in $[1,r]$ such that $p\mid i_l$. 
Now suppose that $h\in S_p\wr S_w$, and let $y_{\s_1} (x_1)\cdots y_{\s_r} (x_r)$ be an $S_p\wr S_w$-conjugate of $h$ where $\s_1,\ldots,\s_r$ are disjoint marked cycles with 
orders summing to $w$. Then we set $\tp_p^{\wre} (h)$ to be the partition consisting of the numbers $o(\s_l)$ where $l$ runs over the indices in $[1,r]$ such that $x_l$ is a $p$-cycle. 
We remark that for $h\in S_p\wr S_w$ we have $\tp_p^{\wre}(h) =\tp_p (h)$ when $w<p$ (by Lemma~\ref{lem:cyctype}), but this is not the case in general.  

For $s\in \mZ$ define a subset $\cl U_s$ of $S_p\wr S_w$ as follows: 
\begin{equation}\label{eq:defU}
 \Cal U_s = \{ h\in S_p\wr S_w \mid |\tp_p^{\wre}(h)| \ge s \}.
\end{equation}

The following lemma is not required for the proof of Theorem~\ref{thm:symirc}: 
it is needed only to derive Theorem~\ref{thm:val} from the more precise Theorem~\ref{thm:main} below.

\begin{lem}\label{lem:centp}
Assume that $p$ is a prime. 
 Let $\Cal W = \{ g\in S_p\wr S_w \mid C_{S_{pw+e}} (g_p) \le (S_p\wr S_w) \times S_e \}$ (cf.~\S\ref{sub:12}). 
If $e=0$, then $\Cal W = \Cal U_{w-1}$. If $e>0$, then $\Cal W = \Cal U_w$. 
\end{lem}

\begin{proof}
 Let $g\in S_p\wr S_w$. Replacing $g$ with an $S_p\wr S_w$-conjugate, we may assume that 
$g = y_{\s_1} (x_1) \cdots y_{\s_r} (x_r)$ where $\s_1,\ldots,\s_r$ are disjoint marked cycles. 
Moreover, we may assume that all elements $x_i$ that are $p$-cycles are equal to a fixed $p$-cycle $u$. 
Let $z_i = (x_i)_p$ for each $i\in [1,r]$, so that each $z_i$ is either $1$ or $u$. 
Let $h=y_{\s_1} (z_1) \cdots y_{\s_r} (z_r)$, and observe that $h_p = g_p$. Without loss of generality, we have 
$h= y_{\s_1}(u) \cdots y_{\s_s} (u) y_{\tau_1} (1) \cdots y_{\tau_t} (1)$ where $\s_1,\ldots,\s_s,\tau_1,\ldots,\tau_t$ are disjoint marked cycles with 
orders summing to $w$. 

First, suppose that $g\in \cl U_{w}$, so that $t=0$. 
It is easy to see that $h_p$ is $S_p\wr S_w$-conjugate to an element of the form $y_{\ups_1} (u) \cdots y_{\ups_m} (u)$ where $\ups_1,\ldots,\ups_m$ are disjoint marked cycles
of $p$-power order with $\sum_{i} o(\ups_i) =w$. 
Due to Lemma~\ref{lem:cyctype}, one deduces that every element of $C_{S_{pw+e}}(h_p)$ centralises the element 
\[
a_l=\prod_{\substack{i\in [1,m] \\ o(\ups_i)=p^l}} y_{\ups_i} (u) \quad \text{ for each }l\ge 0,
\]
and therefore centralises $\prod_{l} a_l^{p^l} = (u,\ldots, u;1)\in S_p \wr S_w$. 
But
\[
 C_{S_{pw+e}} ((u,\ldots, u;1)) =(C_p\wr S_w)\times S_e \le (S_p\wr S_w) \times S_e,
\]
whence $C_{S_{pw+e}}(h_p) \le (S_p\wr S_w) \times S_e$.
Therefore, $g\in \cl W$.

Now suppose that $g\in \cl U_{w-1} - \cl U_w$ and $e=0$. Then $t=1$ and $o(\tau_1)=1$. We may assume that $\tau_1=w$ (as a marked cycle). By an argument similar to that of the preceding
paragraph, we see that $C_{S_{pw}}(h_p)$ must centralise $(u,u,\ldots,u, 1;1)$, and therefore
$
 C_{S_{pw}}(g_p) \le (C_p\wr S_{w-1}) \times S_p \le S_p \wr S_w
$.
 Hence, $g\in \cl W$ in this case too. 

Conversely, suppose that $g\notin \cl U_{w-\d_{e0}}$. Then one of the following holds (after reordering $\tau_1,\ldots,\tau_t$ if necessary):
\begin{enumerate}[(i)]
 \item\label{case1} $o(\tau_1)>1$;
 \item\label{case2} $t\ge 2$ and $o(\tau_1)=o(\tau_2)=1$;
 \item\label{case3} $t=1$, $o(\tau_1) =1$, and $e>0$.
\end{enumerate}

 In case~\eqref{case1}, let $X_1,\ldots, X_{o(\tau_1)}$ be the sets of the family~\eqref{eq:sets} that correspond to $\Supp(\tau_1)$, ordered so that $X_i$ is sent to $X_{i+1}$ by 
$y_{\tau_1} (1)$. Let $j_1,\ldots, j_{o(\tau_1)}$ be the smallest elements of $X_1,\ldots, X_{o(\tau_1)}$ respectively. Then the decomposition of $y_{\tau_1} (1)$ into a product of 
disjoint cycles (as an element of $S_{pw+e}$) includes 
the cycle $(j_1,\ldots,j_{o(\tau_1)})$, which therefore belongs to $C_{S_{pw+e}}(g_p)$ but does not lie in $(S_p\wr S_w)\times S_e$. Hence, $g\notin \cl W$. 

In case~\eqref{case2}, let $X$ and $Y$ be the sets from the family~\eqref{eq:sets} corresponding to $\tau_1$ and $\tau_2$ respectively, and let $j$ and $l$ be the smallest elements of $X$ and $Y$. Then $g_p=h_p$ fixes both $j$ and $l$, so the transposition $(jl)$ centralises $g_p$ but does not lie in $(S_p\wr S_w) \times S_e$. Hence, $g\notin \cl W$.

In case~\eqref{case3}, we may assume that $\tau_1$ corresponds to the set $[p(w-1)+1, pw]$, whence $g_p=h_p$ fixes $pw$ and therefore 
is centralised by the transposition $(pw, pw+1)\in S_{pw+e}$,
which does not belong to $(S_p\wr S_w)\times S_e$. Hence, $g\notin \cl W$. 
\end{proof}

For each $s\in \mZ$ we define
\begin{equation}\label{eq:defK}
 \Cal K_s = \{ \xi \in \Cal C(S_p\wr S_w) \mid \xi(h)=0 \text{ for all } h\in \cl U_s \}.
\end{equation} 

Let $\rho$ be a fixed $p$-core partition of $e$. 
Denote by $\Irr(S_{pw+e}, \rho)$ the set of all $\chi^{\lda} \in \Irr(S_{pw+e})$ such that $\rho$ is the $p$-core of $\lda$. Note that if $p$ is a prime then 
this is precisely the set of irreducible characters belonging to the block corresponding to $\rho$ (see e.g.\ \cite[Statement 6.1.21]{JamesKerber1981}.) 
Also, let $\cl C(S_{pw+e}, \rho) = \mZ[\Irr(S_{pw+e}, \rho)]$.

\begin{remark}
 A theory of generalised $p$-blocks of symmetric groups, where $p$ is not necessarily a prime, is developed in~\cite{KuelshammerOlssonRobinson2003}.
\end{remark}

\begin{defi}\label{def:wrprbl}
The subset $\Irr_{\pri} (S_p\wr S_w)$ of $\Irr(S_p\wr S_w)$ is defined as follows: if 
$\Phi\in \PMap_w (\Irr(S_p))$, then
$\zeta_{\Phi}\in \cl C_{\pri} (S_p\wr S_w)$ if and only if 
 $\Phi(\chi^{\lda})=\varnothing$ for every $\lda\in \Cal P(p)$ that is not a hook partition. 
We write $\Cal C_{\pri}(S_p\wr S_w) = \mZ[\Irr_{\pri}(S_p\wr S_w)]$.
\end{defi}

\begin{remark}
If $p$ is a prime, then $\Irr_{\pri} (S_p\wr S_w)$ is the set of irreducible characters that lie in the principal block of $S_p\wr S_w$ (by Lemma~\ref{lem:blwr}).
\end{remark}


For any map $\Psi\colon [0,p-1] \to \cl{P}$ define $\Lda\Psi\colon \Irr(S_p) \to \cl{P}$ by
\[
 \Lda\Psi(\chi^{\kappa}) = \begin{cases}
	\Psi(i) & \text{if } \kappa = (p-i, 1^i), \; i\in [0,p-1], \\
	\varnothing & \text{if } \kappa \text{ is not a hook partition.}
\end{cases} 
\]
It follows from Definition~\ref{def:wrprbl} that 
\[
 \Irr_{\pri}(S_p\wr S_w) = \{ \zeta_{\Lda\Psi} \mid \Psi\in \PMap_w ([0,p-1])  \}.
\]

For each $i\in [0,p-1]$
let $b_i$ be the number of beads on runner $i$ of the abacus $\b(\rho,p)$, and let $s_i = p(b_i-1)+ i$, 
so that $s_i$ is the number represented by the bottom bead on runner $i$.
(Since $\rho$ is a $p$-core, the beads on each runner $i$ occupy the top $b_i$ positions of that runner.)
For $i\in [0,p-1]$ let  
\[
 \g_{\rho} (i) = |\{ j\in [0,p-1] \mid s_j <s_i \} |.
\]
This defines a bijection $\g=\g_{\rho}\colon [0,p-1] \to [0,p-1]$, which rearranges the runners according to the number of beads in them, in increasing order.
Recall that $(\lda(0),\ldots,\lda(p-1))$ is the $p$-quotient of $\lda\in \cl P$ (see~\S\ref{sub:22}). 
Let $\lda$ be a partition with $p$-core $\rho$. Define
a map $\Psi_p (\lda)\colon [0,p-1] \to \cl{P}$ by 
\begin{equation}\label{eq:defPsi}
 \Psi_p (\lda) (i) =
\begin{cases}
 \lda(\g^{-1}(p-i-1)) & \text{if } i \text{ is even,} \\
 \lda(\g^{-1}(p-i-1))' & \text{if } i \text{ is odd}
\end{cases}
\end{equation}
and a sign $\e^{(p)} (\lda)$ by 
\begin{equation}\label{eq:sign}
 \e^{(p)}(\lda) = \e_{p} (\lda/\rho) \prod_{\substack{i\in [0,p-1] \\ i \text{ odd}}} (-1)^{|\lda(\g^{-1}(p-i-1))|}. 
\end{equation}
Define a signed map $F_{p,w,\rho}\colon \pm \Irr(S_{pw+e}, \rho) \to \pm \Irr_{\pri}(S_p \wr S_w)$ by setting
\begin{equation}\label{eq:defF}
 F_{p,w,\rho}(\chi^{\lda}) =  \e^{(p)} (\lda) \zeta_{\Lda\Psi_p (\lda)}
\end{equation}
and $F_{p,w,\rho}(-\chi^{\lda}) = -F(\chi^{\lda})$ for all $\chi^{\lda} \in \Irr(S_{pw+e}, \rho)$.

\begin{remark}
The map $F_{p,w,\rho}$ is one of a family of maps constructed in~\cite[\S 2.3]{Rouquier1994}, where instead of the bijection $i\mapsto \g^{-1} (p-i-1)$ one is allowed to use 
any permutation of $[0,p-1]$. 
Conjugating partitions labelled by odd $i$ (cf.\ \eqref{eq:defPsi}) takes place also in a combinatorial description~\cite[Proposition 4.4]{ChuangTan2003}
 of the Morita equivalence~\cite{ChuangKessar2002} between
between the principal block of $S_p\wr S_w$ and a RoCK block of $S_{pw+e}$ when $w<p$ and $p$ is prime, 
both for ordinary and modular irreducible characters. 
Further, if $p$ is prime and $\s$ is another $p$-core, then $F_{p,w,\s}^{-1} F_{p,w,\rho}$ 
yields an isometry between the two blocks of symmetric groups of weight $w$ corresponding to $\rho$ and $\s$, which is one of the perfect isometries constructed by Enguehard~\cite{Enguehard1990}. 
\end{remark}

For each $m\in \mZ_{\ge 0}$
let $\tilde{\pi}_{\rho} = \tilde{\pi}_{\rho}^{(m)}\colon \Cal C(S_{pm} \times S_e) \to \Cal C(S_{pm})$ be the homomorphism obtained by extending $\mZ$-linearly 
the following map defined on the set of irreducible characters $\chi^{\lda}\times \chi^{\s}$ of $S_{pm}\times S_e$:
\begin{equation}\label{eq:defpit}
 \tilde{\pi}_{\rho} (\chi^{\lda} \times \chi^{\s}) = 
\begin{cases}
 \chi^{\lda} & \text{if } \s=\rho, \\
 0 & \text{if } \s\ne \rho.
\end{cases}
\end{equation}

\begin{defi}
 Let $\rho$ be a $p$-core. We say that $\rho$ is \emph{circularly non-decreasing} (with respect to $p$) if there is $j\in [0,p-1]$ such that 
$\g_{\rho} (i) = i-j$ for all $i\in [0,p-1]$, where $i-j$ is understood modulo $p$. In this case we say that $j$ is the \emph{starting point} of $\rho$. 
\end{defi}

\begin{thm}\label{thm:main}
Let $p\in \mN$ and $w,e\in \mZ_{\ge 0}$.
 Let $\rho$ be a $p$-core partition of $e$. Then $F_{p,w,\rho}$ is a signed bijection between $\pm \Irr(S_{pw+e}, \rho)$ and 
$\pm \Irr_{\pri}(S_p \wr S_w)$ that satisfies
\begin{equation}\label{eq:stmain1}
\phantom{a}\quad\quad\: F_{p,w,\rho} (\chi) \equiv \Res^{S_{pw}}_{S_p \wr S_w} \tilde{\pi}_{\rho} \Res^{S_{pw+e}}_{S_{pw} \times S_e} \chi \dmod \Cal K_w \qquad\quad \forall
 \chi\in\pm \Irr(S_{pw+e},\rho).
\end{equation}
Moreover, if $\rho$ is circularly non-decreasing, then 
\begin{equation}\label{eq:stmain2}
\qquad \qquad \; F_{p,w,\rho} (\chi) \equiv \Res^{S_{pw}}_{S_p \wr S_w} \tilde{\pi}_{\rho} \Res^{S_{pw+e}}_{S_{pw} \times S_e} \chi \dmod \Cal K_{w-1} \qquad \forall
 \chi\in\pm \Irr(S_{pw+e},\rho). 
\end{equation}
\end{thm}

\begin{remark}
The core $\rho$ is circularly non-decreasing with starting point $j$ if and only if either $j=0$ and $b_0\le \cdots \le b_{p-1}$ or $j>0$ and 
$b_j \le \cdots \le b_{p-1}<b_0\le b_1\le \cdots \le b_{j-1}$. 
In particular, both the empty partition and the core corresponding to a RoCK block (as defined in the beginning of~\cite[Section 4]{ChuangKessar2002}) 
are circularly non-decreasing. For the proof of Theorem~\ref{thm:symirc}, the congruence \eqref{eq:stmain2} will be needed only when $e=0$, the weaker
congruence \eqref{eq:stmain1} sufficing in the other cases.
\end{remark}

\begin{proof}[Proof of Theorem~\ref{thm:main}]
It follows from Theorem~\ref{thm:corequot} and Definition~\ref{def:wrprbl} that $F_{p,w,\rho}$ is a bijection. The result is clear for $w=0$, so we assume that $w\ge 1$. 
The rest of the proof is based on a combinatorial argument due to 
Rouquier~\cite{Rouquier1994} establishing the commutativity of a part of the diagram~\eqref{eq:main3} below. First, we define the vertical maps of that diagram. 

Let $m,n\in \mZ_{\ge 0}$ be arbitrary integers such that $pm\le n$. Let $\a=(\a_1,\ldots,\a_r)\in \Cal P(m)$. 
The map $d^{\a}_n\colon \Cal C(S_n) \to \CF(S_{n-pm}; K)$ is defined as follows. 
Let $g\in S_{pm}$ be an element of cycle type $p\a$. For every $h\in S_{n-pm}$, view $gh$ as an element of $S_n \ge S_{pm} \times S_{n-pm}$ 
(so that $g$ and $h$ have disjoint supports), and set $(d_n^{\a}\chi)(h) = \chi(gh)$ for every $\chi\in \Cal C(S_n)$.
By~\cite[Th{\'e}or{\`e}me 2.6]{Rouquier1994}, the map $d^{\a}_n$ is the same as the map defined 
in \emph{loc.\ cit.}\ in terms of consecutive removing of rim hooks  (where it is denoted by $r^{\a}$). (The statement of the quoted theorem asserts only that both definitions lead to the 
same value of $(d^{\a}_n \chi)(h)$ for $p$-regular elements $h\in S_{n-pm}$. However, the proof, which is an application of the Murnaghan--Nakayama rule, 
works for $p$-singular elements $h$ just as well.)
The definition in~\cite{Rouquier1994} implies that the image of $d^{\a}_n$ is contained in $\Cal C(S_{n-pm})$.

From now on assume that $m\le w$. 
The map $\d^{\a}\colon \Cal C_{\pri}(S_p\wr S_w) \to \CF(S_p\wr S_{w-m}; K)$ is defined as follows. 
Let $u\in S_p$ be a $p$-cycle. Let $\s_1,\ldots,\s_r$ be disjoint marked cycles in $S_m$ of orders 
$\a_1,\ldots,\a_r$ respectively, so that $y_{\s_1}(u)\cdots y_{\s_r}(u)$ is an element of $S_p\wr S_{m}$.
 For each $z\in S_p\wr S_{w-m}$, consider the element $(y_{\s_1}(u)\cdots y_{\s_r}(u)) z \in (S_p\wr S_{m})\times (S_p\wr S_{w-m}) \le S_p\wr S_w$ and 
set $(\d^{\a}\phi)( z) = \phi((y_{\s_1}(u)\cdots y_{\s_r}(u)) z)$ for every $\phi\in \Cal C_{\pri}(S_p\wr S_w)$.
As is proved in~\cite[Corollaire 2.10]{Rouquier1994}, this definition is equivalent to that given in \emph{loc.\ cit.}; moreover, the image of $\d^{\a}$ is contained 
in $\Cal C_{\pri}(S_p\wr S_{w-m})$.

For each $v\in \mZ_{\ge 0}$, we extend linearly the signed bijection $F_{p,v,\rho}$ to an isomorphism $F_{p,v,\rho} \colon \Cal C(S_{pv+e}, \rho)\to \Cal C_{\pri}(S_p\wr S_v)$ of abelian groups.
By~\cite[Th{\'e}or{\`e}me 2.11]{Rouquier1994}, the right-hand square of the following diagram is commutative:
\begin{equation}\label{eq:main3}
\begin{array}{c}
 \xymatrix{
  \Cal C(S_{pw}) \ar[d]_{d^{\a}_{pw}} &&& \Cal C(S_{pw+e},\rho) \ar[lll]_-{\tilde{\pi}_{\rho} \Res^{S_{pw+e}}_{S_{pw}\times S_e}} 
\ar[rr]^-{F_{p,w,\rho}} \ar[d]_{d^{\a}_{pw+e}} && \Cal C_{\pri}(S_{p} \wr S_w) \ar[d]^{\d^{\a}} \\
\Cal C(S_{p(w-m)}) &&& \Cal C(S_{p(w-m)+e}, \rho) \ar[lll]_-{\tilde{\pi}_{\rho} \Res^{S_{p(w-m)+e}}_{S_{p(w-m)}\times S_e}} 
\ar[rr]^-{F_{p,w-m,\rho}} && \Cal C_{\pri}(S_p \wr S_{w-m})
 }
\end{array}
\end{equation}
We claim that the left-hand square is commutative too. Let $\xi\in\Cal C(S_{pw+e},\rho)$ and $x\in S_{p(w-m)}$. Let $g\in S_{pm}$ be an element of cycle type $p\a$. For each 
$z\in S_e$, view $zxg$ as an element of $S_e\times S_{p(w-m)}\times S_{pm} \le S_{pw+e}$. Then we have
\[
 \begin{split}
 (\tilde{\pi}_{\rho} \Res^{S_{p(w-m)+e}}_{S_{p(w-m)}\times S_e} d^{\a}_{pw+e} \xi) (x) & = \frac{1}{e!} \sum_{z\in S_e} (d^{\a}_{pw+e} \xi) (zx) \chi^{\rho}(z^{-1}) 
= \frac{1}{e!} \sum_{z\in S_e} \xi(zxg)\chi^{\rho}(z^{-1})\\
& = (\tilde{\pi}_{\rho} \Res^{S_{pw+e}}_{S_{pw}\times S_e} \xi) (xg) = (d^{\a}_{pw} \tilde{\pi}_{\rho} \Res^{S_{pw+e}}_{S_{pw}\times S_e} \xi) (x),
 \end{split}
\]
proving our claim. 

If $m=w$ or $m=w-1$, then $S_{p(w-m)}$ is identified with $S_p\wr S_{w-m}$ in the obvious way. 
By definition~\eqref{eq:defK} of $\cl K_w$, to check the congruence~\eqref{eq:stmain1} one needs just to check that the two sides of~\eqref{eq:stmain1} take the same 
values on the elements of the form $h=y_{\s_1} (u) \cdots y_{\s_r} (u)$ where $u$ is a $p$-cycle and $\s_1,\ldots,\s_r$ are disjoint marked cycles in $S_w$ with orders summing to $w$.
By Lemma~\ref{lem:cyctype}, the cycle type of such an element $h$ is $(po(\s_1),\ldots, po(\s_r))$. 
It follows that~\eqref{eq:stmain1} holds if and only if
\begin{equation}\label{eq:main4}
 d^{\a}_{pw} \tilde{\pi}_{\rho} \Res^{S_{pw+e}}_{S_{pw}\times S_e}  =  \d^{\a} F_{p,w,\rho}
\end{equation}
for all $\a\in \Cal P(w)$. 
Similarly, the congruence~\eqref{eq:stmain2} is true if and only if~\eqref{eq:main4} holds for all $\a\in \Cal P(w) \cup \Cal P(w-1)$.

Using the commutativity of the diagram~\eqref{eq:main3} for $m=w$ and the fact that $F_{p,0,\rho} = \tilde{\pi}_{\rho}^{(0)}|_{\cl C(S_{e}, \rho)}$ 
(which is clear from the definitions), 
we see that~\eqref{eq:main4} holds for $\a \in \cl P(w)$. 

Therefore, it remains to show only that~\eqref{eq:main4} holds for all $\a\in \cl P(w-1)$ whenever $\rho$ is circularly non-decreasing. 
In fact, due to the commutativity of the diagram~\eqref{eq:main3} for $m=p-1$, it is enough to prove that 
$F_{p,1,\rho} = \tilde{\pi}_{\rho}^{(1)} \Res^{S_{p+e}}_{S_p\times S_e} |_{\cl C(S_{p+e}, \rho)}$. 

The following argument is similar to the proof of~\cite[Lemma 4(2)]{ChuangKessar2002}.
Let 
\[
 \cl N = \{ \nu \supset_p \rho \mid |\nu/\rho| = p \},
\]
so that $\Irr(S_{p+e},\rho) = \{\chi^{\nu} \mid \nu \in \cl N\}$. 
Then $\Cal N=\{ \nu^0,\ldots,\nu^{p-1} \}$ where 
$\nu^t$ is represented by the abacus with $p$ runners obtained from $\b(\rho, p)$ by moving the bottom bead on runner $t$ one step down (for each $t\in [0,p-1]$). 

Let $j$ be the starting point of $\rho$, so that $\g i = j-i$ for all $i\in [0,p-1]$ (where $j-i$ is interpreted modulo $p$).
Let $l\in [0,p-1]$ and $s_l$ be the number represented by the bottom bead on runner $l$ in $\b(\rho, p)$. 
Then the bottom bead on runner $l$ in $\b(\nu^l, p)$ represents $s_l+p$. 
The hypothesis that $\rho$ is circularly non-decreasing with starting point $j$ implies that the part of $\b(\nu^l, 1)$ 
between positions $s_l$ and $s_l+p$ looks as follows: there are beads in positions $s_l +1,s_l+2,\ldots, s_l+(p-\g l-1)$ 
and no beads in positions $s_l+(p-\g l),\ldots,s_l+p-1$ (also, there is a bead in position $s_l+p$ and no bead in position $s_l$). 
That is, the interval in $\b(\nu^l,1)$ between $s_l$ and $s_l+p$ (inclusive) 
is exactly the abacus representing the hook partition $(\g l+1, 1^{p-\g l -1})$. 
As one can easily see, this means that the skew diagram representing $\nu^l/\rho$ is a translate of the Young diagram of $(\g l+1,1^{p-\g l-1})$. 
Therefore, 
\begin{equation*}
\tilde{\pi}_{\rho} \Res^{S_{p+e}}_{S_p\times S_e} \chi^{\nu^l} = \chi^{\nu^l/\rho}=\chi^{(\g l +1, 1^{p-\g l-1})}.   
\end{equation*}
(The first equality is due to~\cite[Eq.~2.4.16]{JamesKerber1981}.)
By~\eqref{eq:defPsi}, $\Psi_p (\nu^l)(p-\g l-1)=(1)$ and $\Psi_p (\nu^l)(i)=\varnothing$ for $i\ne p-\g l-1$. Hence,
$\zeta_{\Lda\Psi_p (\nu^l)} = \chi^{(\g l+1,1^{p-\g l-1})}$. Also, $\e_p (\nu^l/\rho)=(-1)^{p-\g l -1}$ because there are precisely $p-\g l-1$ beads between 
$s_l$ and $s_l+p$ in $\b(\nu^l,1)$. By~\eqref{eq:sign}, it follows that $\e^{(p)}(\nu^l)=1$. Therefore,
\[
 F_{p,1,\rho} (\chi^{\nu^l}) = \chi^{(\g l +1, 1^{p-\g l-1})} = \tilde{\pi}_{\rho} \Res^{S_{p+e}}_{S_p \times S_e} \chi^{\nu^l},
\]
which is what we require.
\end{proof}

Suppose now that $p$ is a prime. 
Recall that the \emph{height} $\Ht (\chi)$ of a character $\chi\in\Irr(G,b)$, where $G$ is a finite group and $b\in \Bl(G)$, is defined by 
$\Ht(\chi) = v_p (\chi(1))-d$ where $d$ is the defect of the block $b$. 
We also set $\Ht(-\chi)=\Ht(\chi)$. 
Using the following well-known lemma, we will show that the bijection of Theorem~\ref{thm:main} preserves heights of characters.
If $u\in \mZ_{\ge 0}$, the \emph{$p$-adic expansion} of $u$ is defined as the unique expression $u = \sum_{j=0}^{\infty} a_j p^j$ with $0\le a_j<p$ for all $j$. (All but finitely many
terms of this sum are zero.)

\begin{lem}\label{lem:binco}
Let $p$ be a prime. Let $w, u_1,\ldots, u_r$ be nonnegative integers such that $w=u_1+\cdots +u_r$. Let $w=\sum_{j=0}^{\infty} a_j p^j$ and 
$u_i = \sum_{j=0}^{\infty} b_{ij} p^j$ be the $p$-adic expansions of $w,u_1,\ldots, u_r$. Then
\[
 v_p  \binom{w}{u_1,\ldots,u_r}  = \frac{1}{p-1} \left(\sum_{i,j} b_{ij} - \sum_{j} a_j \right).
\]
\end{lem}

\begin{proof} 
 If $v=\sum_{j=0}^{\infty} c_j p^j$ is a $p$-adic expansion, then 
$v_p (v!) = (v - \sum_j c_j)/(p-1)$ (see~\cite[Eq.~(3.3)]{Macdonald1971}). Applying this to all terms of
\[
 \binom{w}{u_1,\ldots,u_r} = \frac{w!}{u_1! \cdots u_r!},
\]
we obtain the result. 
\end{proof}

\begin{prop}\label{prop:heights}
 With notation as in Theorem~\ref{thm:main}, assume that $p$ is prime. Then the bijection $F=F_{p,w,\rho}$  satisfies $\Ht(F(\chi)) = \Ht(\chi)$
for all $\chi\in\Irr(S_{pw+e}, \rho)$. 
\end{prop}

\begin{proof}
 Consider a character $\chi^{\lda} \in \Irr(S_{pw+e}, \rho)$. It follows from~\eqref{eq:defzeta} and~\eqref{eq:defF} that 
\begin{equation}\label{eq:htFchi}
 \Ht(F(\chi))=v_p (F(\chi)(1)) = v_p \binom{w}{|\lda(0)|, \ldots,|\lda(p-1)|} + \sum_{l=0}^{p-1} v_p (\chi^{\lda(l)}(1)).
\end{equation}
For any partition $\mu$ and $i\ge 0$, let $c_i (\mu)$ be the sum of the sizes of the partitions in the $i$-th row of the $p$-core tower of $\mu$ (as defined in~\cite[Section 2]{Olsson1976}).
It follows from the definition that for $i>0$
\[
 c_i (\lda) = \sum_{l=0}^{p-1} c_{i-1}(\lda(l)).
\]
 Let $\sum_{i=0}^{\infty} a_i p^i$ be the $p$-adic expansion of $w$ and $\sum_{i=0}^{\infty} b_{li} p^i$ be the $p$-adic expansion of $|\lda(l)|$ for $l=0,\ldots p-1$. 
By~\cite[Proposition (2.3)]{Olsson1976} we have
\[
 v_p (\chi^{\lda(l)}(1)) = \frac{1}{p-1}\left( \sum_{i=0}^{\infty} c_i (\lda(l)) - \sum_{i=0}^{\infty} b_{li} \right) \quad \text{for all } l\in [0,p-1],
\]
and, by~\cite[Lemma (3.1)]{Olsson1976},
\begin{equation}\label{eq:htchi}
 \Ht (\chi^{\lda}) = \frac{1}{p-1}\left( \sum_{i=1}^{\infty} c_i(\lda) - \sum_{i=0}^{\infty} a_i \right).
\end{equation}
Combining the preceding three displayed equalities, we obtain
\begin{equation}\label{eq:ht2}
 \Ht(\chi^{\lda}) - \sum_{l=0}^{p-1} v_p(\chi^{\lda(l)} (1) ) = \frac{1}{p-1} \left( \sum_{\substack{l\in [0,p-1] \\ i\ge 0}} b_{li} - \sum_{i\ge 0} a_i \right)
= v_p \binom{w}{|\lda(0)|, \ldots,|\lda(p-1)|},
\end{equation}
where the second equality follows from Lemma~\ref{lem:binco}. (We remark that this identity is seen more naturally as a step in an alternative proof of~\eqref{eq:htchi} by induction 
than as a consequence of~\eqref{eq:htchi}.)
The result follows immediately from~\eqref{eq:ht2} and~\eqref{eq:htFchi}.
\end{proof}

When $e=0$, Eq.~\eqref{eq:stmain2} becomes
\[
 F_{p,w,\varnothing} (\chi) \equiv \Res^{S_{pw}}_{S_p\wr S_w} \chi \dmod \cl K_{w-1}.
\]

Together with Lemma~\ref{lem:centp} 
and Proposition~\ref{prop:heights}, this means that Theorem~\ref{thm:val} has been proved, with the exception of the uniqueness statement. 
 Moreover, these results together with Eq.\ \eqref{eq:stmain1} show that an analogue of Theorem~\ref{thm:val}, without uniqueness, holds for $e>0$.

\section{On the span of certain induced characters}\label{sec:I}

Throughout this section, we assume that $p$ is a prime. As in Section~\ref{sec:main}, let $w,e\in \mZ_{\ge 0}$ and $\rho$ be a $p$-core partition of $e$.
Fix $Q\in \Syl_p (S_w)$, so that $P=C_p\wr Q$ is a Sylow $p$-subgroup of $S_{pw}$, where $C_p$ is a fixed subgroup of order $p$ in $S_p$. 
Define $\Cal S_e = \Cal S(S_{pw+e}, P, (S_p\wr S_w) \times S_e)$ (see \S\ref{sub:11}). Recall the subgroups $\Cal K_{s}\subset \cl C(S_p\wr S_w)$ defined by~\eqref{eq:defK}.
The main aim of this section is to prove the following result.

\begin{thm}\label{thm:Iset} 
\begin{enumerate}[(i)]
 \item We have $\Cal K_{w-1}\subset \Cal I(S_p \wr S_w, P, \Cal S_0)$.
\item If $e>0$, then $\Cal K_{w}\subset   \Cal I(S_p\wr S_w, P,\Cal S_e)$.
\end{enumerate}
\end{thm}

Together with Theorem~\ref{thm:main} and Proposition~\ref{prop:heights}, this theorem immediately implies 

\begin{cor}
 Let $b$ be a $p$-block of weight $w$ of a symmetric group $S_{pw+e}$ and $P$ be as above. 
Then the property (IRC-Bl) holds for the quadruple $(S_{pw+e}, b, P, (S_p\wr S_w) \times S_e)$.
\end{cor}

The following corollary is not needed for the proof of Theorem~\ref{thm:symirc} but 
makes Theorem~\ref{thm:Iset} more precise when $w<p$. Together with Corollary~\ref{cor:unique} below (which is essentially the uniqueness part of
Theorem~\ref{thm:val}), it shows, in particular, that the signed bijection witnessing (IRC-Bl) for $e=0$ in the previous corollary is unique for $w<p$.

\begin{cor}
 Assume that $w<p$. We have $\Cal I(S_p\wr S_w, P,\Cal S_0) = \Cal K_{w-1}$. 
Also, if $e>0$, then $\Cal I(S_p \wr S_w, P, \Cal S_e) = \Cal K_w$.
\end{cor}

\begin{proof}
Let $G=S_{pw+e}$ and $H=(S_p\wr S_w)\times S_e$.
Since $w<p$, we have $P=C_p^{\times w}$.
By Theorem~\ref{thm:Iset}, we have $\Cal K_{w-\d_{e0}} \subset \cl I(S_p\wr S_w, P,\cl S_e)$. Conversely, let $\xi\in \cl I(S_p \wr S_w, P, \cl S_e)$,
and suppose for contradiction that $\xi\notin \cl K_{w-\d_{e0}}$, that is, $\xi (h) \ne 0$ for some $h\in S_p\wr S_w$ such that $|\tp_p^{\wre} (h)| \ge w-\d_{e0}$. 
By Definition~\ref{def:I}, after replacing $h$ with an $S_p \wr S_w$-conjugate, we have $h\in L$ for some $L\le S_p\wr S_w$ such that $L\cap P \in \Syl_p (L) \cap \cl S_e$. 
Therefore, replacing $h$ by an $L$-conjugate if necessary, we may assume that $h_p\in T$ for some subgroup $T\le P$ such that $T\in \cl S_e$. 
Hence, $h_p \in P\cap \ls{g}P$ for some $g\in G - H$.

First, consider the case when $e=0$. Since $|\tp_p^{\wre} (h)| \ge w-1$, the cycle decomposition of $h_p$ contains at least $w-1$ $p$-cycles. 
Since $h_p\in \ls{g}P$, this forces $\ls{g}P\le (S_p)^{\times w}$, whence $g\in N_G (S_p^{\times w})=H$, a contradiction.

Now suppose that $e>0$. Then $|\tp_p^{\wre} (h)|\ge w$, so the cycle decomposition of $h_p$ contains at least $w$ $p$-cycles. This means that only one $G$-conjugate of $P$ contains $h_p$, 
whence $g\in N_G(P) \le H$, a contradiction.
\end{proof}

\subsection{An induction theorem for wreath products}\label{sub:41}

Our first objective is to show that $\Cal K_1 \subset \Cal I(S_p\wr S_w,Q, \Cal A(Q))$. In this subsection we state and prove a more general result, Theorem~\ref{thm:wrind}, from which 
that containment is derived below (see Proposition~\ref{prop:wrKzero}).
Let $L$ be a fixed finite group. 
By a \emph{composition} of $w$ we understand a finite sequence $\lda = (\lda_1,\ldots,\lda_r)$ of positive integers 
such that $\lda_1+\cdots+\lda_r = w$, and we write $l(\lda) = r$. 

\begin{defi}\label{def:om}
Let $\xi\in \CF(L\wr S_w; K)$ and 
 $\lda=(\lda_1,\ldots,\lda_r)$ be a composition of $w$. The class function $\om_{\lda} (\xi) \in \CF(L^{\times r} ; K)$ is defined by
\[
\om_{\lda} (\xi) (x_1,\ldots,x_r)= \xi(y_{\s_1} (x_1) \cdots y_{\s_r} (x_r) ), \quad x_1,\ldots,x_r\in L, 
\]
where $\s_1,\ldots,\s_r$ are disjoint marked cycles in $S_w$ of orders $\lda_1,\ldots,\lda_r$ respectively. 
\end{defi}

We will often view $\om_{\lda}$ as an element of $\CF(L; K)^{\otimes r}$. 
We write $\om_m$ instead of $\om_{(m)}$ for $m\in \mN$. 
Note that for any $\xi\in \CF(L\wr S_w;K)$, any composition $\lda=(\lda_1,\ldots,\lda_r)$ of $w$ and 
any $g\in S_r$, we have $\om_{g(\lda)} (\xi) = \ls{g}(\om_{\lda} (\xi))$, where
$g(\lda) = (\lda_{g^{-1} (1)}, \ldots, \lda_{g^{-1} (r)})$ and 
 the action of $S_r$ on $\CF(L;K)^{\otimes r}$ is that induced by the natural action on $L^{\times r}$. 
In particular, the values $\om_{\lda} (\xi)$ for all compositions $\lda$ of $w$ are determined by the values of $\om_{\mu} (\xi)$ for all $\mu\in \cl P(w)$. 


\begin{defi}\label{defi:wrset}
 Let $\cl X$ be a subgroup of $\CF(L; K)$. We define the subgroup $\cl X \wr S_w$ of $\CF(L\wr S_w; K)$ to be the $\mZ$-span of the set of class functions of the form 
$\zeta_{\Th}$ where $\Th = ((\phi_1,\chi_1),\ldots, (\phi_n,\chi_n))\in \Tup_w (L)$ is such that $\phi_i\in \cl X$ for each $i\in [1,n]$.  
\end{defi}

For example, $\cl C(L)\wr S_w = \cl C(L\wr S_w)$ by 
Corollary~\ref{cor:wrextgen} and its proof.

\begin{lem}\label{lem:wrindconv}
 Suppose that $\cl X\le \CF(L;K)$ and $\xi\in \cl X\wr S_w$. Then $\om_{\lda} (\xi) \in \cl X^{\otimes l(\lda)}$ for every $\lda\in \cl P(w)$. 
\end{lem}

\begin{proof}
Let $\Th$, $\lda$ and $\s_1,\ldots,\s_r$ be as in Definition~\ref{defi:wrset}. We may assume that $\xi=\zeta_{\Th}$. 
We will give a proof only in the case when $n=1$, i.e.\ $\Th = ((\phi, \chi))$, 
for it is easy to deduce the general result from this case using the definition of induced class function. By~\eqref{eq:wrcharval} and~\eqref{eq:defzeta}, 
\[
\begin{split}
 \om_{\lda} (\xi) (x_1,\ldots,x_r) &= (\Inf_{S_w}^{L\wr S_w} \chi\cdot \phi^{\witi{\times} w})(y_{\s_1} (x_1) \cdots y_{\s_r} (x_r)) \\
&= \chi(\s_1 \cdots \s_r) \phi(x_1) \phi(x_2) \cdots \phi(x_r) \quad \text{ for all } x_1,\ldots,x_r\in L,
\end{split}
\]
whence $\om_{\lda} (\xi) = \chi(\s_1\cdots \s_r) \phi^{\times r}\in \cl X^{\otimes r}$ because $\chi(\s_1\cdots \s_r) \in \mZ$. 
\end{proof}

\begin{defi}
A subgroup $B$ of a torsion-free abelian group $A$ (written additively) is said to be 
 \emph{pure} in $A$ if for every $a\in A$ such that $na\in B$ for some $n\in \mZ-\{0\}$ we have $a\in B$.
\end{defi}

\begin{thm}\label{thm:wrind}
 Let $\cl X$ be a pure subgroup of $\cl C(L)$. Suppose that $\xi\in \cl C(L\wr S_w)$ and $\om_{\lda} (\xi) \in \cl X^{\otimes l(\lda)}$ for 
every $\lda\in \cl P(w)$. Then $\xi \in \cl X\wr S_w$. 
\end{thm}

For a finite group $G$ let $\scr P(G)$ be the set of all virtual characters $\xi\in \cl C(G)$ such that $\xi$ vanishes on all $p$-singular elements of $G$. 
Recall the following classical result: see (the proof of) \cite[Theorem 5]{Brauer1947}.

\begin{thm}[Brauer]\label{thm:brproj}
For any finite group $G$, the set
 $\scr P(G)$ is the $\mZ$-span of virtual characters of the form 
$\Ind_H^G \th$ where $H$ is a $p'$-subgroup of $G$ and $\th\in \cl C(H)$.
\end{thm}

\begin{exam}\label{exam:Br}
 Suppose that $\cl X = \scr P(L)$. Let $\xi\in \cl C(L\wr S_w)$ be such that $\om_{\lda} (\xi) \in \scr P(L)^{\otimes l(\lda)}$ for all $\lda\in \cl P(w)$. 
By Theorem~\ref{thm:wrind}, $\xi\in \scr P(L) \wr S_w$. 
Using Theorem~\ref{thm:brproj} applied to $L$,
it is not very difficult to deduce that 
$\xi$ must be an integer linear combination of virtual characters of the form $\Ind_{\prod_i U_i\wr S_{w_i}}^{L\wr S_w} \th$ where 
$\sum_i w_i =w$, each $U_i$ is a $p'$-subgroup of $L$ and $\th\in\cl C(\prod_i (U_i\wr S_{w_i}) )$ (cf.\ the proof of Proposition~\ref{prop:wrKzero} below). 
When $w<p$, the above condition on $\xi$ means precisely that $\xi\in \scr P(L\wr S_w)$; moreover, each $\prod_i U_i\wr S_{w_i}$ is a $p'$-subgroup of $L\wr S_w$. 
So Theorem~\ref{thm:wrind} generalises a statement closely related to Brauer's Theorem~\ref{thm:brproj} for $L\wr S_w$ with $w<p$ to the case of arbitrary $w$ (indeed,
if $w<p$, we do not need Theorem~\ref{thm:wrind} to prove Theorem~\ref{thm:Iset}, as Proposition~\ref{prop:wrKzero} follows directly from Theorem~\ref{thm:brproj}).
\end{exam}


We will use the usual \emph{dominance} order on $\cl P(w)$: 
for $\lda,\mu\in \cl P(w)$, we say that $\lda$ \emph{dominates} $\mu$ and write 
$\lda\trianglerighteq \mu$ if $\sum_{i=1}^j \lda_i \ge \sum_{i=1}^j \mu_i$ for all 
$i\in [1,\max(l(\lda),l(\mu))]$, where $\lda_i$ is interpreted as $0$ if $i>l(\lda)$ and a similar convention applies to $\mu$ (see e.g.~\cite[Eq.~1.4.6]{JamesKerber1981}).  
The following proof of Theorem~\ref{thm:wrind} was outlined by the referee and replaces a longer proof in a previous version of the paper. 
We begin with a lemma.

\begin{lem}\label{lem:value}
Let $\lda,\mu\in \cl P(w)$, and for each $i\in \mN$, 
let $\phi_{i,1},\ldots,\phi_{i,m_i(\lda)} \in \CF(L; K)$. Then 
$\om_{\mu} (\zeta_{((\phi_{i,j}, 1_{S_i}) \mid i\in\mN,\, j\in [1,m_i (\lda)])})=0$  
unless $\lda \trianglerighteq \mu$, and 
\[
\om_{\lda} (\zeta_{((\phi_{i,j}, 1_{S_i}) \mid i\in \mN,\, j\in [1,m_i (\lda)])}) = 
\prod_{i\ge 1}\left(
\sum_{g\in  S_{m_i(\lda)}} \phi_{i,g^{-1}(1)} \times \cdots \times \phi_{i,g^{-1} (m_i (\lda))}
\right)
\] 
where 
the product over $i$ is taken in the decreasing order.
\end{lem}

\begin{proof}
Let $\s_1,\ldots,\s_{l(\mu)}$ be disjoint marked cycles in $S_w$ of orders $\mu_1,\ldots,\mu_{l(\mu)}$ respectively, 
and consider $\s=\s_1\ldots \s_{l(\mu)} \in S_w$. 
The first statement of the lemma holds because 
$\s$ has no $S_w$-conjugate in $\prod_{j=1}^{l(\lda)} S_{\lda_j}$ if $\lda$ does not dominate $\mu$. 

For the second statement, define $\s$ as above with $\mu=\lda$. 
We may assume that the Young subgroup $\prod_{j=1}^{l(\lda)} S_{\lda_j}$ is chosen to contain $\s$ and identify $\prod_{i\in \mN} (S_i \wr S_{m_i (\lda)})$ with the normaliser of this Young subgroup in $S_w$. 
If $i\in \mN$ and 
$\{ j\in [1,l(\lda)] \mid \lda_j = i\} = \{ t,t+1,\ldots,t+m_i (\lda)-1\}$, 
set $\s_{i,u} = \s_{t+u-1}$ for each $u\in [1, m_i(\lda)]$. 
Observe that, for $g\in S_w$, we have 
$\ls{g}{\s} \in \prod_{j=1}^{l(\lda)} S_{\lda_j}$ if and only if 
$g\in \prod_{i\in \mN} (S_{i} \wr S_{m_i (\lda)})$. 
Therefore, by definition of induced class function,
for any tuple $x= (x_{i,j} \mid i\in\mN, j\in [1,m_i (\lda)])\in L^{\times l(\lda)}$, 
\begin{align*}
& \om_{\lda} (\zeta_{((\phi_{i,j}, 1_{S_j}) \mid i\in\mN,\, j\in [1,m_i (\lda)])}) (x) = && \\
&\qquad =\prod_{i\in \mN} \left( \sum_{g\in S_{m_i (\lda)}} \phi_{i,1}^{\witi\times i} 
(y_{\s_{i,1}}(x_{i,g(1)})) \cdots \phi_{i,m_i(\lda)}^{\witi\times i} (y_{\s_{i,m_i (\lda)}} (x_{i,g(m_i (\lda))}))
\right)\\
&\qquad = 
\prod_{i\in \mN} \left( \sum_{g\in S_{m_i (\lda)}} 
\phi_{i,1} (x_{i,g(1)})
 \cdots \phi_{i,m_i(\lda)} (x_{i,g(m_i (\lda))}) \right) && 
 \text{ by~\eqref{eq:wrcharval} } \\
 &\qquad = 
 \left( \prod_{i\in \mN}  \sum_{g\in S_{m_i (\lda)}} 
 \phi_{i, g^{-1} (1)} \times \cdots \times \phi_{i,g^{-1} (m_i (\lda))} \right) (x), &&
\end{align*}
where $i$ runs in the decreasing order in all cases. 
The result follows. 
\end{proof}

\begin{proof}[Proof of Theorem~\ref{thm:wrind}]
Since $\cl X$ is pure in $\cl C(L)$, there exists a $\mZ$-basis $B$ of $\cl C(L)$ such that a subset $B_{\cl X}$ of $B$ is a $\mZ$-basis of $\cl X$. 

For every $\lda\in \cl P$, let $\pi^{\lda} = \Ind_{\prod_{i=1}^{l(\lda)} S_{\lda_i}}^{S_{|\lda|}} 1_{S_{\lda_1}} \times \cdots \times 1_{S_{\lda_{l(\lda)}}}$ be the corresponding permutation character of $S_{|\lda|}$. 
For each $\Phi\in \PMap_w (B)$, set 
$\eta_{\Phi} = \zeta_{((\phi,\pi^{\Phi(\phi)}) \mid \phi\in B)} \in \cl C(L\wr S_w)$. 
It is well known that $\{ \pi^{\lda} \mid \lda \in \cl P(n) \}$ is a basis of $\cl C(S_n)$ for all $n\in \mZ_{\ge 0}$; see, for example,~\cite[Theorem 2.2.10]{JamesKerber1981}.
Hence, by Corollary~\ref{cor:otherbasis}\eqref{ob2},
$\{ \eta_{\Phi} \mid \Phi\in \PMap_w (B) \}$ is a $\mZ$-basis of $\cl C(L\wr S_w)$. 
Using the identity~\eqref{inftens}, we obtain
$\zeta_{(\phi,\pi^{\lda})}
= \zeta_{((\phi, \lda_1),\ldots, (\phi,\lda_{l(\lda)}))}$
for all $\phi\in \CF(L;K)$ and $\lda\in\cl P$.
Hence,
\begin{equation}\label{etaPhi}
\eta_{\Phi} = \zeta_{\left( \big(\phi, 1_{S_{\Phi(\phi)_j}} \big) \, \mid \, \phi\in B,\, j\in [1,l(\Phi(\phi))]\right)}
\end{equation}
for all $\Phi\in \PMap_w (B)$.

Let 
$\cl N= \{ \Phi\in \PMap_w (B) \mid \Phi (\psi) \ne \varnothing \text{ for some }\psi\in 
B-B_{\cl X}\}$. 
Let $\xi\in \cl C(L\wr S_w)$ be such that $\om_{\lda} (\xi) \in \cl X^{\otimes l(\lda)}$ for all $\lda\in \cl P(w)$. 
We have $\xi = \sum_{\Phi\in \PMap_w (B)} a_{\Phi} \eta_{\Phi}$
for some integers $a_{\Phi}$. 
For every $\Phi\in \PMap_w (B)$, define 
$\lda (\Phi) = \bigsqcup_{\phi\in B} \Phi(\phi) \in \cl P(w)$. 
Assuming for contradiction that $\xi\notin \cl X\wr S_w$, we see that there exists 
$\Phi\in \cl N$ such that $a_{\Phi}\ne 0$. 
Let $\lda$ be maximal with respect to the dominance order among the partitions $\mu\in \cl P(w)$ such that $\mu=\lda(\Phi)$ for some $\Phi\in \cl N$ with $a_{\Phi}\ne 0$. 
Let $\cl M= \{ \Phi\in \cl N \mid \lda(\Phi) = \lda\}$. 
We have $\om_{\lda} (\eta_{\Phi}) \in \cl X^{\otimes l(\lda)}$ for all $\Phi\in \PMap_w (B) - \cl N$ by Lemma~\ref{lem:wrindconv}, and $\om_{\lda} (\eta_\Phi)=0$ for any $\Phi\in \cl N - \cl M$ by Eq.~\eqref{etaPhi} and Lemma~\ref{lem:value}. Since $\om_{\lda} (\xi) \in \cl X^{\otimes l(\lda)}$, it follows that 
$\sum_{\Phi \in \cl M} a_{\Phi} \om_{\lda} (\eta_{\Phi}) \in \cl X^{\otimes l(\lda)}$. 

For every $\Phi\in \cl M$, let $\b_\Phi = \prod_{i\in \mN} \prod_{\phi\in B} \phi^{\times m_i (\Phi(\phi))}\in B^{\times l(\lda)}$; here, and in what follows, the product over $i$ is taken in the decreasing order and the product over $\phi$ is taken in an arbitrary but fixed order. 
By Eq.~\eqref{etaPhi} and Lemma~\ref{lem:value}, 
$\om_{\lda} (\eta_\Phi) = \sum_{g\in \prod_{i\in \mN} S_{m_i (\lda)}} \ls{g}\b_{\Phi}$ for any $\Phi\in \cl M$, where the action of $S_{|\lda|}$ on $\CF(L;K)^{\otimes |\lda|}$ is as described after Definition~\ref{def:om}
and we view $\prod_{i\in \mN} S_{m_i(\lda)}$ as a subgroup of $S_{l(\lda)}$ 
by identifying each factor $S_{m_i (\lda)}$ with the set of elements of 
$S_{l(\lda)}$ that fix all the elements of $[1,l(\lda)]$ lying outside the subset 
$[1+\sum_{t>i} m_t (\lda), \sum_{t\ge i} m_t (\lda)]$. 
 Observing that class functions $\b_{\Phi}$ and $\b_{\Psi}$ are not $\prod_{i\in \mN} S_{m_i (\lda)}$-conjugate if $\Phi,\Psi\in \cl M$ are distinct and that 
the stabiliser in $\prod_{i\in \mN} S_{m_i (\lda)}$ of $\b_{\Phi}$ is 
isomorphic to $\prod_{i\in \mN} \prod_{\phi\in B} S_{m_i (\Phi(\phi))}$ for 
$\Phi\in \cl M$, we 
deduce that any such $\b_{\Phi}$ appears with coefficient 
$a_{\Phi} \prod_{i\in \mN} \prod_{\phi\in B} m_i (\Phi (\phi))!$
in the expansion of $\sum_{\Psi\in \cl M} a_{\Psi} \om_{\lda} (\eta_{\Psi})$ with respect to the $\mZ$-basis $B^{\times l(\lda)}$ of $\cl C(L)^{\otimes l(\lda)}$. 
We also have $\b_{\Phi} \notin B_{\cl X}^{\times l(\lda)}$ (as $\Phi(\psi)\ne \varnothing$ for some $\psi\in B-B_{\cl X}$), so 
the above coefficient must be 0, whence $a_{\Phi} =0$ for all $\Phi\in \cl M$. This is a contradiction. 
\end{proof}

\subsection{Proof of Theorem~\ref{thm:Iset}}\label{sub:42}

Recall that $Q\in \Syl_p (S_w)$ and $\cl A(Q)$ is the set of all subgroups of $Q$. As in \S\ref{sub:41}, $L$ is an arbitrary finite group. 

\begin{lem}\label{lem:wrKzero_ind}
 Suppose that $w=w_1+\cdots +w_r$ where $w_1,\ldots,w_r\in \mN$. Let us view $S_{w_1},\ldots S_{w_r}$ as subgroups of $S_w$ acting on disjoint subsets of $S_w$, so that $\prod_i S_{w_i}$ is a Young subgroup. Let $Q^{(i)}\in \Syl_p (S_{w_i})$ and 
$\xi_i\in \cl I(L\wr S_{w_i}, Q^{(i)}, \cl A(Q^{(i)}) )$. 
Then 
\[
 \Ind_{\prod_{i} L\wr S_{w_i}}^{L\wr S_w} \prod_{i=1}^r \xi_i \in \cl I(L\wr S_w, Q, \cl A(Q)).
\]
\end{lem}

\begin{proof}
We may assume that $Q\ge Q^{(1)} \times \cdots \times Q^{(r)}$. 
For each $i\in [1,r]$ let $\cl T^{(i)}$ be the set of subgroups $M\le L\wr S_{w_i}$ such that 
$M\cap Q^{(i)} \in \Syl_p (M)$. Let $i\in [1,r]$. Since  $\xi_i\in \cl I(L\wr S_{w_i}, Q^{(i)}, \cl A(Q^{(i)}) )$, we may 
write $\xi_i$ as a sum of virtual characters of the form $\Ind_M^{L\wr S_{w_i}} \a$ where $M\in \cl T^{(i)}$ and $\a \in \cl C(M)$. 
Therefore, $\prod_i \xi_i$ expands as a sum of terms of the form 
\[
 \eta = \prod_{i=1}^r \Ind_{M_i}^{L\wr S_{w_i}} \a_i = \Ind_{\prod_i M_i}^{\prod_i (L\wr S_{w_i})} \prod_{i=1}^r \a_i
\]
where 
$M_i\in \cl T^{(i)}$ and $\a_i \in \cl C(M_i)$ for each $i$. 

Let us fix such a term $\eta$. Since $M_i\cap Q^{(i)}\in \Syl_p (M_i)$ for all $i$, we have $(\prod_{i} M_i) \cap (\prod_i Q^{(i)}) \in \Syl_p (\prod_i M_i)$,
and hence $(\prod_i M_i) \cap Q \in \Syl_p (\prod_i M_i)$ (as $Q$ contains $\prod_i Q^{(i)}$). 
Therefore, $\Ind_{\prod_i (L\wr S_{w_i})}^{L\wr S_w} \eta \in \cl I(L\wr S_w, Q, \cl A(Q))$, and the result follows. 
\end{proof}

\begin{prop}\label{prop:wrKzero}
We have $\cl K_1 \subset \cl I(S_p\wr S_w, Q, \cl A(Q))$. 
\end{prop}

\begin{proof}
Arguing by induction, we may assume that the result holds for smaller values of $w$. 
 Let $\xi\in \cl K_1$. Let $\lda\in \cl P(w)$ and $l(\lda)=r$. By~\eqref{eq:defK}, we have $\om_{\lda} (\xi) (x_1,\ldots,x_r) =0$ whenever
 $x_1,\ldots,x_r\in S_p$ and at least one $x_i$ is a $p$-cycle. Hence, $\om_{\lda} (\xi)\in \scr P(L)^{\otimes r}$. 
By Theorem~\ref{thm:wrind}, it follows that $\xi\in \scr P(L) \wr S_w$. 

We are to show that $\xi\in \cl I(S_p\wr S_w, Q, \cl A(Q))$.
 So we may assume that $\xi = \zeta_{\Phi}$ for some $\Phi=((\phi_i,\chi_i))_{i=1}^s\in \Tup_w(S_p)$ such that $\phi_i\in \scr P(S_p)$
for all $i\in [1,s]$.
By Lemma~\ref{lem:wrKzero_ind} and the inductive hypothesis, we may in fact assume that $s=1$, i.e.\ $\Phi=((\phi, \chi))$ for some $\phi\in \scr P(S_p)$ and $\chi\in \cl C(S_w)$. 
By Theorem~\ref{thm:brproj}, we may write $\phi = \sum_j n_j \Ind_{M_j}^{S_p} \a_j$ where $j$ runs over some finite set, $n_j\in \mZ$, 
$M_j$ is a $p'$-subgroup of $S_p$, and $\a_j\in \Irr(M_j)$ (for each $j$). 
Hence, Lemma~\ref{lem:lincomb} expresses $\phi^{\witi{\times} w}$ as 
$\sum_j (\Ind_{M_j}^{S_p} \a_j)^{\witi{\times} w} \Inf_{S_w}^{S_p\wr S_w} \k_{w,n_j}$ 
plus a sum of terms that all lie in $\scr P(S_p) \wr S_w$ by
Lemma~\ref{lem:wrKzero_ind} and the inductive hypothesis. By Lemma~\ref{lem:wrmodind}, 
\[
 (\Ind_{M_j}^{S_p} \a_j)^{\witi{\times} w} = \Ind_{M_j\wr S_w}^{S_p\wr S_w} (\a_j^{\witi{\times} w}) \in \cl I(S_p\wr S_w, Q, \cl A(Q)).
\]
It follows from Definition~\ref{def:I} and~\cite[Problem (5.3)]{IsaacsBook} that
$\cl I(S_p\wr S_w, Q, \cl A(Q))$ is an ideal of the ring $\cl C(S_p\wr S_w)$. Hence,
$\zeta_{\Phi} = \phi^{\witi{\times} w}\Inf_{S_w}^{S_p\wr S_w}\chi \in \cl I(S_p\wr S_w, Q, \cl A(Q))$. 
\end{proof}

For each $i\in [0,w]$ let $A_i\simeq S_i$ be the group of permutations of 
$[1,i]$ and $B_i\simeq S_{w-i}$ be the group of permutations of $[i+1, w]$, so that $A_i\times B_i$ is a Young subgroup of $S_w$. 
Let $Q_i\in \Syl_p (B_i)$, chosen so that $Q_0 \ge Q_1 \ge \cdots \ge Q_w = \mbf 1$. We may (and do) assume that $Q=Q_0$. 
Further, let $P_i$ be a Sylow $p$-subgroup of $S_p\wr A_i$ such that $P_i\cap A_i \in \Syl_p (A_i)$. 
We choose $P_0,\ldots, P_w$ so that $P_0\le P_1\le \cdots \le P_w$ and $P_w=P$.

Recall that $\Cal S_e = \Cal S(S_{pw+e}, P, (S_p\wr S_w) \times S_e)$.
The following lemma describes what we need to know about the sets $\Cal S_e$ to prove Theorem~\ref{thm:Iset}.

\begin{lem}\label{lem:S}
\begin{enumerate}[(i)]
\item\label{l1} For each $i\in [0,w-2]$ we have $P_i\times Q_i\in\Cal S_0$. 
\item\label{l2} Assume that $e>0$. Then $P_i \times Q_i\in\Cal S_e$ for each $i\in [0,w-1]$. 
\end{enumerate}
\end{lem}

\begin{proof}
\eqref{l1} Let $i\in [0,w-2]$. 
Choose a non-identity element $z\in C_{B_i}(Q_i)$. Such an element always exists: if $Q_i$ is non-trivial, 
we can take $z$ to be any element of $Z(Q_i)-\mbf 1$; otherwise, take $z$ to be any non-trivial element of $B_i$. 
Let $g$ be the element of $S_{pw}$ that acts on the set $X=\{ pi+1, p(i+1)+1, p(i+2)+1, \ldots, p(w-1)+1 \}$ in the same way as $z$ does and fixes $[1,pw]-X$ pointwise.  
Then $g\notin S_p\wr S_w$ and $g$ centralises both $S_p \wr A_i$ and $Q_i$. Therefore, $P_i\times Q_i \le \ls{g}P\cap P\in \Cal S_0$. 

\eqref{l2} The preceding argument shows that $P_i\times Q_i\in \Cal S_e$ for $i\in [0,w-2]$. 
Observe that $Q_{w-1}$ is trivial and $P_{w-1}$ is centralised by the involution
$v=(pw,pw+1)\in S_{pw+e} - ((S_p\wr S_w)\times S_e)$. Thus, $P_{w-1}\times Q_{w-1}\le \ls{v}P\cap P \in \Cal S_e$.
\end{proof}

For each $i\in [0,w]$ let 
\[
\Cal S^i = \{ S\le P \mid S\le P_j \times Q_j \text{ for some } j\le i \}.  
\]

\begin{prop}\label{prop:Isetprecise}
For each $i\in [1,w]$ we have $\Cal K_i \subset \Cal I(S_p\wr S_w, P, \Cal S^{i-1})$. 
\end{prop}

\begin{proof}
We will prove the proposition by induction on $i$. For $i=1$ the result is true by Proposition~\ref{prop:wrKzero}.
 
Suppose that $1<i\le w$ and the proposition holds for all smaller values of $i$. 
Similarly to~\eqref{eq:defK}, for $j\in [0,w]$ and $s\in \mZ$, set 
\[
\begin{split}
 \Cal K_s (A_j) &= \{ \phi\in \Cal C(S_p\wr A_j) \mid \phi(h)=0 \text{ for all } h\in \Cal U_s \cap (S_p\wr A_j) \} \quad \text{and} \\
 \Cal K_s (B_j) & = \{ \phi \in \Cal C(S_p\wr B_j) \mid \phi(h) =0 \text{ for all } h\in \Cal U_s \cap (S_p\wr B_j) \},
\end{split}
\]
where $\cl U_s$ is given by \eqref{eq:defU}.
Then $\Cal K_{i-1} (A_{i-1})$ 
is pure in $\cl C(S_p\wr A_{i-1})$, so we can choose a subset $\{\b_1,\ldots,\b_r\}$ of $\Cal C(S_p\wr A_{i-1})$ such that 
$\{ \b_j + \Cal K_{i-1}(A_{i-1}) \}_{j=1}^r$ is a basis of 
$\Cal C(S_p\wr A_{i-1})/\Cal K_{i-1}(A_{i-1})$. 
Let $\xi \in \Cal K_i$. Then 
\begin{equation}\label{eq:Isetprecise1}
 \Res^{S_p\wr S_w}_{(S_p\wr A_{i-1})\times (S_p\wr B_{i-1})} \xi  \equiv  \sum_{j=1}^r (\b_j \times \g_j)  \dmod \Cal K_{i-1}(A_{i-1}) \otimes \Cal C(S_p\wr B_{i-1}) 
\end{equation}
for uniquely determined $\g_1,\ldots,\g_r \in \Cal C(S_p\wr B_{i-1})$. We claim that $\g_j\in \Cal K_1 (B_{i-1})$ for each $j\in [1,r]$. 
Indeed, consider $z\in \Cal U_1\cap (S_p\wr B_{i-1})$.
For each $h\in \Cal U_{i-1} \cap (S_p\wr A_{i-1})$, we have $hz\in \Cal U_i$, and therefore, since $\xi\in \cl K_i$, 
\[
 0=\xi (hz) = \sum_{j=1}^r \g_j (z) \b_j (h),
\]
where the second equality is due to~\eqref{eq:Isetprecise1}.
Since all characters of $S_p\wr B_{i-1} \simeq S_p\wr S_{w-i+1}$ are $\mZ$-valued, this 
means that 
\[
 \sum_{j=1}^r \g_j (z) \b_j \in  \Cal K_{i-1} (A_{i-1}), 
\]
whence $\g_j(z)=0$ (by the choice of $\b_1,\ldots,\b_r$). So $\g_j \in \Cal K_1 (B_{i-1})$, as claimed. 

By Proposition~\ref{prop:wrKzero}, 
we have $\g_j \in \cl I(S_p\wr B_{i-1}, Q_{i-1}, \Cal A(Q_{i-1}))$ for each $j$. Hence, the virtual character
\begin{equation}\label{eq:Isetprecise2}
 \th = \Ind_{(S_p\wr A_{i-1})\times (S_p\wr B_{i-1})}^{S_p\wr S_w} \left( \sum_{j=1}^r \b_j \times \g_j \right)
\end{equation}
belongs to $\Cal I(S_p\wr S_w, P_{i-1} \times Q_{i-1}, \Cal A(P_{i-1}\times Q_{i-1}))\subset \Cal I(S_p\wr S_w, P, \Cal S^{i-1})$. 

Let $\xi'= \xi - \th$. We claim that $\xi'\in \cl K_{i-1}$. That is, we will show that $\xi'(h)=0$ for all $h\in \cl U_{i-1}$. 
First, suppose that no $S_p\wr S_w$-conjugate of $h$ lies in $((S_p\wr A_{i-1})\cap \Cal U_{i-1})\times (S_p\wr B_{i-1})$. This implies that $|\tp^{\wre}(h)|\ne i-1$, 
so in fact $h\in \Cal U_{i}$. Then $\xi(h)=0$ because $\xi\in \Cal K_i$, and
 $\th(h)=0$, so $\xi'(h)=0$.

Now assume that some $S_p\wr S_w$-conjugate of $h$ belongs to $((S_p\wr A_{i-1})\cap \Cal U_i)\times (S_p\wr B_{i-1})$. We may replace $h$ with this conjugate, so that $h=lz$ for some 
$l\in (S_p\wr A_{i-1})\cap \Cal U_{i-1}$ and $z\in S_p\wr B_{i-1}$. 
Without loss of generality, $l= y_{\s_1}(u) \cdots y_{\s_t} (u)$ where $u$ is a $p$-cycle in $S_p$ and $\s_1,\ldots,\s_t$ are disjoint marked cycles in $S_{i-1}$ with orders summing to $i-1$. 
By~\eqref{eq:Isetprecise1}, 
\[
 \xi(h) = \sum_{j=1}^r \b_j (l) \g_j(z).
\]
We assert that $\th(h)$ is also equal to $\sum_{j=1}^r \b_j (l) \g_j(z)$.
Due to~\eqref{eq:Isetprecise2}, in order to prove this, it suffices to show that,
whenever $g\in (S_p\wr S_w)-((S_p\wr A_{i-1})\times (S_p\wr B_{i-1}))$ and 
$\ls{g}{h}\in (S_p\wr A_{i-1})\times (S_p\wr B_{i-1})$, we have 
$(\sum_{j=1}^r \b_j\times \g_j) (\ls{g}{h})=0$. But in this situation, writing $\ls{g}{h} =l'z'$ (with $l'\in S_p\wr A_{i-1}$, $z'\in S_p\wr B_{i-1})$, we see that $g$ must send the support of at least one of the cycles $\s_s$ into $[i+1,w]$, whence $z'\in \cl U_1$. Hence $\g_j(z')=0$ for each $j$ (as $\g_j\in \Cal K_1 (B_{i-1})$), and so $(\sum_{j=1}^r \b_j\times \g_j)(\ls{g}h)=0$. 
We have proved that $\th(h) = \xi(h)$, that is, $\xi'(h)=0$, as claimed. 

Since $\xi'\in \cl K_{i-1}$, we have 
 $\xi' \in \Cal I(S_p\wr S_w, P, \Cal S^{i-1})$ by the inductive hypothesis (as $\cl S^{i-2} \subset \cl S^{i-1}$). 
As $\th\in \cl I(S_p\wr S_w, P, \cl S^{i-1})$, we conclude that 
$\xi = \xi' +\th \in \cl I(S_p\wr S_w, P, \cl S^{i-1})$. 
\end{proof}

\begin{proof}[Proof of Theorem~\ref{thm:Iset}]
By Lemma~\ref{lem:S}, $\Cal S^{w-1-\d_{e0}} \subset \Cal S_e$ for all $e\in \mZ_{\ge 0}$. By Proposition~\ref{prop:Isetprecise}, 
$\cl K_{w-\d_{e0}} \subset \cl I(S_p\wr S_w, P, \cl S^{w-\d_{e0} -1})$. Hence, $\cl K_{w-\d_{e0}} \subset \cl I(S_p\wr S_w, P, \cl S_e)$. 
\end{proof}

\section{From $S_p\wr S_w$ to the normaliser of a Sylow subgroup}\label{sec:indarg}

Throughout the section we assume that $p$ is a fixed prime and $w\in \mN$, noting that Theorem~\ref{thm:symirc} is trivial for blocks of weight $0$. 
As before, let $Q\in \Syl_p (S_w)$ and $P=C_p\wr Q$. Consider the chain 
\[
 S_p\wr S_w \ge N_{S_p} (C_p) \wr S_w \ge N_{S_p} (C_p) \wr N_{S_w} (Q) \ge N_{S_{pw}} (P). 
\]
For each pair $(G,H)$ of neighbouring elements of this chain, we will construct a signed bijection from $\pm \Irr_{p'} (G)$ onto $\pm\Irr_{p'} (H)$
that satisfies an appropriate property (cf.\ Definition~\ref{defi:compliant} below). Composing these signed bijections with each other and with the bijection
$F_{p, w, \rho}$ defined by~\eqref{eq:defF}, we will be able to prove Theorem~\ref{thm:symirc}.

\subsection{Derived equivalences between wreath products}\label{sub:51}

In this subsection we will show that there is a signed bijection between principal block characters of $S_p\wr S_w$ and those of $N_{S_p} (C_p) \wr S_w$
witnessing (IRC-Bl) (see Corollary~\ref{cor:wrmarcus}).
It is shown in~\cite[\S 3.1]{Evseev2010m} that such a bijection exists provided there is a derived equivalence between these blocks satisfying certain conditions. 
Thus, our signed bijection will be obtained as a character-theoretic ``shadow'' of a stronger result (Theorem~\ref{thm:wrequiv}) concerning a derived equivalence between blocks of wreath products that was constructed by Marcus~\cite{Marcus1996}.
In this subsection we rely on the definitions and conventions of~\cite[\S 3.1]{Evseev2010m}, some of which we now recall. 

Let $G$ and $H$ be finite groups. If $M$ is an $\cl O G$-$\cl OH$-bimodule, then $M$ is identified with the $\cl O(G\times H)$-module 
defined by $(g,h)\cdot m = gmh^{-1}$ for $g\in G$, $h\in H$ and $m\in M$. Recall that if $\cl S$ is a set of subgroups of $G$, then an
$\cl OG$-module $N$ is said to be \emph{$\cl S$-projective} if for every indecomposable summand $N'$ of $N$ there exists $S\in \cl S$ such that $N'$ is relatively $S$-projective.
If $P$ is a subgroup of both $G$ and $H$, then we write
\[
 \D P = \{ (x,x) \mid x\in P \} \le G\times H.
\]
If $\cl X$ is a set of subgroups of such a subgroup $P$, we define $\D \cl X = \{ \D X \mid X\in \cl X \}$.

Let $C$ be a bounded chain complex of $\Cal OG$-$\Cal OH$-bimodules:
\[
 C = \qquad \cdots \to C^{-1} \to C^0  \to C^{1} \to \cdots.
\]
If $c$ is an element of a term $C^i$, we write $\deg(c)=i$. 
Note that $C^{\otimes w}$ is a complex of $\Cal O(G^{\times w})$-$\Cal O(H^{\times w})$-bimodules. 
By~\cite[Lemma 4.1]{Marcus1996}, there exists a function
$\e\colon S_w \times (\mZ/2)^w \to \mZ/2$ such that, 
writing $\e_\s (i_1,\ldots,i_w)$ instead of $\e(\s, (i_1+2\mZ,\ldots,i_w+2\mZ))$ (for $\s\in S_w$, $i_1,\ldots,i_w\in \mZ$) and 
defining 
\begin{equation}\label{eq:der1}
 \sigma\cdot (c_1\otimes \cdots \otimes c_w) = (-1)^{\e_{\sigma} (\deg(c_1), \ldots, \deg(c_w)) } c_{\sigma^{-1} (1)} \otimes \cdots \otimes c_{\sigma^{-1}(w)},
\end{equation}
where $c_1,\ldots,c_w$ are elements of terms of $C$,
we obtain an action of $S_w$ on $C^{\otimes w}$. Moreover, this map $\e$ may be chosen to satisfy the equations
\begin{align} 
 \e_{\s} (0,\ldots,0) &= 0,  \label{eq:eps1} \\ 
 \e_{\s}(i_1,\ldots,i_l+1,\ldots, i_w) &= \e_{\s} (i_1,\ldots, i_l, \ldots i_w)+i_1+\cdots +i_{l-1} \notag \\
 & + i_{\s^{-1}(1)}+\cdots + i_{\s^{-1}(\s(l)-1)},  \label{eq:eps2} \\
 \e_{\tau\s} (i_1,\ldots, i_w) &= \e_{\tau}(i_1,\ldots, i_w) + \e_{\s}(i_{\tau^{-1} (1)}, \ldots, i_{\tau^{-1}(w)} ), \label{eq:eps3}
\end{align}
for all $\s,\tau\in S_w$, $i_1,\ldots, i_w\in \mZ$ and $l\in [1,w]$: see the proof of Lemma 4.1 in~\cite{Marcus1996}, particularly, Eqs.\ (4.2.1) and (4.2.2).
It is straightforward to deduce that $C^{\otimes w}$ becomes a complex of $\Cal O((G\times H) \wr S_w)$-modules 
when one combines the $\Cal OG$-$\Cal OH$-bimodule structure with~\eqref{eq:der1}.
By $(G\times H) \wr \D S_w$ we understand the subgroup $(G\times H)^{\times w} \rtimes \D S_w$ of $(G\wr S_w) \times (H\wr S_w) = (G\times H)^{\times w} \rtimes (S_w \times S_w)$.
Following~\cite{Marcus1996}, we define
\[
 C\wr S_w = \Ind_{(G\times H)\wr \D S_w}^{(G\wr S_w) \times (H\wr S_w)} C^{\otimes w}. 
\]
(In~\cite{Marcus1996}, the complex $C\wr S_w$ is constructed using somewhat different language, but the outcome may be easily checked to be isomorphic to the one given above.)

Now let $b$ and $e$ be $p$-blocks of $G$ and $H$ respectively. 
Then $b^{\otimes w}$ and $e^{\otimes w}$ are central idempotents of $\Cal O(G\wr S_w)$ 
and $\Cal O(H\wr S_w)$ respectively.

\begin{thm}[\protect{\cite[Theorem 4.3(b)]{Marcus1996}}]\label{thm:marcus}
 If $C$ is a Rickard tilting complex of $\Cal OGb$-$\Cal OHe$-bimodules, 
then $C\wr S_w$ is a Rickard tilting complex of $\Cal O(G\wr S_w)b^{\otimes w}$-$\Cal O(H\wr S_w)e^{\otimes w}$-bimodules.
\end{thm}

In order to state our theorem, we need to recall some further definitions of~\cite[\S 3.1]{Evseev2010m}.
Let $D$ be a fixed $p$-subgroup of an arbitrary finite group $A$, and 
let $\Cal S$ be a set of subgroups of $D$. Let $X$ be a  complex of $\cl OA$-modules. 
As in~\cite[Definition 3.3]{Evseev2010m}, we say that $X$ is \emph{$\Cal S$-tempered} if at most one term $X^i$ is not $\Cal S$-projective. 
If such a term $X^i$ exists, let $X^i \simeq M\oplus N$ where $N$ is the $\Cal S$-projective summand of $X^i$ of maximal dimension. 
We say that $M$ is the \emph{pivot} (or \emph{$\Cal S$-pivot}) of $X$ and that this pivot is located in degree $i$.
If all terms of $X$ are $\Cal S$-projective, the pivot of $X$ is defined to be $0$.
Further, suppose that $B$ is a subgroup of $A$ containing $N_A (D)$. Let $f\in \Bl(A)$ be such that $D$ is a defect group of $f$. 
The $\cl O(A\times B)$-module that is the Green correspondent of the $\cl O(A\times A)$-module $\cl OAf$ is denoted by $\fr{Gr}(A,f,B)$. 

\begin{thm}\label{thm:wrequiv}
 Let $G$ be a finite group and $b\in \Bl(G)$. Let $D$ be a defect group of $b$, and assume that $D\ne \mbf 1$. 
Suppose that $H$ is a subgroup of $G$ containing $N_G (D)$, and let $e\in \Bl(H)$ be the Brauer correspondent of $b$. 
Let $\cl S = \cl S(G,D,H)$ and $\Cal S' = \Cal S(G\wr S_w, D\wr Q, H\wr S_w)$ (where $Q\in \Syl_p(S_w)$). 
Assume that $C$ is a 
Rickard tilting complex of $\Cal OGb$-$\Cal OHe$-bimodules which is $\Delta \Cal S$-tempered with pivot $\fr{Gr}(G,b, H)$ located in degree $0$. Then
$C\wr S_w$ is $\Delta \cl S'$-tempered with pivot $\fr{Gr}(G\wr S_w, b^{\otimes w}, H\wr S_w)$ located in degree $0$.  
\end{thm}

In the remainder of this subsection we will use the notation specified by the preceding hypothesis. 
We note that, by Lemma~\ref{lem:blwr}, $b^{\otimes w}$ and $e^{\otimes w}$ are blocks
of $G\wr S_w$ and $H\wr S_w$ respectively. 

Let $\Cal X$ be a set of subgroups of $D$. 
We define $\tilde{\Cal X} = \tilde{\Cal X}^{(w)}$ to be the set of subgroups $S$ of $D\wr Q$ such that $S$ is contained in a subgroup of the form 
$(D_1\wr Q_1) \times \cdots \times (D_r\wr Q_r)$ where 
\begin{enumerate}[(i)]
 \item for some $w_1,\ldots,w_r$ such that $\sum_j w_j=w$, there are subgroups $S_{w_1},\ldots, S_{w_s}$ of $S_w$ acting on disjoint subsets of $[1,w]$, 
and $Q_j \in \Syl_p(S_{w_j})$ for all $j$; 
 \item $Q_j\le Q$ for each $j$;
 \item $D_j\le D$ for each $j\in [1,r]$ and there is at least one $t\in [1,r]$ such that $D_t \in \Cal X$.
\end{enumerate}

\begin{lem}\label{lem:S'}
 We have $\tilde{\Cal S}\subset \Cal S'$. 
\end{lem}

\begin{proof}
Let $S=(D_1\wr Q_1) \times \cdots \times (D_r\wr Q_r)$ be a maximal element of $\tilde{\Cal S}$. 
 With notation as above, choose $t$ such that $D_t\in \Cal S$. Then $D_t \le \ls{x}D\cap D$ for some $x\in G-H$. The element $g=(1,\ldots,1,x,\ldots,x,1,\ldots,1; 1)\in G\wr S_w$, where the symbols $x$ occupy the chosen subset of size $w_t$, centralises the subgroup $D_s\wr Q_s$ for each $s\ne t$ as well as the subgroup $Q_t$. It follows that 
$S\le \ls{g}{(D\wr Q)} \cap (D\wr Q)$. Also, $g\notin H\wr S_w$, and hence $S\in \Cal S'$.  
\end{proof}

\begin{lem}\label{lem:rp1}
  Let $\Cal X$ be a set of subgroups of $D$.
 Suppose that $Y$ is a $\D \Cal X$-projective indecomposable $\Cal O(G\times H)$-module. Then
\[
 W = \Ind_{(G\times H)\wr \D S_w}^{(G\wr S_w)\times (H\wr S_w)} Y^{\widetilde{\otimes} w}
\]
is $\D \tilde{\cl X}$-projective. 
\end{lem}

\begin{proof}
 Let $T\in \Cal X$ and an $\Cal OS$-module $U$ be such that $Y$ is a summand of $\Ind_{\Delta T}^{G\times H} U$. 
Then $Y^{\witi{\otimes} w}$ is a summand of $(\Ind_{\D T}^{G\times H} U)^{\witi{\otimes} w}$, which is isomorphic to
$\Ind_{\D (T\wr S_w)}^{(G\times H) \wr \D S_w} U^{\witi{\otimes} w}$ by Lemma~\ref{lem:wrmodind}. 
Hence, $W$ is a summand of 
$\Ind_{\D(T\wr S_w)}^{(G\wr S_w)\times (H\wr S_w)} U^{\witi{\otimes} w}$. 
\end{proof}

\begin{lem}\label{lem:gr}
 We have 
\[
 \Ind_{(G\times H)\wr \D S_w}^{(G\wr S_w) \times (H\wr S_w)} \fr{Gr}(G,b,H)^{\witi{\otimes} w} \simeq \fr{Gr}(G\wr S_w, b^{\otimes w}, H\wr S_w) \oplus Z
\]
for some  $\Delta \Cal S'$-projective module $Z$.
\end{lem}

\begin{proof}
 We claim that 
\begin{equation}\label{eq:gr1}
 \Ind_{(G\times H) \wr \D S_w}^{(G\wr S_w)\times (H\wr S_w)} (\Cal OGb)^{\witi{\otimes} w} \simeq \Cal O(G\wr S_w)b^{\otimes w}
\end{equation}
as $\cl O(G\wr S_w)$-$\cl O(H\wr S_w)$-bimodules.
Since $\{ (1,\s) \mid \s\in S_w \}$ is a set of representatives of left $(G\times H)\wr \D S_w$-cosets in $(G\wr S_w) \times (H\wr S_w)=(G\times H) \wr (S_w\times S_w)$, we 
have the following equality (of $\cl O$-modules):
\begin{equation}\label{eq:gr2}
\begin{split}
 \Ind_{(G\times H) \wr \D S_w}^{(G\wr S_w)\times (H\wr S_w)} (\Cal OGb)^{\witi{\otimes} w} &=   
\O ((G\wr S_w)\times (H\wr S_w)) \otimes_{\cl O((G\times H) \wr \D S_w)} (\cl OGb)^{\witi{\otimes} w} \\
&= \bigoplus_{\s\in S_w} (1,\s) \otimes (\cl OGb)^{\otimes w}. 
\end{split}
\end{equation}
Also (again, as $\cl O$-modules), 
\begin{equation}\label{eq:gr3}
 \cl O(G\wr S_w) b^{\otimes w} = \bigoplus_{\s\in S_w} (\cl OGb)^{\otimes w} \cdot \s
\end{equation}
With the identifications~\eqref{eq:gr2} and~\eqref{eq:gr3} in mind, we see that the map
\[
 (1,\s) \otimes (a_1\otimes \cdots \otimes a_w) \mapsto (a_1\otimes \cdots \otimes a_w) \cdot \s^{-1}, \quad \s\in S_w, \; a_1,\ldots,a_w\in \cl OGb,
\]
defines an $\cl O$-module isomorphism between the two sides of~\eqref{eq:gr1}. It is straightforward (but tedious) to verify that this map is, in fact, an 
isomorphism of $\cl O(G\wr S_w)$-$\cl O(H\wr S_w)$-bimodules, proving the claim. 

By definition of $\fr{Gr}(G,b,H)$, we have 
\[
\Cal OGb = \fr{Gr}(G,b,H) \oplus Y_1\oplus \cdots \oplus Y_m
\]
 for some indecomposable $\Delta \Cal S$-projective $\cl O(G\times H)$-modules $Y_1,\ldots,Y_m$. 
Thus, it is easy to see that $(\cl OGb)^{\witi{\otimes} w}$ is a direct sum of $\fr{Gr}(G,b,H)^{\witi{\otimes} w}$ and summands of the form
\begin{equation}\label{eq:gr3_5}
\Ind_{\prod_{i=0}^m ((G\times H) \wr S_{w_i})}^{(G\times H) \wr S_w} (\fr{Gr}(G,b,H)^{\witi{\otimes} w_0}\otimes Y_1^{\witi{\otimes} w_1}\otimes \cdots \otimes Y_m^{\witi{\otimes} w_m})
\end{equation}
where $w_0+w_1+\cdots+w_m = w$ and $w_0<w$
(cf.\ the proof of Lemma~\ref{lem:lincomb}). 
Therefore, the left-hand side of~\eqref{eq:gr1} is a direct sum of $\Ind_{(G\times H)\wr \D S_w}^{(G\wr S_w)\times (H\wr S_w)} (\fr{Gr}(G,b,H)^{\witi{\otimes} w})$
and modules of the form
\begin{equation}\label{eq:gr4}
  \Ind_{(G\times H)\wr \D S_{w_0}\times \cdots \times (G\times H)\wr \D S_{w_m}}^{(G\wr S_w)\times (H\wr S_w)} (\fr{Gr}(G,b,H)^{\witi{\otimes} w_0}\otimes Y_1^{\witi{\otimes} w_1}
\otimes \cdots \otimes Y_m^{\witi{\otimes} w_m})
\end{equation}
 where not all of $w_1,\ldots, w_m$ are zero. 

Consider a fixed summand~\eqref{eq:gr4}. Let $t>0$ be such that $w_t>0$. By Lemma~\ref{lem:rp1}, 
\[
\begin{split}
 \Ind_{(G\times H)\wr \D S_{w_t} }^{(G\wr S_{w_t})\times (H\wr S_{w_t}) } Y_t^{\witi{\otimes} w_t} & \text{ is } \D \tilde{\Cal S}^{(w_t)} \text{-projective and } \\
 \Ind_{(G\times H)\wr \D S_{w_s} }^{(G\wr S_{w_s})\times (H\wr S_{w_s}) } Y_s^{\witi{\otimes} w_s} & \text{ is } \D (D\wr S_{w_s}) \text{-projective for } s\in [0,w]-\{ t \},
\end{split}
\]
where we write $Y_0 = \fr{Gr}(G,b,H)$. It follows that the summand~\eqref{eq:gr4} is $\D\tilde{\cl S}$-projective; hence, by Lemma~\ref{lem:S'}, it is
$\cl S'$-projective. 

We have proved that the $\cl O((G\wr S_w)\times (H\wr S_w))$-module $\cl O(G\wr S_w) b^{\otimes w}$ is a direct sum of 
$\Ind_{(G\times H)\wr \D S_w}^{(G\wr S_w)\times (H\wr S_w)} \fr{Gr}(G,b,H)^{\witi{\otimes} w}$ and a $\D \cl S'$-projective module. 
Now, by definition, $\fr{Gr}(G\wr S_w, b^{\otimes w}, H\wr S_w)$ is the only indecomposable summand of $\Cal O(G\wr S_w)b^{\otimes w}$ that is not $\Cal S'$-projective. 
The lemma follows from these two facts. 
\end{proof}

\begin{proof}[Proof of Theorem~\ref{thm:wrequiv}]
 Fix a decomposition of each term $C^i$ into indecomposable summands, so that $\fr{Gr}(G,b,H)$ is one of the summands of $C^0$ and all the other summands of the terms of $C$ are $\D\Cal S$-projective. 
Due to~\eqref{eq:der1}, 
this induces a decomposition of each term of $C^{\otimes w}$ (as an $\cl O((G\times H)^{\times w})$-module) into certain summands. The group $S_w$ acts on the set of these summands. 
One of the summands is 
$\fr{Gr}(G,b,H)^{\otimes w}$. It is stabilised by $S_w$ and becomes $\fr{Gr}(G,b,H)^{\witi{\otimes} w}$ as an $\cl O((G\times H)\wr S_w)$-module by~\eqref{eq:eps1}. Write $(C\wr S_w)^i$ for the $i$-th term of $C\wr S_w$. 
Due to Lemma~\ref{lem:gr}, it suffices to show that 
$(C\wr S_w)^i$ is $\D\cl S'$-projective for all $i\ne 0$ and 
that the direct complement of
$\Ind_{(G\times H)\wr \D S_w}^{(G\wr S_w) \times (H\wr S_w)} \fr{Gr}(G,b,H)$ 
in $(C\wr S_w)^0$ is $\D\cl S'$-projective. 

Let $Y=Y_1^{\otimes w_1} \otimes \cdots \otimes Y_m^{\otimes w_m}$ be a summand in the above decomposition of 
$C^{\otimes w}$ (with $w=\sum_i w_i$), where each $Y_j$ is a summand of the term $C^{d_j}$ (for some 
$d_j\in \mZ$) and there exists $t\in [1,w]$ such that $Y_t$ is not isomorphic to $\fr{Gr}(G,b,H)$. Then $Y_t$ is $\D \cl S$-projective.
The $S_w$-stabiliser of this summand $Y$ is $\prod_j S_{w_j}$, so $Y$ is endowed with a structure of a $(\prod_{j} S_{w_j})$-module, denoted by $\tilde{Y}$. 
Then the (direct) sum of elements of the $S_w$-orbit of $Y$ is a summand of the $(\sum_j d_j)$-term of $C^{\otimes w}$ as an 
$\cl O((G\times H)\wr S_w)$-module and is
isomorphic to $\Ind_{(G\times H)\wr (\prod_j S_{w_j})}^{(G\times H) \wr S_w} \tilde{Y}$.
 By the argument above, it is enough to show that the module 
\[
 Z = \Ind_{(G\times H)\wr \D (S_{w_1} \times \cdots \times S_{w_m})}^{(G\wr S_w)\times (H\wr S_w)} \tilde{Y},
\]
is $\D \Cal S'$-projective for each such $Y$.

It follows from Eq.~\eqref{eq:eps3} that the map $(\s_1,\ldots, \s_m) \mapsto \e_{\s} (d_1,\ldots,d_1,\ldots, d_m,\ldots, d_m)$ (where $d_j$ occurs $w_j$ times) 
is a homomorphism from 
$\prod_j S_{w_j}$ to $\mZ/2\mZ$. 
So this map may be written as $(\s_1,\ldots,\s_m)\mapsto \sgn(\s_1)^{s_1} \cdots \sgn(\s_m)^{s_m}$ for some $s_1,\ldots, s_m\in \{0,1\}$ (where $\sgn$ is the sign character). 
Hence, 
\[
 \tilde{Y} \simeq (T_1 \otimes Y_1^{\witi{\otimes} w_1}) \otimes \cdots \otimes (T_m \otimes Y_m^{\witi{\otimes} w_m})
\]
 (as $\Cal O((G\times H)\wr (\prod_j S_{w_j}))$-modules) 
where $T_j$ is the inflation from $\Cal OS_{w_j}$ to $\Cal O((G\times H)\wr S_{w_j})$ of the $1$-dimensional module affording the character $\sgn^{s_j}$. 
By Lemma~\ref{lem:rp1}, the module $\Ind_{(G\times H)\wr \D S_{w_t}}^{(G\wr S_{w_t}) \times (H\wr S_{w_t})} (T_t\otimes Y_t^{\witi{\otimes} w_t})$
is $\D(\tilde{\Cal S}^{(w_t)})$-projective. Further, for each $j\in [1,m]$, the module 
$\Ind_{(G\times H)\wr \D S_{w_j}}^{(G\wr S_{w_j}) \times (H\wr S_{w_j})} (T_j\otimes Y_j^{\witi{\otimes} w_j})$
is $\D (D\wr Q_j)$-projective, where $Q_j$ is a Sylow $p$-subgroup of $S_{w_j}$ (as $Y_j$ is $\D D$-projective). From these two facts and the definition of $\tilde{\cl S}^{(w)}$,
one deduces that   
\[
 Z \simeq \Ind_{\prod_j (G\wr S_{w_j}) \times (H\wr S_{w_j})}^{(G\wr S_w) \times (H\wr S_w)} 
\prod_{j=1}^m \Ind_{(G\times H)\wr \D S_{w_j}}^{(G\wr S_{w_j}) \times (H\wr S_{w_j})} (T_j\otimes Y_j^{\witi{\otimes} w_j})
\]
is $\D (\tilde{\cl S}^{(w)} )$-projective. By Lemma~\ref{lem:S'}, $\tilde{\cl S}^{(w)} \subset \cl S'$, so $Z$ is $\D \cl S'$-projective.
\end{proof}

\begin{cor}\label{cor:wrmarcus}
Let $b_0$ be the principal $p$-block of $S_p\wr S_w$. 
Then (IRC-Bl) holds for the quadruple $(S_p\wr S_w, b_0, P, N_{S_p}(C_p) \wr S_w)$.
\end{cor}

\begin{proof}
Let $c_0$ be the principal block of $S_p$, so that $b_0 = c_0^{\otimes w}$, and $b'_0$ be the principal block of $N_{S_p}(C_p)\wr S_w$. 
Let  $\cl S = \cl S(S_p, C_p, N_{S_p}(C_p))$ (so that $\cl S=\{ \mbf 1 \}$ for $p>3$). 
 By~\cite[Theorem 3.10]{Evseev2010m}, there exists a Rickard tilting complex $C$ of $\Cal O S_p c_0$-$\Cal O N_{S_p}(C_p)$-bimodules that is $\cl S$-tempered with pivot 
$\fr{Gr}(S_p, c_0, N_{S_p}(C_p))$. We may assume that this pivot is located in degree $0$ of $C$ (for we can shift $C$ as necessary). 
By Theorems~\ref{thm:marcus} and~\ref{thm:wrequiv}, $C\wr S_w$ is a Rickard tilting complex of 
$\Cal O (S_p\wr S_w)b_0$-$\Cal O(N_{S_p}(C_p) \wr S_w, b'_0)$-bimodules  
and is $\cl S(S_p\wr S_w, P, N_{S_p}(C_p)\wr S_w)$-tempered with pivot $\fr{Gr}(S_p\wr S_w, b_0, N_{S_p}(C_p)\wr S_w)$. By (the proof of) \cite[Proposition 3.8]{Evseev2010m}, 
the existence of such a 
complex implies that (IRC-Bl) holds for the quadruple $(S_p\wr S_w, b_0, P, N_{S_p}(C_p) \wr S_w)$. 
\end{proof}

\subsection{Preliminary Lemmas}

\begin{defi}[Cf.\ Definition~\ref{def:ircbl}]\label{defi:compliant}
 Let $H$ be a subgroup of a finite group $G$, and let $X_G$ and $X_H$ be subsets of $\Cal C(G)$ and $\cl C(H)$ respectively, both
closed under multiplication by $-1$. Let $D$ be a $p$-subgroup of $H$. 
We say that a map $F\colon X_G \to X_H$ is \emph{$D$-IRC-compliant} if $F$ is signed and 
\[
 F(\chi) \equiv \Proj_D \Res^G_H \chi \dmod \Cal I(H,D,\cl S(G,D,H))
\]
for all $\chi\in X_G$.
\end{defi}

In the case when, in addition, $|G:H|$ is prime to $p$ and $D\in \Syl_p (H)$, we will say that the map $F$ as above is \emph{IRC-compliant} if it is 
$D$-IRC-compliant. (In this situation, $\Proj_D$ is the identity map, and the definition does not depend on the particular choice of $D$.)

For blocks of maximal defect one can replace the property (IRC-Bl) with a slightly simpler one. 
Suppose that $G$ is a finite group, $D\in \Syl_p (G)$, and $N_G(D) \le H\le G$. We write
\[
 \Irr_{p'} (G) = \{ \chi\in \Irr(G) \mid \chi(1) \not\equiv 0 \mod p \}.
\]
The pair $(G,H)$ is said to satisfy the property (IRC-Syl) if there exists an IRC-compliant signed bijection $F\colon \pm\Irr_{p'}(G) \ra \pm\Irr_{p'}(H)$ 
(see~\cite[Section 1]{Evseev2010m}).
Note that (IRC-Syl) holds if and only if (IRC-Bl) holds for each quadruple $(G,b,D,H)$ where $b$ runs through the blocks 
of $G$ with defect group $D$ (see~\cite[Section 2]{Evseev2010m}).

We now prove several technical lemmas concerning (IRC-Syl).

\begin{lem}\label{lem:irccliff}
 Suppose that $N\le H\le G$ are finite groups and $N$ is normal in $G$. Let $\phi\in\Irr(N)$ and  $T=G_{\phi}$, $L=H_{\phi}$ be its inertia groups. 
Assume that $|G:L|$ is prime to $p$.
Let $X_T$ and $X_L$ be subsets of $\pm \Irr(T|\phi)$ and $\pm \Irr(H|\phi)$ respectively, closed under multiplication by $-1$. Let
\[
 X_G  =  \{ \Ind_T^G \chi \mid \chi\in X_T \} \quad \text{and} \quad
 X_H  =  \{ \Ind_L^H \th \mid \th\in X_L \}.
\]
Suppose that $F\colon X_T \to X_L$ is an IRC-compliant signed bijection. 
 Then the map $\tilde{F}\colon X_G \to X_H$ defined by $\tilde{F}(\Ind_T^G \chi)=\Ind_L^H F(\chi)$ for $\chi\in X_T$ is also an IRC-compliant signed bijection. 
\end{lem}

\begin{proof}
By Clifford theory, $\tilde{F}$ is a well-defined bijection. Let $D\in \Syl_p (L)$. 
By~\cite[Theorem 1.4]{Evseev2010m}, the fact that $F$ is IRC-compliant implies that
\[
 \chi \equiv \Ind_L^T F(\chi) \dmod \Cal I(T,D,\Cal S(T,D,L)) \qquad \text{for all } \chi\in X_T.
\]
 Applying $\Ind_T^G$ to both sides of this congruence and using the definition of $\tilde{F}$, we obtain
\[
 \eta \equiv \Ind_H^G \tilde{F}(\eta) \dmod \Ind_T^G (\Cal I(T,P,\Cal S(T,D,L)))  \qquad \text{for all } \eta\in X_G.
\]
 Since $H=T\cap L$, it follows from the definitions that $\Cal S(T,D,L)\subset \Cal S(G,D,H)$ and hence that 
$\Ind_T^G (\Cal I(T,D,\Cal S(T,D,L))) \subset \Cal I(G,D,\Cal S(G,D,H))$. Therefore,
\[
 \eta \equiv \Ind_T^G \tilde{F}(\eta) \dmod \Cal I(G,D,\Cal S(G,D,H)) \qquad \text{for all } \eta \in X_G.
\]
 The lemma follows from this after another application of~\cite[Theorem~1.4]{Evseev2010m}.
\end{proof}

\begin{lem}\label{lem:normirc}
Let $N$ be a normal subgroup of a finite group $G$ and $\bar{\phantom{g}}\colon G\to G/N = \bar{G}$ be the projection map.
Let $H$ be a subgroup of $G$ containing $N$ such that $|G:H|$ is prime to $p$. 
 Suppose that $\phi\in\Irr(N)$ is a character that
extends to $\hat{\phi}\in \Irr(G)$, and write $\hat{\phi}_H = \Res^G_H \hat{\phi}$. 
Let $X_{\bar{G}}$ and $X_{\bar{H}}$ be subsets of $\pm \Irr(\bar{G})$ and $\pm\Irr(\bar{H})$ respectively that are closed under multiplication by $-1$,
and suppose that there is an IRC-compliant signed bijection $\bar{F}\colon X_{\bar{G}} \to X_{\bar{H}}$. 
Let $X_G = \{\hat{\phi}\Inf_{\bar G}^G \chi \mid \chi\in X_G \}$ and $X_H= \{\hat{\phi}_H \Inf_{\bar H}^H \th \mid \th \in X_H \}$. 
Then the bijection $F\colon X_G \to X_H$ defined by $F(\hat{\phi} \Inf_{\bar G}^G\chi) = \hat{\phi}_H \Inf_{\bar H}^H \bar{F}(\chi)$ for $\chi\in X_{\bar G}$
is also IRC-compliant. 
\end{lem}

\begin{proof}
First, we note that the map $\chi \mapsto \hat{\phi} \Inf_{\bar G}^G \chi$ is a bijection from $X_{\bar G}$ onto $X_G$ by~\cite[Theorem 6.16]{IsaacsBook}, 
and the same holds for $H$. So the map $F$ is well-defined and is a signed bijection. Let $D\in \Syl_p (H)$, so that $D\cap N \in \Syl_p (N)$. 
 Let $\g = \hat{\phi} \Inf_{\bar G}^G \chi\in X_G$, where $\chi\in X_{\bar G}$. Then 
\begin{equation}\label{eq:ni1}
 F(\g) - \Res^G_H \g = \hat{\phi}_H \cdot \Inf_{\bar H}^H (F(\chi) - \Res^{\bar G}_{\bar H} \chi),
\end{equation}
and $F(\chi)-\Res^{\bar G}_{\bar H} (\chi)$ is, by the hypothesis, an element of the $\mZ$-span of virtual characters of the form $\Ind_L^G \xi$ where $\xi\in \Cal \Irr(L)$ and 
 $L$ is a subgroup of $\bar H$ satisfying
 $L\cap \bar{D}\in \Syl_p(L)$ and $L\cap \bar{D} \le \ls{x}{\bar D}$ for some $x\in \bar G - \bar H$. 
In the remainder of the proof $L$ is allowed to run through the set $\cl L$ of subgroups of $\bar H$ with these properties, and we denote by $M$ the preimage 
of $L$ under the projection $\bar{\phantom{g}}$. 

From~\eqref{eq:ni1} we deduce that $F(\g)-\Res^{G}_{H} \g$ belongs to the $\mZ$-span of virtual characters of the form 
\begin{equation}\label{eq:ni2}
\hat{\phi}_H\cdot (\Inf_{\bar H}^{H} \Ind_L^{\bar H} \xi)  =  \Ind_{M}^H ( (\Res^G_M \hat{\phi}) \cdot (\Inf_{L}^M \xi) )
\end{equation}
where $L\in \cl L$, $\xi\in \Irr(L)$, and $M$ is as above. 
It suffices to show that in this situation 
 $M\cap D \in \Syl_p (M)$ and $M\cap D \le \ls{g}D$ for some $g\in G-H$; for then the image under the map 
$\Ind_M^H$ of any virtual character of $M$ belongs to $\Cal I(H,D,\Cal S(G,D,H))$, in particular, the right-hand side of~\eqref{eq:ni2} does. 

The condition that $L\cap \bar D\in \Syl_p (L)$ means that $p$ does not divide $|L:L\cap \bar D|=|M:M\cap DN|$. 
Also, since $N\le M$, 
\[
 |M\cap DN: M\cap D| = |(M\cap D)N : M\cap D| = |N: N\cap D|
\]
is prime to $p$. 
It follows that $|M:M\cap D|$ is prime to $p$, whence $M\cap D\in \Syl_p (M)$. 

Now the fact that $L\cap \bar D\le \ls{x}{\bar D}$ (with $x\in \bar G - \bar H$) implies that $M\cap D \le \ls{g}{(DN)}$, where $g$ is any element of $G$ such that $\bar{g}=x$.
That is, $\ls{g^{-1}}{(M\cap D)} \le DN$.
But since $D\in \Syl_p (DN)$, there exists $z\in DN$ such that $\ls{g^{-1}}{(M\cap D)} \le \ls{z}D$; moreover, we may assume that $z\in N$. Hence, $M\cap D \le \ls{gz}D$. And
$gz\notin H$ because $\bar{g} = x\notin \bar{H}$ and $z\in N$. Thus, $M$ satisfies all the required conditions. 
\end{proof}

\begin{lem}\label{lem:resI}
Let $H\le L$ be finite groups such that $|L:H|$ is prime to $p$. Let $D\in \Syl_p (H)$. Suppose that  
$\Cal S$ is a downward closed set of subgroups of $D$ such that $\Cal S\supset \Cal S(L,D,H)$.
 Then
\[
 \Res^{L}_H (\Cal I (L, D, \Cal S)) \subset \Cal I(H,D,\Cal S).
\]
\end{lem}

\begin{proof} Let $A$ be a subgroup of $L$ such that $A\cap D\in \Syl_p (A)$ and $A\cap D\in \Cal S$. By Definition~\ref{def:I}, it suffices to show that
$\Res^L_H \Ind^L_A \xi\in \Cal I(H,D,\Cal S)$ for all such $A$ and all $\xi\in \Cal C(A)$. By the Mackey formula,
\[
 \Res^L_H \Ind_A^L \xi = \Ind_{A\cap H}^{H} \Res^A_{A\cap H} \xi + \sum_{g\in T} \Ind_{\lsb{g}{A}\cap H}^H \Res^{\lsb{g}{A}}_{\lsb{g}{A}\cap H} \ls{g}{\xi}
\]
where $T$ is a set of representatives of double $H$-$A$-cosets in $L-HA$. We will show that each summand on the right-hand side belongs to $\Cal I(H,D,\Cal S)$. 
We have $(A\cap H) \cap D=A\cap D\in \Syl_p (A\cap H) \cap \cl S$, so $\Ind_{A\cap H}^H \Res^A_{A\cap H} \xi \in \Cal I(H, D, \Cal S)$.

Now let $g\in T$. Let $E\in \Syl_p (\lsb{g}{A}\cap H)$. Then $\ls{h}{E} \le D$ for some $h\in H$ and, replacing $g$ with $hg$ (and $E$ with $\ls{h}E$), 
we may assume that $E\le D$. 
Also, $E\le \lsb{g}{A}$, whence $\ls{g^{-1}}E \le A$. Since $D$ contains a Sylow $p$-subgroup of $A$, 
we have 
$E\le \ls{ga}{D}$ for some $a\in A$ (by the Sylow theorems). Since $ga\notin H$, we have $E\in \Cal S(L,D,H)\subset \Cal S$, so
\[
 \Ind_{\lsb{g}{A}\cap H}^H \Res^{\lsb{g}A}_{\lsb{g}A \cap H} \ls{g}\xi \in \Cal I(H,D,\Cal S). \qedhere
\]  
\end{proof}

\begin{lem}\label{lem:comp}
 Let $H\le L\le G$ be finite groups such that $|G:H|$ is prime to $p$. Suppose that \emph{(IRC-Syl)} holds for the pairs
$(G,L)$ and $(L,H)$. Then \emph{(IRC-Syl)} holds for $(G,H)$.
\end{lem}

\begin{proof}
Let $D\in \Syl_p (H)$. 
Let $F_1\colon \pm \Irr_{p'}(G)\to \pm \Irr_{p'}(L)$ and $F_2\colon \pm \Irr_{p'}(L)\to \pm \Irr_{p'}(H)$ be IRC-compliant signed bijections. 
We will show that $F=F_2 F_1$ is also IRC-compliant. Let $\chi\in \pm \Irr_{p'} (G)$. Then
\[
 F(\chi) - \Res^G_H \chi = (F_2(F_1(\chi))-\Res^L_H(F_1(\chi)) ) + \Res^L_H (F_1(\chi)-\Res^G_L (\chi)).
\]
Since $F_2$ is IRC-compliant, we have 
\[
F_2 (F_1 (\chi)) - \Res^L_H (F_1(\chi)) \in \Cal I(H,D,\Cal S(L,D,H)) \subset \Cal I(H,D,\cl S(G,D,H)).
\]
(The containment is clear 
from the definitions.) Since $F_1$ is IRC-compliant,  we have
\[
\begin{split}
 \Res^L_H (F_1(\chi) - \Res^G_L (\chi)) \in \Res^L_H (\Cal I(L,D,\Cal S(G,D,L) ) ) & \subset \Res^L_H (\cl I(L,D, \cl S(G,D,H) )\\ 
& \subset \Cal I(H,D,\Cal S(G,D,H) ),
\end{split}
\]
where the first containment holds because $\cl S(G,D,L) \subset \cl S(G,D,H)$ and the second one is due to Lemma~\ref{lem:resI}.
Hence, $F$ is indeed IRC-compliant. 
\end{proof}

\begin{lem}\label{lem:ircprod}
 Let $G_1,\ldots, G_n$ be finite groups. Suppose that for each $i\in [1,n]$:
\begin{enumerate}[(i)]
 \item $D_i \in \Syl_p(G_i)$ and $H_i$ is a subgroup of $G_i$ containing $N_{G_i} (D_i)$;
 \item there are subsets $X_i$ and $Y_i$ of $\Irr(G_i)$ and $\Irr(H_i)$ respectively and an IRC-compliant signed bijection 
$F_i \colon \pm X_i \to \pm Y_i$.
\end{enumerate}
Then the signed bijection $F\colon \pm \prod_i X_i \to \pm \prod_i Y_i$ defined by
\[
 F(\chi_1 \times \cdots \times \chi_n) = F(\chi_1) \times \cdots \times F(\chi_n), \quad \chi_1,\ldots, \chi_n\in \pm \Irr(X_i),
\]
is IRC-compliant.
\end{lem}

\begin{proof}
Without loss of generality, $n=2$. 
First, observe that the lemma holds in the special case when $H_2 = G_2$:
this can be seen by 
applying Lemma~\ref{lem:normirc} with $N= G_1$ and $\phi$ running through $X_1$. 
Thus, the map defined by $\chi_1 \times \chi_2\mapsto \chi_1\times F(\chi_2)$ is an IRC-compliant signed bijection from $\pm (X_1 \times X_2)$ onto $\pm (X_1\times Y_2)$. 
We have a similar bijection from $\pm (X_1\times Y_2)$ onto $\pm (Y_1\times Y_2)$. The composition of these two maps is clearly equal to $F$ and is IRC-compliant by the proof of
Lemma~\ref{lem:comp}. 
\end{proof}

\subsection{Proof of Theorem~\ref{thm:symirc}}\label{sub:proofirc}

Recall that $Q\in \Syl_p (S_w)$ and $P=C_p\wr Q$.
The following result will be used in the proof of Theorem~\ref{thm:symirc} 
and relies on the inductive hypothesis that will be available in that proof. 

\begin{prop}\label{prop:cliff}
Assume that the statement of Theorem~\ref{thm:symirc} is true for all blocks of all groups $S_m$ such that $m\le w$. 
Then (IRC-Syl) holds for the pair $(N_{S_p}(C_p)\wr S_w, N_{S_p}(C_p)\wr N_{S_w}(Q) )$.
\end{prop}

\begin{proof}
Let $G=N_{S_p}(C_p)\wr S_w$.
For a fixed character $\phi\in \Irr(N_{S_p}(C_p)^{\times w})$, consider the set $\Irr_{p'}(G|\phi) = \Irr_{p'} (G) \cap \Irr(G|\phi)$.
Let $T=\Stab_{S_w}(\phi)$. 
If $|S_w:T|$ is divisible by $p$, then $\Irr_{p'}(G|\phi)=\varnothing$. 
So we assume that $|S_w:T|$ is prime to $p$ and hence that $Q\le T$ (possibly, after replacing $\phi$ with a 
$G$-conjugate). 
Without loss of generality, 
\[
 \phi = \phi_1^{\times w_1} \times \cdots \times \phi_r^{\times w_r}
\]
where $\phi_1,\ldots,\phi_r\in \Irr(S_p)$ are distinct. If $p$ divides the degree of some $\phi_i$, then $\Irr_{p'}(G|\phi)=\varnothing$, 
so we assume that $\phi_1,\ldots, \phi_r\in \Irr_{p'}(S_p)$. (We denote by $\cl A$ a set of representatives of $G$-orbits on the set of 
characters $\phi$ satisfying all of the above assumptions.)

We have $T=S_{w_1}\times \cdots \times S_{w_r}$
and $Q = Q_1 \times \cdots \times Q_r$ where $Q_l=Q\cap S_{w_i}$ for all $i$ (so that $Q_i\in \Syl_p (S_{w_i})$). 
For each $i\in [1,r]$, by the hypothesis applied to $S_{w_i}$, there 
exists an IRC-compliant signed bijection $\bar{F}_i\colon \pm \Irr_{p'}(S_{w_i}) \to \Irr_{p'}(N_{S_{w_i}}(Q_i) )$. 
 Since $\phi_i^{\times w_i}$ extends to the character $\phi^{\witi{\times} w_i}$  of $N_{S_p}(C_p)\wr S_{w_i}$, Lemma~\ref{lem:normirc} yields an IRC-compliant signed bijection 
$F_i \colon \pm \Irr_{p'} (N_{S_p}(C_p) \wr S_{w_i} \mid \phi_i^{\times w_i} ) \to \pm \Irr_{p'} (N_{S_p}(C_p)\wr N_{S_{w_i}} (Q_i) \mid \phi_i^{\times w_i})$.
 Hence, due to Lemma~\ref{lem:ircprod}, there is an IRC-compliant signed bijection $F\colon \pm \Irr_{p'}(N_{S_p}(C_p) \wr T |\phi)\to \pm \Irr_{p'} (N_{S_p}(C_p) \wr N_T(Q))$.
Finally, by Lemma~\ref{lem:irccliff}, there exists an IRC-compliant signed bijection 
$\tilde{F}\colon \pm \Irr_{p'}(N_{S_p}(C_p) \wr S_w \mid \phi) \to \pm \Irr_{p'}(N_{S_p}(C_p) \wr N_{S_w} (Q) \mid \phi)$. 
Combining such bijections for representatives of all $\phi\in \cl A$, we obtain the result. 
\end{proof}


\begin{lem}\label{lem:resind}
Let $G$ be a finite group. Let $L\le G$ and $N$ be a normal subgroup of $G$.
 Let $\phi\in \Irr(N)$ and $\psi\in \Irr(L\cap N)$. Suppose that $G=LN$ and that the inertia subgroups $L_{\psi}$ and
$G_{\phi}$ satisfy $G_{\phi}=L_{\psi} N$. 
\begin{enumerate}[(i)]
 \item If $\Res^N_{L\cap N} \phi = \psi$, then the map $\Res^G_L$ restricts to a bijection from $\Irr(G|\phi)$ onto $\Irr(L|\psi)$.
\item If $\Ind_{L\cap N}^N \psi = \phi$, then the map $\Ind_L^G$ restricts to a bijection from $\Irr(L|\psi)$ onto $\Irr(G|\phi)$.
\end{enumerate}
\end{lem}

\begin{proof}
 (i) First, we prove the result in the case when $G_{\phi}=G$. 
We will use the well-known theory of characters of twisted group algebras, referring to~\cite{Dade1994} for terminology and some results. 
Suppose that $\chi\in\Irr(\fr G)$ for a finite group $\fr G$. We write $e_{\chi}$ for the corresponding primitive central idempotent of the group algebra $K\fr G$ and
$M_n (K)$ for the algebra of $n\times n$ matrices over $K$.
Furthermore, we will use the following elementary fact: if $\th\in\Irr(\fr H)$ for some $\fr H\le \fr G$, we have  $e_{\th}e_{\chi} = e_{\chi}$ if and only if $\Res^{\fr G}_{\fr H} \chi$ is a multiple of $\th$. 

In particular, by the hypothesis of (i), we have $e_{\psi} e_{\phi} = e_{\phi}$. Moreover, any  representation affording $\phi$ maps $K(L\cap N)$  onto $M_{\phi(1)}(K)$; that is,
$KNe_{\phi} = K(L\cap N)e_{\phi}$. Following~\cite[Proposition 11.15]{Dade1994}, let $A$ be the $(G/N)$-graded algebra $C_{KG}(KN)e_{\phi}$ where the 
$gN$-component $A_{gN}$ is the intersection of $A$ with the $K$-span of $gN$ (for each $g\in G$). By~\cite[Proposition 11.20]{Dade1994}, $A$ is a totally split twisted group algebra for $G/N$; moreover, by~\cite[Statement (12.15)]{Dade1994}, $\Cal E_{\phi}=\{e_{\chi} \mid \chi\in \Irr(G|\phi) \}$ is precisely the set of primitive central idempotents of $A$. Similarly,
 $B = C_{KL}(K(L\cap N))e_{\psi}$ is a totally split twisted group algebra for $L/(L\cap N)$,
 and $\Cal E_{\psi}=\{e_{\th} \mid \th\in \Irr(L|\psi) \}$ is precisely the set of all primitive central idempotents of 
$B$. Let $W = G/N = L/(L\cap N)$, where the last two groups are identified in the usual way. Then both $A$ and $B$ are twisted group algebras for $W$. Since $G_{\phi}=G$, the idempotent $e_{\phi}$ centralises $KG$ and, in particular, $B$. 
Each $u\in B$ centralises $K(L\cap N)e_{\phi} = KNe_{\phi}$; hence, $ue_{\phi}$ centralises $KN$, so $ue_{\phi}\in A$. Thus we have a homomorphism 
$\iota\colon B\to A$ defined by $\iota(u) = ue_{\phi}$. Moreover, if $g\in L$ and $u_g$ is a homogeneous unit of $B$ of degree $g(L\cap N)$, then
$u_g e_{\phi}$ is a homogeneous unit of $A$ of degree $gN$ because $(u_g e_{\phi})(u_g^{-1} e_{\phi}) = e_{\psi} e_{\phi} = e_{\phi}$. So the image of $\iota$ contains a homogeneous unit in each degree, whence $\iota$ is surjective. Since $\dim_K A = \dim_K B=|W|$, it follows that $\iota$ is an isomorphism.
In particular, 
the map $e_{\th} \mapsto e_{\th} e_{\phi}$, which is the restriction of $\iota$ to 
$\cl E_{\psi}$, 
 is a bijection from
$\Cal E_{\psi}$ onto $\Cal E_{\phi}$. 
By the fact quoted at the end of the previous paragraph, for each $\chi\in \Irr(G|\phi)$ there 
exists $m_{\chi}\in \mN$ such that $\Res^G_{L} \chi = m_{\chi} \th$, where $e_{\th} = \iota^{-1}(e_{\chi})$. For every $\th\in \Irr(L|\psi)$, we have $\Res^L_{L\cap N} \th = n_{\th} \psi$, where
$n_{\chi} = \th(1)/\psi(1)$. Similarly, if $\chi\in \Irr(G|\phi)$ and $e_{\th} = \iota^{-1}(e_{\chi})$,
then 
$\Res^G_N \chi = n_{\chi} \phi$, where $n_{\chi} = \chi(1)/\phi(1)= m_{\chi} \th(1)/\psi(1) = 
m_{\chi} n_{\th}$ (as $\phi(1)=\psi(1)$). Since $G=G_{\phi}$ and $L=L_{\psi}$, we have
\[
\begin{split}
 \sum_{\th\in\Irr(L|\psi)} n_{\th}^2 &= 
\sum_{\th\in\Irr(L|\psi)} \lan \Res^L_{L\cap N} \th, \psi \ran^2 =
\lan \Ind_{L\cap N}^L \psi, \Ind_{L\cap N}^L \psi \ran =
\lan \psi, \Res^L_{L\cap N} \Ind_{L\cap N}^L \psi \ran \\
&=\lan \psi, |L:(L\cap N)| \psi \ran = |L:(L\cap N)| = |G:N| = 
\lan \Ind_N^G \phi, \Ind_N^G \phi \ran  
= \sum_{\chi} n_{\chi}^2.
\end{split}
\]
Due to the identity $n_{\chi} = m_{\chi} n_{\th}$, we deduce that $m_{\chi}=1$ for all 
$\chi\in \Irr(G|\phi)$. So the conclusion of part (i) holds. 

Let us now deduce part (i) of the lemma in the general case. By Clifford theory, induction gives a bijection from $\Irr(G_{\phi} |\phi)$ onto $\Irr(G|\phi)$, and a similar statement holds for $L$. 
Since $\psi= \Res^G_N \phi$, we have $G_{\phi} \cap L \le L_{\psi}$. By the hypothesis, $L_{\psi} \le G_{\phi}$, so $G_{\phi}\cap L = L_{\psi}$.
Thus, by the Mackey formula, if $\chi\in\Irr(G|\phi)$, then $\Res^G_L \Ind_{G_{\phi}}^G \chi = \Ind_{L_{\psi}}^{L} \Res^{G_{\phi}}_{L_{\psi}} \chi$ (as $LG_{\phi}\ge LN=G$). 
This identity means that the conclusion of (i) holds for $G$ and $L$ provided it holds for $G_{\phi}$ and $L_{\psi}$.

(ii)  Using Clifford theory (as in the preceding paragraph), one can easily see that it suffices to prove the result when $G=G_{\phi}$. 
Assuming this, consider $\th,\th'\in \Irr(L|\psi)$. We claim that 
\begin{equation}\label{eq:resind1}
 \lan \Ind_L^G \th, \Ind_L^G \th' \ran = \d_{\th\th'}.
\end{equation}
Suppose that~\eqref{eq:resind1} is false. Then, by Frobenius reciprocity and the Mackey formula, there exists
$g\in N - (L\cap N)$ such that $\lan \Res^{\ls{g}L}_{\ls{g}L\cap L} \ls{g}{\th}, \Res^{L}_{\ls{g}L\cap L} \th' \ran >0$. Then $\Res^{\ls{g}L}_{\ls{g}L\cap L \cap N}\ls{g}\th$ and 
$\Res^L_{\ls{g}L\cap L\cap N} \th$ have a common irreducible constituent. 
However, $\psi$ is the only irreducible constituent of $\Res^L_{L\cap N} \th$, and $\ls{g}\psi$ is the only irreducible constituent of 
$\Res^{\ls{g}L}_{\ls{g}L\cap N} \ls{g}\th$, so we have
$\lan \Res^{\ls{g}L\cap N}_{\ls{g}L\cap L\cap N} \ls{g} \psi, \Res^{L\cap N}_{\ls{g}L\cap L\cap N} \psi \ran >0$. Since $g\notin L\cap N$, 
it follows by the Mackey formula that
$\lan \psi, \Res^L_{L\cap N} \Ind_{L\cap N}^L \psi \ran >1$, whence $\lan \phi, \phi \ran = \lan \Ind_{L\cap N}^N \psi, \Ind_{L\cap N}^N \psi \ran >1$, a contradiction.   

Having proved~\eqref{eq:resind1}, we see immediately that $\Ind_{L}^G$ maps $\Irr(L|\psi)$ into $\Irr(G|\phi)$ and that this map is injective. To prove that it is also surjective, suppose that $\chi\in \Irr(G|\phi)$. Then $\psi$ is a constituent of $\Res^G_{L\cap N} \chi$, so there exists a constituent $\th$ of $\Res^G_L \chi$ such that $\th\in \Irr(L|\psi)$.
Then $\chi$ is a constituent of $\Ind_L^G \th$. However, by~\eqref{eq:resind1}, $\Ind_L^G \th$ is irreducible, whence $\chi = \Ind_L^G \th$. This completes the proof. 
\end{proof}

\begin{lem}\label{lem:wrnorm}
 Let $G$ be a subgroup of $S_w$ that acts transitively on $[1,w]$. Let $B$ be a normal subgroup of a group $A$. Then
\[
 N_{A \wr G} (B\wr G) \cap A^{\times w} = B^{\times w} \cdot \D A,
\]
where $\D A = \{ (x,\ldots,x) \in A^{\times w} \mid x\in A \}$. 
\end{lem}

\begin{proof}
 Since $\D A$ centralises $G$, we have $B^{\times w} \cdot \D A \le N_{A \wr G} (B\wr G) \cap A^{\times w}$. Conversely, suppose that $(x_1,\ldots,x_w) \in A^{\times w}$ 
normalises $B\wr G$. Then for each $\s\in G$ we have
\[
 B\wr G \ni (x_1,\ldots,x_w) \cdot \s \cdot (x_1^{-1}, \ldots, x_w^{-1}) = (x_1 x_{\s^{-1}(1)}^{-1}, \ldots, x_w x_{\s^{-1} (w)}^{-1}; \s).
\]
Since $G$ is transitive on $[1,w]$, we deduce that $x_i x_j^{-1}\in B$ for all $i,j\in [1,w]$. Therefore, $(x_1,\ldots,x_w) \in B^{\times w} \cdot \D A$. 
\end{proof}

\begin{prop}\label{prop:inn}
Let $G= N_{S_p} (C_p) \wr N_{S_w} (Q)\le S_p\wr S_w$ and $H=N_G (P)$.
Then the pair $(G,H)$ satisfies the property (IRC-Syl).
\end{prop}

\begin{proof}
 Let $M=N_{S_p} (C_p)^{\times w}$ be the base subgroup of $G$ and $U=M\cap H$. 
Let $O_1,\ldots, O_r$ be the orbits of the action of $Q$ on $[1,w]$. It follows from Clifford theory that if $\chi\in \Irr_{p'}(G)$, then 
$\chi\in \Irr_{p'}(G |\phi)$ for some $\phi\in \Irr_{p'}(M)$ such that $|G:G_{\phi}|$ is prime to $p$, whence $Q\le \Stab_{S_w} (\phi)$. 
Then $\phi$ is of the form $\phi = \prod_{i=1}^r \a_i^{\times |O_i|}$ for some $\a_i\in \Irr(N_{S_p} (C_p))$ 
(that is, the factors of $\phi$ corresponding to any two points of $[1,w]$ lying in the same $Q$-orbit must be the same). 

Let $\cl A$ be the set  of characters $\phi$ of this form and $\cl B$ be a set of representatives of $G$-orbits on $\cl A$, so that 
$\Irr_{p'} (G)=\sqcup_{\phi\in \cl B} \Irr_{p'} (G|\phi)$. 
The proof will proceed as follows. For each $\phi\in \cl A$ we will define a character $\phi'\in \Irr(U)$, and we will show that
\begin{equation}\label{eq:inn1}
 \Irr_{p'} (H) = \bigsqcup_{\phi\in \cl B} \Irr_{p'} (H|\phi').
\end{equation}
For each $\phi\in \cl B$ we will construct an IRC-compliant signed bijection $F_{\phi} \colon \pm \Irr_{p'} (G|\phi) \to \pm \Irr_{p'} (H|\phi')$. This will suffice, for by 
combining such bijections for all $\phi\in \cl B$ one obtains an IRC-compliant signed bijection from $\pm \Irr_{p'} (G)$ onto $\pm \Irr_{p'}(H)$.

For each $i\in [1,r]$ let 
$M_i$ be the subgroup of $M$ that is the direct product of the $|O_i|$ factors $N_{S_p} (C_p)$ corresponding to the orbit $O_i$, so that $M=\prod_{i=1}^r M_i$.
 Denoting by $Q_i$ the subgroup of $Q$ consisting of the elements that fix $[1,w]-O_i$ pointwise, we see that 
$Q = \prod_{i=1}^r Q_i$.
Since $H= N_G (C_p\wr Q)$, it is clear that $U = \prod_{i=1}^r U_i$ where  $U_i = M_i \cap H$. 

Fix $i\in [1,r]$. 
For each subgroup $A$ of $N_{S_p} (C_p) = C_p\rtimes C_{p-1}$, write 
\[
 \D A = \{ (x,\ldots,x) \in M_i \mid x\in A \}.
\]
By Lemma~\ref{lem:wrnorm}, $U_i = C_p^{\times |O_i|} \cdot \D C_{p-1}$ (this product is semidirect). The subgroup $\D (C_p \rtimes C_{p-1})$ of $U_i$ has a normal complement,
namely the subgroup 
\[
 V_i = \{ (x_1,\ldots,x_{|O_i|}) \in C_p^{\times |O_i|} \mid x_1 \cdots x_{|O_i|} =1 \}.
\]
Let $\b_i$ be the inflation from $U_i/V_i$ to $U_i$ of the character $\a_i\in \Irr(C_p\rtimes C_{p-1})$. Define $\phi' =  \prod_{i=1}^r \b_i \in \Irr(U)$. 
Since $Q_i$ acts transitively on $O_i$, it is easy to see that $V_i$ is contained in the derived subgroup of $C_p\wr Q_i$, and hence $V_i \le [P,P]$. Since $P$ is normal in $H$, the kernel of every
$\chi\in \Irr_{p'}(H)$ must contain $[P,P]\ge \prod_{i=1}^r V_i$, so $\chi\in \Irr_{p'} (H|\phi')$ for some $\phi\in \cl A$. 
Note that $H/U \simeq G/M$ and both these groups may be identified with $N_{S_w} (Q)$ (indeed, the subgroup $N_{S_w} (Q)$ of $S_w$ normalises $P$). 
It is clear that the map $\phi \mapsto \phi'$ is an $N_{S_w} (Q)$-equivariant bijection from $\cl A$ onto the set of irreducible characters of $U$ with kernel containing $\prod_i V_i$. 
Therefore,~\eqref{eq:inn1} holds. 

Now fix $\phi = \prod_{i=1}^r \a_i^{\times |O_i|} \in \cl B$: it remains to prove the existence of an IRC-compliant signed bijection $F_{\phi}$ as above. We begin by showing this in 
two special cases, from which we will then deduce the general result. 
Note that $\Irr(C_p\rtimes C_{p-1})$ consists of $p-1$ linear characters with kernels containing $C_p$ and one character $\g$ of degree $p-1$ with kernel not containing $C_p$.

First, assume that the kernels of $\a_1,\ldots,\a_r\in \Irr(N_{S_p} (C_p))$ all contain $C_p$. 
For each $i\in [1,r]$,
\[
 \Res^{M_i}_{U_i} \a_i^{\times |O_i|} =\Inf_{U_i/V_i}^{U_i} \a_i^{|O_i|} = \Inf_{U_i/V_i}^{U_i} \a_i = \b_i,
\]
where the second equality holds because $|O_i| \equiv 1 \!\! \mod p-1$ (as $|O_i|$ is a power of $p$). 
Hence, $\phi' = \Res^M_U \phi$, and it follows by Lemma~\ref{lem:resind}(i) 
that $\Res^G_H$ restricts to a bijection between
$\pm \Irr_{p'}(G|\phi)$ and $\pm \Irr_{p'}(H|\phi')$. 
This bijection is clearly IRC-compliant, and so may be taken as $F_{\phi}$. 

Secondly, assume that $\a_1=\cdots=\a_r=\g$. 
Note that $\g = \Ind_{C_p}^{N_{S_p} (C_p)} \th$ where $\th$ is any non-trivial linear character of $C_p$. For a fixed $i$, writing $t=|O_i|$,  we have
\[
\Ind_{C_p^{\times t}}^{C_p^{\times t} \cdot \D C_{p-1}} \th^{\times t} = \Ind_{C_p^{\times t}}^{C_p^{\times t} \cdot \D C_{p-1}} (\Inf_{C_p^{\times t}/V_i}^{C_p^{\times t}} \th) 
 = \Inf_{\D (C_p \rtimes C_{p-1})}^{C_p^{\times t} \cdot \D C_{p-1}} (\Ind_{\D C_p}^{\D (C_p\times C_{p-1})} \th) = \b_i,
\]
where $C_p^{\times t}/V_i$ is identified with $C_p$ in the obvious way. Hence,
\[
 \a_i^{\times t} = \g^{\times t} = \Ind_{C_p^{\times t}}^{(C_p\rtimes C_{p-1})^{\times t}} \th^{\times t} = \Ind_{C_p^{\times t} \cdot \D C_{p-1}}^{(C_p\rtimes C_{p-1})^{\times t}} \b_i.
\]
It follows that $\phi = \Ind_U^M \phi'$. Hence, by Lemma~\ref{lem:resind}(ii), the
map $\Ind_H^G$ restricts to a bijection from $\pm \Irr_{p'}(H|\phi')$ onto $\pm \Irr_{p'}(G | \phi)$.
By~\cite[Theorem 1.4]{Evseev2010m}, the inverse of this bijection 
 is IRC-compliant and therefore may be taken as $F_{\phi}$.

Finally, consider the case of arbitrary $\phi\in \cl A$. 
Without loss of generality, the kernels of $\a_1,\ldots, \a_l$ contain $C_p$ and $\a_{l+1}= \cdots = \a_r = \g$.
Let $\Om_1$ be the union of the subsets of $S_{pw}$ associated with the elements of $\cup_{i=1}^l O_i$ of $S_w$ 
(each element $u\in [1,w]$ is associated with the subset $[p(u-1)+1, pu]$ of $[1,pw]$). Let $\Om_2 = [1,pw] - \Om_1$. 
 For $j=1,2$, let $G_j$ be the group consisting of those elements of $G$ that fix every point in $[1,pw]-\Om_j$; the subgroups $H_1$ and $H_2$ of $H$ are defined similarly. 
The decomposition $[1,pw] = \Om_1 \sqcup \Om_2$ leads to the factorisations $\phi = \phi_1 \times \phi_2$ and $\phi' = \phi'_1\times \phi'_2$. By the facts proved in the 
previous two paragraphs, there exist IRC-compliant signed bijections $F_1\colon \pm \Irr_{p'}(G_1|\phi_1) \to \pm \Irr_{p'}(H_1|\phi'_1)$ and 
$F_2 \colon \pm \Irr_{p'} (G_2 |\phi_2) \to \pm \Irr_{p'} (H_2 |\phi'_2)$. 
Therefore, by Lemma~\ref{lem:ircprod}, there is an IRC-compliant signed bijection 
$F\colon \pm \Irr_{p'} (G_1\times G_2 |\phi) \to \pm \Irr_{p'} (H_1\times H_2 |\phi')$

Now $G_{\phi} \le G_1 \times G_2$, so by Clifford theory
the map $\Ind_{G_1\times G_2}^{G}$ restricts to a bijection from $\Irr_{p'}(G_1\times G_2 |\phi)$ onto $\Irr_{p'}(G|\phi)$, and the corresponding statement holds for $H$.
Therefore, the map $F_{\phi}\colon \pm \Irr_{p'} (G|\phi) \to \pm \Irr_{p'}(H|\phi')$, defined by
\begin{equation}\label{eq:inn2}
 F_{\phi}(\Ind_{G_1\times G_2}^G (\chi)) = \Ind_{H_1\times H_2}^{H} F(\chi), \quad \chi\in \Irr_{p'}(G_1\times G_2|\phi),
\end{equation}
is a signed bijection. 
Since $H/U \simeq G/M$, we have $(G_1\times G_2) \cdot H = G$. 
By the Mackey formula, 
for each $\chi\in \Irr_{p'}(G_1\times G_2 |\phi)$, we have
\[
\begin{split}
 F_{\phi}(\Ind_{G_1\times G_2}^G (\chi)) - \Res^{G}_H (\Ind_{G_1\times G_2}^G \chi) &= 
\Ind_{H_1\times H_2}^H (F(\chi) - \Res^{G_1\times G_2}_{H_1\times H_2} \chi) \\
&\in
\Ind_{H_1\times H_2}^H (\Cal I(H_1\times H_2, P, \Cal S (G_1\times G_2, P, H_1\times H_2)).
\end{split}
\]
It is clear that $\cl S (G_1\times G_2, P, H_1\times H_2) \subset \cl S (G,P,H)$, whence
\[
 \Ind_{H_1\times H_2}^H (\Cal I(H_1\times H_2, P, \Cal S (G_1\times G_2, P, H_1\times H_2)) \subset \Cal I(H, P, \Cal S (G,P,H)).
\]
 Therefore, $F_{\phi}$ is IRC-compliant.  
\end{proof}

\begin{proof}[Proof of Theorem~\ref{thm:symirc}]
Let $\rho$ be the $p$-core corresponding to $b\in \Bl(S_m)$. Let $e=|\rho|$ and $w$ be the weight of $b$, so that $m=pw+e$. 
Let $c$ be the block of $S_e$ corresponding to $\rho$, so that $\Irr(S_e, c) = \{ \chi^{\rho} \}$.
Write $G=S_{pw+e}$, $L=S_p\wr S_w$ and $H=N_{S_{pw}}(P)$. 
Note that $P$ is a defect group of $b$ and $N_G (P) = H\times S_e$. Further, 
it follows from Lemma~\ref{lem:blwr} that the principal block is the only block of $L$ with defect group $P$ (for, by~\cite[Theorem 9.26]{NavarroBook}, any block 
of $L$ with defect group $P$ must cover a block of $S_p^{\times w}$ with defect group $C_p^{\times w}$). Let $c'\in \Bl(L\times S_e)$ be the tensor product of this block of $L$ with $c$.
A similar argument shows that the group $H$ has only one block. So the block $c$ of $H\times S_e$ is the Brauer correspondent of $b$. 
 
Arguing by induction, we may assume that the theorem holds for all blocks of all groups $S_{m'}$ with $m'<m$. 
By Corollary~\ref{cor:wrmarcus}, the property (IRC-Syl) holds for the pair $(L, N_{S_p}(C_p) \wr S_w)$. By the inductive hypothesis and Proposition~\ref{prop:cliff}, 
(IRC-Syl) holds for $(N_{S_p}(C_p) \wr S_w, N_{S_p}(C_p)\wr N_{S_w} (Q) )$. By Proposition~\ref{prop:inn}, (IRC-Syl) holds for $(N_{S_p}(C_p) \wr N_{S_w} (Q), H)$. 
By these facts and Lemma~\ref{lem:comp}, there exists an 
IRC-compliant signed bijection
$F_{LH} \colon \pm\Irr_{p'}(L) \to \pm\Irr_{p'}(H)$. 
We define a signed bijection $\tilde{F}_{LH}\colon \pm \Irr_0(L\times S_e, c) \to \pm \Irr_0 (H\times S_e, c)$ by
$\tilde{F}_{LH}(\g\times \chi^{\rho}) = F_{LH}(\g)\times \chi^{\rho}$ for all $\g\in \pm \Irr_{p'}(L)$. This bijection satisfies
\begin{equation}\label{eq:pfirc1}
 \tilde{F}_{LH} (\th) \equiv \Res^{L\times S_e}_{H\times S_e} \th \dmod \Cal I(H\times S_e,P,\Cal S(L\times S_e,P,H\times S_e)) \quad \forall \th\in \pm \Irr_0 (L\times S_e, c'),
\end{equation}
where $c'$ is the unique block of $L\times S_e$ with defect group $P$ such that $c'c=c'$.  

Further, by Theorem~\ref{thm:main}, Proposition~\ref{prop:heights} 
and Theorem~\ref{thm:Iset}, the map $F_{p,w,\rho}$ (see Eq.~\eqref{eq:defF}) 
restricts to a signed bijection from $\pm \Irr_0(G,b)$ onto $\pm \Irr_{p'}(L)$ satisfying
\[
 F_{p,w,\rho} (\chi) \equiv \Res^{S_{pw}}_{L} \tilde{\pi}_{\rho} \Res^{G}_{S_{pw}\times S_e} \chi \dmod \Cal I(L,P,\Cal S(G,P,L)).
\]
Define a signed bijection $F_{GL}\colon \pm \Irr_0 (G,b) \to \pm \Irr_0 (L\times S_e, c')$ by $F_{GL} (\chi) = F_{p,w,\rho}(\chi) \times \chi^{\rho}$. 
Due to the obvious identity
\[
 \Proj_c \Res^G_{L\times S_e} \xi = (\Res^{S_{pw}}_{L} \tilde{\pi}_{\rho} \Res^{G}_{S_{pw}\times S_e} \xi) \times  \chi^{\rho}  \quad \forall \xi\in \Cal C(G),
\]
we have 
\begin{equation}\label{eq:pfirc2}
 F_{GL} (\chi) \equiv \Proj_c \Res^G_{L\times S_e} \chi \dmod \Cal I(L, P, \Cal S(G,P,L \times S_e)) \otimes \lan \chi^{\rho} \ran \quad \forall \chi\in\pm\Irr_{0}(G,b),
\end{equation}
where $\lan \chi^{\rho}\ran$ denotes the $\mZ$-span of $\chi^{\rho}$. 

Let $F = \tilde{F}_{LH} F_{GL}$, so that $F\colon \pm \Irr_0 (G,b) \to \Irr_0 (H\times S_e, c)$ is a signed bijection. Write
$\Cal S = \Cal S(G,P,H\times S_e)$, and let $\chi\in \pm \Irr_0 (G,b)$.
We claim that 
\begin{equation}\label{eq:pfirc3}
F(\chi)-\Proj_c \Res^G_{H\times S_e} \chi \in \Cal I(H\times S_e, P, \Cal S).
\end{equation} 
By~\cite[Proposition 2.11(i)]{Evseev2010m}, this claim is equivalent to the same statement with $\Proj_c$ replaced by $\Proj_P$. So the theorem will follow once we establish~\eqref{eq:pfirc3}.

We have $F(\chi)- \Proj_c \Res^G_{H\times S_e} \chi = \xi_1 + \xi_2$ where 
\[
\begin{split}
\xi_1 &= \tilde{F}_{LH}(F_{GL}(\chi)) - \Res^{L\times S_e}_{H\times S_e} F_{GL}(\chi) \qquad \quad \text{and} \\
\xi_2 &= \Res^{L\times S_e}_{H\times S_e} (F_{GL}(\chi) - \Proj_c \Res^G_{L\times S_e} \chi). 
\end{split}
\]
(Here we use the obvious identity $\Res^{L\times S_e}_{H\times S_e} \Proj_c \Res^{G}_{L\times S_e} =  \Proj_c \Res^{G}_{H\times S_e}$.)
By~\eqref{eq:pfirc1}, $\xi_1\in \Cal I(H\times S_e,P,\cl S)$. By~\eqref{eq:pfirc2}, we have 
\[
\xi_2 \in \Res^{L}_{H} (\Cal I(L,P, \Cal S(G,P,L \times S_e)) )  \otimes \lan \chi^{\rho} \ran.
\]
Now $\Cal S(G,P,L\times S_e) \subset \Cal S$, so 
\[
 \Res^L_H (\Cal I(L,P,\Cal S(G,P,L\times S_e)) ) \subset \Res^L_H  (\Cal I(L,P, \Cal S)) \subset \Cal I(H,P,\Cal S),
\]
where the second containment follows from Lemma~\ref{lem:resI} (note that $\Cal S\supset \Cal S(L, P, H)$). 
Hence, $\xi_2 \in \Cal I(H,P,S) \otimes \lan \chi^{\rho} \ran$.
That is, $\xi_2$ belongs to the $\mZ$-span of virtual characters of the form $\Ind_A^H \a$ where $A\le H$ satisfies $A\cap P \in \Syl_p(H) \cap \cl S$ and $\a \in \cl C(A)$. 
Since the block $c$ of $S_e$ has defect $0$, we have $\chi^{\rho}\in \scr P(S_e)$. Hence, by Theorem~\ref{thm:brproj}, 
$\chi^{\rho}$ lies in the $\mZ$-span of virtual characters of the form $\Ind_B^{S_e} \b$ where 
$B\le S_e$ is a $p'$-subgroup and $\b \in \cl C(B)$. Thus, $\xi_2$ is an integer linear combination of 
virtual characters of the form $\Ind_{A\times B}^{H\times S_e} (\a \times \b)$ where $A,B,\a,\b$ are as above. But we have $(A\times B) \cap P = A\cap P \in\Syl_p(P) \cap \cl S$, so 
$\xi_2\in \Cal I(H\times S_e,P,\Cal S)$. 
This completes the proof of~\eqref{eq:pfirc3} and hence of Theorem~\ref{thm:symirc}. 
\end{proof}

\section{Uniqueness of the isometry}\label{sec:un}

Recall the subgroups $\cl K_s$ of $\cl C(S_p\wr S_w)$ defined by~\eqref{eq:defK}. 

\begin{prop}\label{prop:unique}
 Let $p\in \mN$ and $w,e \in \mZ_{\ge 0}$. Suppose that $\chi_1, \chi_2 \in \Irr(S_{pw}, \varnothing)$. 
\begin{enumerate}[(i)]
 \item\label{pun1} If $\chi_1\ne \chi_2$, then $\Res^{S_{pw}}_{S_p \wr S_w}(\chi_1-\chi_2) \notin \cl K_{w-1}$ and $\Res^{S_{pw}}_{S_p\wr S_w}(\chi_1+\chi_2)\notin \cl K_{w-1}$;
 \item\label{pun2} $\Res^{S_{pw}}_{S_p\wr S_w} \chi_1\notin \cl K_{w}$.
\end{enumerate}
\end{prop}

When $p$ is prime, the following corollary is equivalent to the uniqueness statement of Theorem~\ref{thm:val}, due to Lemma~\ref{lem:centp}.

\begin{cor}\label{cor:unique}
Let the notation be as in Theorem~\ref{thm:main}, and suppose that $e=0$. Then the map $F_{p,w,\varnothing}$ is the only signed bijection $F$ from $\pm \Irr(S_{pw}, \varnothing)$
onto $\pm \Irr(S_p \wr S_w)$ such that
\[
 F (\chi) \equiv \Res^{S_{pw}}_{S_p \wr S_w} \chi \dmod \Cal K_{w-1}.
\]
\end{cor}

\begin{proof}
 Let $F$ be a signed bijection satisfying the above congruence, and suppose that $F(\chi) = F_{p,w,\varnothing}(\th)$ for some $\chi, \th\in \pm\Irr(S_{pw}, \varnothing)$. 
Then $\Res^{S_{pw}}_{S_p\wr S_w} (\chi - \th) \in \cl K_{w-1}$, which implies, by Proposition~\ref{prop:unique}, that $\chi = \th$. 
So $F$ and $F_{p,w,\varnothing}$ have the same inverses, and
the result follows. 
\end{proof}

We now prove Proposition~\ref{prop:unique}.
In order to do so, we use the standard correspondence between characters of symmetric groups and symmetric functions, which is briefly described below:
the details may be found in~\cite[Chapter I]{Macdonald1995}.
Let $R = \bigoplus_{n\ge 0} \cl C(S_n)$. We endow $R$ with a product $*$ as follows: if $\chi\in \cl C(S_m)$ and $\th\in \cl C(S_n)$, then
$\chi * \th = \Ind_{S_m \times S_n}^{S_{m+n}} (\chi\times \th)$. Thus $R$ becomes a graded commutative ring. 

Let $\Sym$ be the ring of symmetric functions with integer 
coefficients in variables $x_1, x_2,\ldots$. 
There is a canonical ring isomorphism $\ch\colon R\to \Sym$ (see~\cite[\S I.8]{Macdonald1995}). 
If $\lda/\mu$ is any skew partition, then $\ch (\chi^{\lda/\mu})$ is the skew Schur function $s_{\lda/\mu}$. 

We also use the complete symmetric functions $h_n$ defined in~\cite[Section I.2]{Macdonald1995}. 
It is well known that the functions $h_n$, $n\in \mZ_{\ge 0}$, are algebraically independent over $\mZ$ and generate the ring $\Sym$. 
We will use the following result, which is a special case of~\cite[Theorem 1]{BarekatReinervanWilligenburg2009}\footnote{Lemma~\ref{lem:Schirr} is stated as Theorem 1 in~\cite{Farahat1958}, but the proof in~\cite{Farahat1958} appears to be incomplete (the formula for $\partial^{0}\{ \mu \}$ on p.\ 623 
implicitly assumes that $\lda_{k-1}>1$).}.

\begin{lem}\label{lem:Schirr}
 The Schur functions $s_{\lda}$, expressed as polynomials in $h_0, h_1, h_2,\ldots$, are irreducible. 
\end{lem}

 If $m\in \mN$ and $d\ge 0$, define the ``shrinking'' map $\shr_m\colon \CF (S_{md};K) \to \CF(S_d;K)$ by $\shr_m (\xi) (g_{\lda}) = \xi(g_{m\lda})$ for $\xi\in \CF(S_{md};K)$ and 
 $\lda\in\cl P(d)$, where $g_{\lda}$ is as defined before Lemma~\ref{lem:lincomb}. 
This map is described by the following known result. 

\begin{thm}\label{thm:shrfar}
 Suppose that $\lda/\mu$ is a skew partition and $|\lda/\mu| = md$. We have $\shr_m (\chi^{\lda/\mu}) =0$ unless $\lda \supset_m \mu$, in which case
\[
 \shr_m (\chi^{\lda/\mu}) = \e_m (\lda/\mu) 
 \Ind_{\prod_{i=0}^{m-1} S_{|\lda(i)/\mu(i)|}}^{S_d}
 (\chi^{\lda(0)/\mu(0)} \times \cdots \times \chi^{\lda(m-1)/\mu(m-1)} ).
\]
Hence, $\shr_m (\cl C(S_{md})) \subset \cl C(S_d)$. 
\end{thm}

\begin{proof}
 The statements concerning $\chi^{\lda/\mu}$ are due to Farahat~\cite[\S 4]{Farahat1954}. The last statement is an immediate consequence. 
\end{proof}

We fix $p\in \mN$ and $w\ge 0$. 
Let $f\colon \CF(S_{pw};K) \to \CF(S_{w-1}\times S_p; K)$ be the map defined by 
\[
 f(\xi) (g_{\a}, g_{\b})  = \xi(g_{p\a \sqcup \b}).
\]
We adopt the convention that $\chi^{\lda/\mu} =0$ and $s_{\lda/\mu} =0$ whenever $\lda\not\supset \mu$.

\begin{lem}\label{lem:f}
 If $\lda\in \cl P(pw)$ has empty $p$-core, then 
\[
 f(\chi^{\lda}) = \e_p (\lda) \sum_{j=0}^{p-1} (-1)^{p-j-1}  \Ind_{\prod_i {S_{|\lda(i)|-\d_{ij}}}}^{S_{w-1}} 
(\chi^{\lda(0)} \times \cdots \times \chi^{\lda(j)/(1)} \times \cdots \times \chi^{\lda(p-1)} ) \times  
 \chi^{(j+1,1^{p-j-1})}.
\]
\end{lem}

\begin{proof}
It follows from the above definition (together with~\cite[Eq.\ 2.4.16]{JamesKerber1981}) that 
\[
f(\chi^{\lda}) = \sum_{\mu \in \cl P(p)} \shr_p (\chi^{\lda/\mu}) \times \chi^{\mu}.
\]
By Theorem~\ref{thm:shrfar}, $\shr_p (\lda/\mu)=0$ unless $\lda\supset_p \mu$. By considering the abacus, we see that if $\mu\in \cl P(p)$ and $\lda\supset_p \mu$, then
$\mu$ is a hook partition, i.e.\ $\mu = (j+1, 1^{p-j-1})$ for some $j\in [0,p-1]$. Write $\mu^j = (j+1,1^{p-j-1})$ for each $j$. Hence,
\[
 f(\chi^{\lda}) = \sum_{j=0}^{p-1} (\shr_p(\chi^{\lda/\mu^j}) \times \chi^{\mu^j}).
\]
By Theorem~\ref{thm:shrfar},
\[
 \shr_p (\chi^{\lda/\mu^j}) = \e_p (\lda/\mu^j) \Ind_{\prod_i S_{|\lda(i)|-\d_{ij}}}^{S_{w-1}} (\chi^{\lda(0)} \times \cdots \times \chi^{\lda(j)/(1)} \times \cdots \times \chi^{\lda(p-1)}).
\]
For each $j$ we have $\e_p (\lda/\mu^j) = \e_p (\lda) \e_p (\mu^j) = (-1)^{p-j-1} \e_p (\lda)$ as long as $\lda \supset_p \mu^j$. 
The lemma follows from this equality together with the last two displayed equations.
\end{proof}

\begin{proof}[Proof of Proposition~\ref{prop:unique}]
Let $\lda$ and $\mu$ be partitions of $pw$ such that $\chi_1 = \chi^{\lda}$ and $\chi_2 = \chi^{\mu}$. Then $\lda$ and $\mu$ have empty $p$-cores.  By 
Theorem~\ref{thm:shrfar}, $\shr_p (\chi^{\lda})\ne 0$, which means precisely that $\Res^{S_{pw}}_{S_p\wr S_w} \chi_1 \notin \cl K_w$, so~\eqref{pun2} holds. 

To prove~\eqref{pun1}, suppose for contradiction that $\Res^{S_{pw}}_{S_p\wr S_w} (\chi_1 - \e \chi_2) \in \cl K_{w-1}$ for some $\e\in \{ \pm 1\}$. 
Due to the definition of $\cl K_{w-1}$ and Lemma~\ref{lem:cyctype}, this means precisely that $f(\chi_1) = \e f(\chi_2)$.
 Considering the 
coefficient in $\chi^{(j+1, 1^{p-j-1})}$ in Lemma~\ref{lem:f}, we see that there is a fixed $\e'\in \{ \pm 1\}$ such that for all $j\in [0,p-1]$
\[
\begin{split}
 & \quad \Ind_{\prod_i {S_{|\lda(i)|-\d_{ij}}}}^{S_{w-1}} (\chi^{\lda(0)}  \times \cdots \times \chi^{\lda(j)/(1)} \times \cdots \times \chi^{\lda(p-1)} ) \quad \\
 &= \quad \e' \Ind_{\prod_i {S_{|\mu(i)|-\d_{ij}}}}^{S_{w-1}} (\chi^{\mu(0)}  \times \cdots \times \chi^{\mu(j)/(1)} \times \cdots \times \chi^{\mu(p-1)} ). \quad 
\end{split}
\]
Further, the fact that $\Res^{S_{pw}}_{S_p\wr S_w} (\chi_1 - \e \chi_2) \in \cl K_w$ means precisely that $\shr_p(\chi_1) = \e \shr_p (\chi_2)$, so by Theorem~\ref{thm:shrfar} we have
\[
 \Ind_{\prod_i S_{|\lda(i)|}}^{S_w} \prod_{i=0}^{p-1} \chi^{\lda(i)} = \e \Ind_{\prod_i S_{|\mu(i)|}}^{S_w} \prod_{i=0}^{p-1} \chi^{\mu(i)}
\]
Applying the map $\ch$ to the above two displayed equalities, we obtain
\begin{equation}\label{eq:prun1}
 s_{\lda(0)} s_{\lda(1)} \cdots s_{\lda(j)/(1)} \cdots s_{\lda(p-1)} = \e' s_{\mu(0)} s_{\mu(1)} \cdots s_{\mu(j)/(1)} \cdots s_{\mu(p-1)} \quad \text{for all } j\in [0,p-1]
\end{equation}
and
\begin{equation}\label{eq:prun2}
 s_{\lda(0)} s_{\lda(1)} \cdots s_{\lda(p-1)} = \e s_{\mu(0)} s_{\mu(1)} \cdots s_{\mu(p-1)}.
\end{equation}
By Eq.~\eqref{eq:prun2} and  Lemma~\ref{lem:Schirr}, the tuple $(\lda(0),\ldots, \lda(p-1))$ is a permutation of the tuple $(\mu(0),\ldots,\mu(p-1))$. 

Consider the linear order $>$ on the set $\Cal P$ defined as follows: $\k > \nu$ if and only if either $|\k|>|\nu|$ or $|\k|=|\nu|$ and $\k$ is greater than $\nu$ in the 
lexicographic order (as defined e.g.~by~\cite[Eq.~1.4.5]{JamesKerber1981}). Let $(i_1,\ldots, i_p)$ be a permutation of $[0,p-1]$ such that $\lda(i_1)\ge \cdots \ge \lda(i_p)$. 

Since $\lda\ne \mu$, we have $\lda(l) \ne \mu(l)$ for some $l$. Let $r\ge 0$ be the smallest index such that $\lda(i_r) \ne \mu(i_r)$. If $\mu(i_r)>\lda(i_r)$, then 
$\mu(i_r)$ appears more times in the tuple $(\mu(0),\ldots,\mu(p-1))$ than in the tuple $(\lda(0),\ldots,\lda(p-1))$ (because $\lda(i_t)=\mu(i_t)$ for $t<r$), which cannot be the 
case as these two tuples are permutations of each other. So $\mu(i_r)<\lda(i_r)$; in particular, $\lda(i_r) \ne \varnothing$. 
Let $m$ be the multiplicity of the irreducible polynomial $s_{\lda(i_r)}$, written in variables $h_0, h_1,\ldots$, in $s_{\lda(i_1)}  \cdots s_{\lda(i_p)}$
(i.e.\ $m$ is the largest nonnegative integer such that $s_{\lda(i_r)}^m$ divides that product). 
Since the degree of $s_{\lda(i_r)/(1)}$ in variables $x_1,x_2,\ldots$ is one less than the degree of $s_{\lda(i_r)}$,
the former polynomial is not divisible by the latter one (even when written in terms of $h_0,h_1,\ldots$). 
Hence, the multiplicity of $s_{\lda(i_r)}$ in 
$s_{\lda(i_1)} \cdots s_{\lda(i_r)/(1)} \cdots s_{\lda(i_p)}$ is $m-1$.
On the other hand, by~\eqref{eq:prun2}, the multiplicity of $s_{\lda(i_r)}$ in $s_{\mu(i_1)} \cdots s_{\mu(i_p)}$ is $m$. Since $s_{\mu(i_r)}$ is not divisible by $s_{\lda(i_r)}$, 
the multiplicity of $s_{\lda(i_r)}$ in $s_{\mu(i_1)} \cdots s_{\mu(i_r)/(1)} \cdots s_{\mu(i_p)}$ must be at least $m$. We have reached a contradiction to~\eqref{eq:prun1} (for $j=r$).
\end{proof}

Thus, we have completed a proof of Theorem~\ref{thm:val}.

\section{Properties of the isometry}\label{sec:prop}

Let $p$ be a fixed prime.
Let $G$ and $H$ be finite groups and $\mu\in \cl C(G\times H)$. The virtual character $\mu$ is called \emph{perfect} (\cite[D{\'e}finition 1.1]{Broue1990}) if
\begin{enumerate}[(i)]
 \item\label{perf1} $\mu(g,h)/|C_G(g)|\in \cl O$ and $\mu(h)/|C_H (h)|\in \cl O$ for all $g\in G$ and $h\in H$;
 \item\label{perf2} $\mu(g,h)=0$ if either $g\in G_{p'}$ and $h\in H-H_{p'}$ or $g\in G-G_{p'}$ and $h\in H_{p'}$.
\end{enumerate}
As in~\cite{Broue1990} (but using different conventions), 
define the maps $I_{\mu} \colon \CF(H; K) \to \CF(G;K)$ and $R_{\mu}\colon \CF(G;K)\to \CF(H;K)$ 
by 
\[
\begin{split} 
I_{\mu} (\th) (g) &= \frac{1}{|H|} \sum_{h\in H} \mu(g,h) \th(h) \quad \text{and} \\
R_{\mu} (\chi) (h) &= \frac{1}{|G|} \sum_{g\in G} \mu(g,h) \chi(g).
\end{split}
\]

Let $G= S_{pw+e}$ and $H=S_p\wr S_w$, and let $\rho$ be a $p$-core of size $e$. The signed bijection $F=F_{p,w,\rho}$ of Theorem~\ref{thm:main} 
corresponds to a certain $\mu\in \cl C(G\times H)$ (see Eq. \eqref{eq:mu} below), and this virtual character $\mu$ is perfect if $w<p$ by~\cite[Corollaire 2.12]{Rouquier1994}.
If $w\ge p$, this is not the case.
However, in this section we prove Propositions~\ref{prop:sep} and~\ref{prop:perfproj}, which show that analogues of some properties of perfect isometries hold for all $w$. 

More precisely, Proposition~\ref{prop:sep} proves the analogue of condition~\eqref{perf2} above obtained by replacing $H_{p'}$ with a certain larger set $H^{(p')}$.
This subset $H^{(p')}$ of $H$ is defined in \S\ref{sub:71}, 
where we describe a generalisation of modular character theory for wreath products such as $H$, considering values of characters on $H^{(p')}$ rather than on $H_{p'}$.
Proposition~\ref{prop:perfproj} shows that $\mu$ is well-behaved with respect to this generalisation, 
in the same way as perfect isometries are well-behaved in relation to the 
usual modular characters: cf.\ \cite[Proposition 1.3]{Broue1990}. We do not include a generalisation of condition~\eqref{perf1}, which is harder to state.

\begin{remark}
A generalised version of a perfect isometry defined by K{\"u}lshammer, Olsson, and Robinson~\cite{KuelshammerOlssonRobinson2003} is well suited to the present situation. 
In fact, it is possible to deduce from Proposition~\ref{prop:type} 
and~\cite[Lemma 3.3]{KuelshammerOlssonRobinson2003} 
that $F$ yields a perfect isometry between
$\Irr(G,\rho)$ and $\Irr_{\pri}(H)$ with respect to the subsets $G_{p'}$ and $H^{(p')}$ in the sense of~\cite[Section 1]{KuelshammerOlssonRobinson2003}.
Moreover, $F$ is related to the isometry between $\Irr(G,\rho)$ and $\Irr(C_p\wr S_w)$ constructed in \emph{op.\ cit.} (see~\cite[Proposition 5.11]{KuelshammerOlssonRobinson2003}).
Gramain has obtained a similar result in the case when $w<p$ but $p$ is not necessarily a prime and one considers $N_{S_p}(C_p)\wr S_w$ instead of $S_p\wr S_w$:
see~\cite[Theorem 4.1]{Gramain2008}.
\end{remark}

In the case of abelian defect groups, perfect isometries between blocks are often viewed as character-theoretic shadows of derived equivalences 
at the level of module categories (see~\cite{Broue1990} and, for symmetric groups,~\cite{ChuangKessar2002, ChuangRouquier2008}). The question of whether there
is a deeper phenomenon of some sort underpinning the isometry $F$ when $w\ge p$ seems to be very much open.

\subsection{Generalised Cartan--Brauer triangle for wreath products}\label{sub:71}

Let $w\in \mZ_{\ge 0}$ and $L$ be a finite group. Consider the group $H=L\wr S_w$. 
We construct an analogue of the modular character theory of $H$, which for $w<p$ coincides with the standard theory of Brauer characters. 
In particular, 
Proposition~\ref{prop:orth} and 
Corollaries~\ref{cor:Brbasis} and~\ref{cor:decnos} below generalise classical results on modular characters (cf.\ \cite[\S 18]{CRI}).

Define $H^{(p')}$ to be the set of all elements $h\in H$ such that $h$ is $H$-conjugate to an element of the form
$y_{\s_1} (x_1) \cdots y_{\s_r} (x_r)$ where $\s_1,\ldots,\s_r$ are disjoint marked cycles in $S_w$ and $x_1,\ldots,x_r\in L_{p'}$.
Let $\IBr(L)$ be the set of irreducible Brauer characters of $L$, so that each element of $\IBr(L)$ is a class function defined on $L_{p'}$ 
(i.e.\ a map from $L_{p'}$ to $K$ that is constant on $L$-conjugacy classes). 
For $\Psi\in \PMap_w (\IBr(L))$ we define an $\cl O$-valued class function $\zeta_{\Psi}$ on $H^{(p')}$ by generalising~\eqref{eq:defzeta} in the obvious way. (Formally, 
let $\Psi' = ((\phi', \chi^{\Psi(\phi)} ) \mid \phi \in \IBr(L) )$ where $\phi'$ is an arbitrary extension of $\phi$ to a class function on the whole of $L$. Then $\zeta_{\Psi}$ is 
the restriction of $\zeta_{\Psi'}$ to $H^{(p')}$.)
We define the set $\GIBr(H)$ of ``generalised irreducible Brauer characters'' of $H$ by
\[
 \GIBr(H) = \{ \zeta_{\Psi} \mid \Psi\in \PMap_w (\IBr(L)) \}.
\]
If $w<p$, then $\GIBr(H) = \IBr(H)$ by~\cite[Theorem 4.3.34]{JamesKerber1981}. 
Where the context requires it, we view $\zeta_{\Psi}$ as a class function on 
all of $H$, setting it to be $0$ on $H-H^{(p')}$.

For each $\phi\in \IBr(L)$ let $\wh{\phi}$ be the character afforded by the projective indecomposable $\cl OL$-module corresponding to $\phi$ (see e.g.\ \cite[Chapter 2]{NavarroBook}).
For $\Psi\in \PMap_w(\IBr(L))$, define
\[
 \wh{\Psi} = ( (\wh{\phi}, \Psi(\phi)) \mid \phi\in \IBr(L) ).
\]
Since the characters $\wh{\phi}$ vanish outside $L_{p'}$, we have $\zeta_{\wh{\Psi}} (h) =0$ for all $h\in H-H^{(p')}$ (due to~\eqref{eq:wrcharval}).

\begin{lem}\label{lem:orth}
 Let $\phi,\psi\in \IBr(L)$ and $\lda,\mu\in \cl P(w)$. Then $\lan \zeta_{(\phi,\chi^{\lda})}, \zeta_{(\wh \psi,\chi^{\mu})} \ran = \d_{\phi\psi}\d_{\lda\mu}$.
\end{lem}

\begin{proof}
 Consider $\s\in S_w$ and let $\s=\s_1\cdots \s_r$ be its decomposition into disjoint marked cycles with orders summing to $w$. Let $X_{\s}$ be the preimage of $\s$ in $L\wr S_w$.
Consider the equivalence relation on $X_{\s}$ defined by the rule that $(z_1,\ldots,z_w; \s) \sim y_{\s_1} (x_1)\cdots y_{\s_r} (x_r)$ (where $z_i,x_j\in L$) if and only if
$x_j = z_{t} z_{\s_j^{-1}(t)} \cdots z_{\s_j^{-(o(\s_j)-1)} (t)}$, where $t$ is the smallest element of $\Supp(\s_j)$, for all $j$. 
Then the elements of each equivalence class are $H$-conjugate to each other (by~\cite[Theorem 4.2.8]{JamesKerber1981}), and each equivalence class contains $|L|^{w-r}$ elements. 
Hence, by~\eqref{eq:wrcharval},
\[
\begin{split}
& \; |H|^{-1} \sum_{h\in X_{\s}} \zeta_{(\phi, \chi^{\lda})}(h) \zeta_{(\wh\psi, \chi^{\mu})} (h^{-1}) = \\
&= (w! |L|^r)^{-1} \chi^{\lda} (\s) \chi^{\mu} (\s) \sum_{x_1,\ldots,x_r\in L} \phi(x_1) \wh{\psi}(x_1^{-1}) \cdots \phi(x_r) \wh{\psi}(x_r^{-1}) = (w!)^{-1} \chi^{\lda} (\s) \chi^{\mu}(\s) \d_{\phi\psi},
\end{split}
\]
where the last equality follows from the fact that $\lan \phi, \wh \psi \ran = \d_{\phi\psi}$ (see e.g.\ \cite[Theorem 2.13]{NavarroBook}).
The lemma follows after summing over all $\s\in S_w$ and using the standard orthogonality relation. 
\end{proof}

\begin{prop}\label{prop:orth}
 If $\Phi, \Psi\in \GIBr(H)$, then $\lan \zeta_{\wh \Psi}, \zeta_{\Phi} \ran =\d_{\Psi\Phi}$.
\end{prop}

\begin{proof}
 Throughout the proof, $\phi$ and $\psi$ are assumed to run over $\IBr(L)$. 
Let $A=\prod_{\phi} A_{\phi}$ and $B=\prod_{\phi} B_{\phi}$ be fixed Young subgroups of $S_w$ such that
$A_{\phi} \simeq S_{|\Psi(\phi)|}$ and $B_{\phi} \simeq S_{|\Phi(\phi)|}$ for all $\phi$. Let $\s\in S_w\le H$.  
Then $\lsb{\s}{A} \cap B = \prod_{\phi,\psi} (\lsb{\s}A_{\phi} \cap B_{\psi})$. It follows from Lemma~\ref{lem:orth} that 
\[
 \lan \Res^{H}_{ \,\ls{\s}{(L\wr A)}} \ls{\s}{\zeta_{\Phi}}, \Res^H_{L\wr B} \zeta_{\wh \Psi}\ran =0 
\]
unless $\lsb{\s}{A_{\phi}} \cap B_{\psi} = \mbf 1$ for all $\phi,\psi$ such that $\phi\ne \psi$, in which case we have $\lsb{\s}{A_{\phi}} = B_{\phi}$ for all $\phi$. 
Assuming the last statement holds 
and applying Lemma~\ref{lem:orth} again, we see that the above inner product is $0$ unless $\Phi(\phi) = \Psi(\phi)$ for all $\phi$, in which case it is equal to $1$.
The proposition now follows by the Mackey formula. 
\end{proof}

If $w<p$, this proposition and the fact that $\GIBr(H) =\IBr(H)$ imply that $\zeta_{\wh \Psi}$ is the character afforded by the projective indecomposable $\cl OH$-module corresponding
to $\zeta_{\Psi}$ for each $\Psi\in \PMap_w(\IBr(L))$. 

\begin{cor}\label{cor:Brbasis}
 The set $\GIBr(H)$ is a basis of the space of $K$-valued class functions on $H^{(p')}$.
\end{cor}

\begin{proof}
 It follows immediately from Proposition~\ref{prop:orth} that the set $\GIBr(H)$ is linearly independent over $K$. Also, it is clear that the number of $H$-conjugacy classes contained in $H^{(p')}$ is 
equal to $\PMap_w (X)$, where $X$ is a set of size $|\IBr(L)|$, and therefore is equal to $|\GIBr(H)|$. The result follows. 
\end{proof}

For any class function $\xi$ on $H$, let $\xi^{(p')}$ be the restriction of $\xi$ to the set $H^{(p')}$. By Corollary~\ref{cor:Brbasis}, for each $\th\in \Irr(H)$, we have
\[
 \th^{(p')} = \sum_{\Psi\in\PMap_w(\IBr(L))} d_{\th\Psi} \zeta_{\Psi}
\]
for uniquely determined \emph{decomposition numbers} $d_{\th\Psi}\in K$.

\begin{cor}\label{cor:decnos}
 For every $\th\in \Irr(H)$ and every $\zeta_{\Psi}\in \GIBr(H)$, we have 
 $d_{\th\Psi} = \lan \th, \zeta_{\wh{\Psi}} \ran$. In particular, $d_{\th\Psi}\in \mZ$. 
\end{cor}

\begin{proof}
 We have
\[
 \lan \th, \zeta_{\wh{\Psi}} \ran = \sum_{\zeta_{\Phi}\in\GIBr(H)} d_{\th\Phi} \lan \zeta_{\Phi}, \zeta_{\wh \Psi} \ran = d_{\th\Psi}
\]
by Proposition~\ref{prop:orth}.
\end{proof}

\subsection{Generalised properties of perfect isometries}\label{sub:72}

Let the notation be 
as in the beginning of Section~\ref{sec:prop}; in particular,
$G=S_{pw+e}$ and $H=S_p\wr S_w$. Denote by $f_{\rho}$ the block of $G$ corresponding to $\rho$ and by  
$b_0$ the principal block of $H$.
The signed bijection $F=F_{p,w,\rho}$ extends to an isometry $F\colon \cl C(G,\rho)\to \cl C(H,b_0)$.
Then $F$ corresponds to the virtual character $\mu\in\cl C(G\times H)$ defined by
\begin{equation}\label{eq:mu}
 \mu = \sum_{\chi\in \Irr(G,\rho)} \chi \times \ol{F(\chi)}.
\end{equation}
That is, $I_{\mu}|_{\cl C(H,b_0)} =F^{-1}$ and $R_{\bar{\mu}}|_{\cl C(G,\rho)}  = F$.
Also, $R_{\bar{\mu}} (\xi)=0$ and $I_{\mu}(\th)=0$ for all $\xi\in \CF(G, 1-f_{\rho}; K)$ and $\th\in \CF(H, 1-b_0; K)$.
(We note that, in fact, $\bar{\mu} =\mu$ since all the values of $\mu$ are integers.)

If $\lda\in \cl P$ and $|\lda|\le w$, define
\[
\begin{split}
 \cl L_{\lda} (G,\rho) &= \{ \xi \in \CF(G, f_{\rho}; K) \mid \xi(g)=0 \text{ for all } g\in G \text{ such that } \tp_p (g) \ne \lda \} \qquad \text{and} \\
 \cl L_{\lda} (H, b_0) &= \{ \xi \in \CF(H, b_0; K) \mid \xi(h)=0 \text{ for all } h\in H \text{ such that } \tp_p^{\wre} (h) \ne \lda \}.
\end{split}
\]
In what follows, we will use the notation of the proof of Theorem~\ref{thm:main}, in particular, the maps $d^{\lda} = d^{\lda}_n$ and $\d^{\lda}$. 

\begin{prop}\label{prop:type}
Let $\lda$ be a partition with $|\lda|\le w$.
 Then $F (\cl L_{\lda} (G,\rho)) = \cl L_{\lda} (H, b_0)$. 
\end{prop}

\begin{proof}
Suppose that $\xi\in \cl L_{\lda} (G,\rho)$. Let $\a$ be a partition such that $|\a|\le w$ and $\a\ne \lda$. 
 Then $d^{\a} (\xi) =0$ and, considering the right-hand square of the commutative diagram~\eqref{eq:main3}, we see that $\d^{\a} (F (\xi))=0$.  
Hence, $F (\xi) \in \cl L_{\lda} (H, b_0)$. So $F (\cl L_{\lda} (G,\rho)) \subset \cl L_{\lda} (H, b_0)$.
Since $F$ is an isomorphism, we have equality for all $\lda$.
\end{proof}

If $g\in G$, let $g^{(p)}$ be the product of all the cycles in the cycle decomposition of $g$ that have orders divisible by $p$ and $g^{(p')}$ be the product of the other cycles, 
so that $g= g^{(p)} g^{(p')}$.
If $h\in H$, we may choose $z\in H$ such that $\ls{z}h = y_{\s_1} (x_1) \cdots y_{\s_r} (x_r)$ for some disjoint marked cycles $\s_1,\ldots,\s_r$ in $S_w$ and $x_1,\ldots,x_r\in S_p$.
Then we define $h^{(p')} = \ls{z^{-1}}{(\prod_j y_{\s_j} (x_j))}$ where $j$ runs over all the elements of $[1,r]$ such that $x_r$ is not a $p$-cycle. 
In particular, $h\in H^{(p')}$ if and only if $h=h^{(p')}$.


\begin{lem}\label{lem:vanish}
Let $\lda\in \cl P$.
 If $\xi\in \CF(G; K)$ vanishes on all $g\in G$ such that $\tp_p (g)=\lda$, then so does $\Proj_{f_{\rho}} \xi$. 
\end{lem}

\begin{proof}
For each $\th\in \CF(G; K)$ let $\th^{(\lda)}\in \CF(G;K)$ be defined by
\[
 \th^{(\lda)} (g) =
\begin{cases}
 \th(g) & \text{if } \tp_p (g) = \lda, \\
 0 & \text{otherwise.}
\end{cases}
\]
By~\cite[Lemma 3.3]{KuelshammerOlssonRobinson2003}, if $\th\in \Irr(G,\rho)$, then $\th^{(\lda)}$ is orthogonal to 
$\Irr(G)-\Irr(G,f_{\rho})$, and hence $\th^{(\lda)} \in \CF(G,f_{\rho};K)$.
Therefore, $(\Proj_{f_{\rho}} \xi)^{(\lda)} \in \CF(G,f_{\rho}; K)$. Similarly, $(\Proj_{1-f_{\rho}} \xi)^{(\lda)} \in \CF(G,1-f_{\rho};K)$.
Further,
\[
 0= \xi^{(\lda)} = (\Proj_{f_{\rho}} \xi)^{(\lda)} +
 (\Proj_{1-f_{\rho}} (\xi))^{(\lda)}.
\]
Since $\CF(G, f_{\rho}; K)\cap \CF(G,1-f_{\rho}; K)=0$, this implies that 
$(\Proj_{f_{\rho}} \xi)^{(\lda)} =0$, as claimed.
\end{proof}

\begin{prop}\label{prop:sep}
 Let $g\in G$ and $h\in H$. If $\tp_p (g)\ne \tp_p^{\wre} (h)$, then $\mu(g,h)=0$. 
\end{prop}

\begin{proof}
Let $\xi\in \CF(G;K)$ be the class function defined by $\xi(g')=1$ if $g'$ is $G$-conjugate to $g$ and $\xi(g')=0$ otherwise. Let $\lda=\tp_p (g)$.
By Lemma~\ref{lem:vanish}, $\Proj_{f_{\rho}}(\xi)(x)=0$ for all $x\in G$ with $\tp_p(x)\ne \lda$. This means that 
$\Proj_{f_{\rho}}\xi \in \cl L_{\lda} (G,\rho)$, whence by Proposition~\ref{prop:type} we have $R_{\mu}(\Proj_{f_{\rho}}\xi) \in \cl L_{\lda} (H, b_0)$. 
Since $R_{\mu}(\xi)=R_{\mu}(\Proj_{f_{\rho}} \xi)$, we deduce that 
$R_{\mu}(\xi)(h)=0$.
However, by the definition of $\xi$, we have $R_{\mu}(\xi)(h) = |G|^{-1} \mu(g,h)$. Hence, $\mu(g,h)=0$. 
\end{proof}

Let $\cl G$ be the $\mZ$-span of characters of $G$ afforded by the projective indecomposable $\cl OG$-modules belonging to the block $f_{\rho}$. 
Let $\cl H$ be the $\mZ$-span of the characters $\zeta_{\wh\Psi}$ where $\Psi$ runs over the elements of $\PMap_w (\IBr(S_p))$ such that $\Psi(\phi) = \varnothing$ whenever
$\phi\in \IBr(S_p)$ does not belong to the principal block of $S_p$. (If $w<p$, then $\cl H$ is simply the span of characters afforded by the projective indecomposable modules belonging
to the principal block of $H$.)

\begin{prop}\label{prop:perfproj}
We have $R_{\mu} (\cl G) = \cl H$.
\end{prop}

\begin{proof}
 By a well-known general result (cf.\ \cite[Corollary 2.17]{NavarroBook}), $\cl G = \scr P(G) \cap \cl C(G,\rho) = \cl L_{\varnothing}(G,\rho) \cap \cl C(G)$.
Let $\cl H' = \cl L_{\varnothing}(H,b_0)\cap \cl C(H)$. Since $R_{\mu}$ maps $\cl L_{\varnothing}(G,\rho)$ onto $\cl L_{\varnothing}(H,b_0)$
 (by Proposition~\ref{prop:type}) and $\cl C(G, f_{\rho})$ onto $\cl C(H,b_0)$, we have $R_{\mu}(\cl G)= \cl H'$.

Therefore, it remains to prove that $\cl H = \cl H'$. Let $\xi\in \cl H$; without loss of generality, $\xi =\zeta_{\wh\Psi}$ for an appropriate $\Psi$. 
As we observed just before Lemma~\ref{lem:orth}, $\zeta_{\wh \Psi}$ vanishes outside $H^{(p')}$. Further, it is easy to see (using Lemma~\ref{lem:wrindconv}) 
that the above condition imposed on $\Psi$ means that $\zeta_{\wh\Psi}\in \CF(H, b_0; K)$. Hence, $\zeta_{\wh\Psi}\in \cl H'$.

Conversely, suppose that $\xi\in \cl H'$. Let $c_0$ be the principal block of $S_p$. 
Then for each $\lda\in \cl P(w)$ we have $\om_{\lda} (\xi) \in (\scr P(S_p) \cap \cl C(S_p, c_0))^{\otimes l(\lda)}$. 
By Theorem~\ref{thm:wrind}, it follows that $\xi$ is a $\mZ$-linear combination of virtual characters of the form $\zeta_{\Theta}$ where $\Theta$ runs through tuples  
$((\a_1,\chi_1), \ldots (\a_r,\chi_r))\in \Tup_w (L)$ such that $\a_i\in \cl C(S_p, c_0)\cap \scr P(S_p)$ for each $i$. Each $\a_i$ therefore belongs to the $\mZ$-span of the characters $\wh{\phi}$ with $\phi\in \IBr(S_p, c_0)$. Using Lemma~\ref{lem:lincomb}, we deduce that every such $\zeta_{\Theta}$ lies in $\cl H$.  Therefore, $\xi \in \cl H$. 
\end{proof}



\providecommand{\href}[2]{#2}

\medskip
School of Mathematics,
University of Birmingham,
Edgbaston,
Birmingham B15 2TT, UK 

\texttt{a.evseev@bham.ac.uk}

\end{document}